\documentclass[11pt]{article}

\usepackage{url}
\usepackage{amssymb,enumerate}
\usepackage{amsmath,amsfonts}
\usepackage{amsthm}
\usepackage{latexsym}
\usepackage{mathrsfs}
\usepackage{color}
\usepackage{microtype}
\usepackage{geometry}

\usepackage{algorithm}
\usepackage{algorithmicx}
\usepackage{algpseudocode}
\usepackage{amsmath,amsfonts,bm}

\newcommand{\rmem}{\mathrm{em}}

\newcommand{\rmint}{\mathrm{int}}

\def\eqref#1{equation~\ref{#1}}

\def\1{\bm{1}}

\DeclareMathAlphabet{\mathsfit}{\encodingdefault}{\sfdefault}{m}{sl}
\SetMathAlphabet{\mathsfit}{bold}{\encodingdefault}{\sfdefault}{bx}{n}

\def\gK{{\mathcal{K}}}

\def\gO{{\mathcal{O}}}

\newcommand{\E}{\mathbb{E}}

\newcommand{\R}{\mathbb{R}}

\newcommand{\rmrm}{\mathrm{rm}}
\newcommand{\rmpm}{\mathrm{pm}}

\usepackage{hyperref}

\usepackage{multirow}
\allowdisplaybreaks[4]
\usepackage[capitalise,nameinlink]{cleveref}
\crefformat{equation}{(#2#1#3)}
\usepackage{xcolor}
\hypersetup{backref=false,       
	pagebackref=true,               
	hyperindex=true,                
	colorlinks=true,                
	breaklinks=true,                
	urlcolor= black,                
	linkcolor= black,                
	bookmarks=true,                 
	bookmarksopen=false,
	filecolor=black,
	citecolor=black,
	linkbordercolor=blue
}

\theoremstyle{plain}
\newtheorem{theorem}{Theorem}

\newtheorem{lemma}{Lemma}

\newtheorem{assumption}{Assumption}

\theoremstyle{definition}
\newtheorem{definition}{Definition}

\newtheorem{remark}{Remark}

\usepackage{natbib}
\usepackage{graphicx, color}
\usepackage{algorithm, algpseudocode} 
\usepackage{mathrsfs} 

\usepackage{lipsum}

\title{Stochastic interior-point methods for smooth conic optimization with applications}
\author{Chuan He\thanks{Department of Mathematics, Link\"oping University, Sweden (email: \url{chuan.he@liu.se}). The work of this author was partially supported by the Wallenberg AI, Autonomous Systems and Software Program (WASP) funded by the Knut and Alice Wallenberg Foundation.}
\and
Zhanwang Deng\thanks{Academy for Advanced Interdisciplinary Studies, Peking University, China (email: \url{dzw_opt2022@stu.pku.edu.cn}).}
}

\date{
December 18, 2024 (Revised: May 23, 2025; August 20, 2025)
}

\begin{document}
\maketitle

\begin{abstract}
Conic optimization plays a crucial role in many machine learning (ML) problems. However, practical algorithms for conic constrained ML problems with large datasets are often limited to specific use cases, as stochastic algorithms for general conic optimization remain underdeveloped. To fill this gap, we introduce a stochastic interior-point method (SIPM) framework for general conic optimization, along with four novel SIPM variants leveraging distinct stochastic gradient estimators. Under mild assumptions, we establish the iteration complexity of our proposed SIPMs, which, up to a polylogarithmic factor, match the best-known {results} in stochastic unconstrained optimization. Finally, our numerical experiments on robust linear regression, multi-task relationship learning, and clustering data streams demonstrate the effectiveness and efficiency of our approach.
\end{abstract}


\medskip


\section{Introduction}
Conic optimization covers a broad class of optimization problems with constraints represented by convex cones, including common forms such as linear constraints, second-order cone constraints, and semidefinite constraints. Over the years, this class of optimization problems has found applications across various fields, including control \cite{fares2001augmented}, energy systems \cite{zohrizadeh2020conic}, combinatorial optimization \cite{wolkowicz2012handbook}, and machine learning (ML) \cite{sra2011optimization}. In practice, the efficient handling of traditional conic optimization problems has been extensively studied for decades, with interior-point methods (IPMs) standing out due to their ability to effectively and elegantly solve a wide range of conic constrained problems within a unified framework (see the monograph by \cite{NN1994IPM}).


\subsection{Stochastic optimization with conic constraints for machine learning}\label{ss:eg}

Existing studies on IPMs primarily focus on the deterministic regime, despite the recent widespread applications of stochastic conic optimization in ML. These applications include multi-task relationship learning \cite{argyriou2008convex,zhang2012convex}, robust learning with chance constraints \cite{shivaswamy2006second,xu2012optimization}, kernel learning \cite{bach2004multiple}, and clustering data streams \cite{bidaurrazaga2021k,peng2007approximating,sun2021convex}. Next, we briefly highlight two examples.

\paragraph{Multi-task relationship learning}
Multi-task learning is an ML paradigm where related tasks are learned together to improve generalization by sharing information (see \cite{zhang2021survey}). In many applications, task correlations are not explicitly available. To learn these correlations from data, a regularization framework is proposed, where the relationships between models for different tasks are controlled by a regularizer defined using the covariance matrix {$\Sigma$} \cite{argyriou2008convex,zhang2012convex}:
\begin{align}\label{multi-task}
\min_{W\in\R^{p\times d},\Sigma\in\R^{p\times p}}\frac{1}{p}\sum_{i=1}^p \frac{1}{m} \sum_{j=1}^{m}\ell(w_i,a_{ij}) + \lambda\mathrm{tr}(W^TP(\Sigma)W) \quad \mathrm{s.t.}\quad \Sigma\in\mathbb{S}_+^p,\quad \mathrm{tr}(\Sigma)=1,  
\end{align}
where $W = [w_1, \ldots, w_p]^T$ {denotes the model coefficients}, $\mathbb{S}^p_+$ denotes the positive semidefinite cone, $\ell(\cdot, \cdot)$ is the loss function, $w_i$ and $\{a_{ij}\}_{j=1}^{m}$ are respectively the model weight and the training set for the $i$th task, $1 \le i \le p$, $\lambda>0$ is a tuning parameter, $P : \mathbb{R}^{p \times p} \to \mathbb{R}^{p \times p}$ is a given map that controls the interaction between $W$ and {$\Sigma$}, and $\mathrm{tr}(\cdot)$ denotes the trace of a matrix. {The constraint $\mathrm{tr}(\Sigma)=1$ in \cref{multi-task} is imposed to control the complexity of $\Sigma$, as discussed in \cite{zhang2012convex}.}

\paragraph{Robust learning with chance constraints} Consider supervised learning with feature-label pairs $\{(a_i,b_i)\}_{i=1}^p$, where the $a_i$'s are assumed to be generated from a distribution $\mathcal{D}$ with expected value $\bar{a}$ and covariance matrix $\Sigma$. To mitigate the uncertainty in the $a_i$'s, a chance constraint $\mathbb{P}_{a\sim \mathcal{D}}(|w^T(a-\bar{a})|\ge\theta) \le \eta$ is proposed to be incorporated when performing robust linear regression \cite{shivaswamy2006second,xu2012optimization}, where $\theta$ and $\eta$ are the confidence level and desired probability, respectively, {and $w$ denotes the model coefficients to be optimized}. By approximating this chance constraint with a second-order cone constraint, \cite{shivaswamy2006second} propose the following robust regression problem with conic constraints:
\begin{equation}\label{rbst-socp}
\min_{w\in\R^d,\theta,v\ge 0}  \frac{1}{p}\sum_{i=1}^p\phi(w^Ta_i - b_i) + \lambda_1\theta + \lambda_2 v\quad\mathrm{s.t.}\quad (w,v)\in \mathbb{Q}^{d+1},\quad (\Sigma^{1/2}w,\sqrt{\eta}\theta)\in \mathbb{Q}^{d+1},
\end{equation}
where $\phi(\cdot)$ is the loss, $\lambda_1,\lambda_2>0$ are tuning parameters, and $\mathbb{Q}^{d+1}\doteq\{(u,t)\in\R^d\times\R_+:\|u\|\le t\}$ denotes the second-order cone.

Both examples in \cref{multi-task} and \cref{rbst-socp} involve nonlinear conic constraints and are typically coupled with large datasets in real applications. However, stochastic algorithms for addressing general conic optimization problems in ML remain largely underdeveloped. Existing algorithmic developments for conic constrained ML problems are limited to specific use cases: for instance, alternating minimization is widely employed in the literature \cite{argyriou2008convex,zhang2012convex} to solve \cref{multi-task}, while \cite{shivaswamy2006second} formulates \cref{rbst-socp} as {a} second-order cone {program}, assuming that $\phi$ is convex. Nevertheless, these developments do not unify the treatment of conic constraints and often lack convergence guarantees for problems with nonconvex objective functions.


In this paper, we aim for proposing a stochastic interior-point method (SIPM) framework for smooth conic optimization problems, including \cref{multi-task} and \cref{rbst-socp} as special cases. This class of problems takes the following form:
\begin{equation}\label{eq:cco}
\min_{x}f(x)\quad \mathrm{s.t.}\quad x\in\Omega\doteq\{x:Ax=b,x\in\gK\}, 
\end{equation}
where $f$ is continuously differentiable and possibly nonconvex on $\Omega$, but its derivatives are not accessible. Instead, we rely on stochastic estimators $G(\cdot;\xi)$ of $\nabla f(\cdot)$, where $\xi$ is a random variable with sample space $\Xi$ (see Assumptions \ref{asp:basic}(c) and \ref{asp:ave-smth} for assumptions on $G(\cdot;\xi)$). Here, $A\in\R^{m\times n}$ is of full row rank, $b\in\R^m$, and $\mathcal{K}\subseteq\R^n$ is a closed and pointed convex cone with a nonempty interior. Assume that \cref{eq:cco} has at least one optimal solution.

\subsection{Our contributions}

In this paper, we propose an SIPM framework (\cref{alg:unf-sipm}) for solving problem~\cref{eq:cco}, which, to the best of our knowledge, is the first stochastic algorithmic framework for this problem. Building on \cref{alg:unf-sipm}, we introduce four novel SIPM variants by incorporating different stochastic gradient estimators: mini-batch estimators, Polyak momentum, extrapolated Polyak momentum, and recursive momentum. Under mild assumptions, we establish the {iteration complexity} for these variants. For our SIPMs, we summarize the samples per iteration, {iteration complexity}, and smoothness assumptions in \cref{table:sum-ic}. In addition, numerical results demonstrate the practical advantages of our SIPMs over existing methods.

\begin{table}[h!]
\centering
\caption{\small Samples per iteration, {iteration complexity for finding an $\epsilon$-SSP}, and smoothness assumptions.}
\smallskip
\resizebox{\textwidth}{!}{
\begin{tabular}{c|c|c|c}
\hline
method & samples per iteration & {iteration complexity} & smoothness assumption \\
\hline
SIPM-ME & {$\widetilde{\mathcal{O}}(\epsilon^{-2})$} & {$\widetilde{\mathcal{O}}(\epsilon^{-2})$}  & locally smooth {$\nabla f$} (Asm. \ref{asp:basic}(b)) \\
SIPM-PM & $1$ & {$\widetilde{\mathcal{O}}(\epsilon^{-4})$}  &  locally smooth {$\nabla f$} (Asm. \ref{asp:basic}(b)) \\
SIPM-EM & $1$ & {$\widetilde{\mathcal{O}}(\epsilon^{-7/2})$} & locally smooth {$\nabla f$ \& $\nabla^2 f$} (Asm. \ref{asp:basic}(b) {\&} \ref{asp:2nd-smooth}(b)) \\
SIPM-RM & $1$ &  {$\widetilde{\mathcal{O}}(\epsilon^{-3})$}  & {locally} average smooth {$G$} (Asm. \ref{asp:ave-smth}) \\
\hline
\end{tabular}
}
\label{table:sum-ic}
\end{table}

Our main contributions are highlighted below.
\begin{itemize}
\item We propose an SIPM framework (\cref{alg:unf-sipm}) for solving problem \cref{eq:cco}. To the best of our knowledge, this is the first stochastic algorithmic framework for ML problems with general conic constraints. Building upon this framework, we introduce four novel SIPM variants that incorporate different stochastic gradient estimators.
\item We establish the {iteration complexity} of our SIPMs using mini-batch estimations, Polyak momentum, extrapolated Polyak momentum, {and recursive momentum under \textit{locally Lipschitz-type} conditions (Assumptions \ref{asp:basic}(b), \ref{asp:2nd-smooth}(b), and \ref{asp:ave-smth}).} All of { the iteration complexity results} established in this paper are entirely new and match, up to a {polylogarithmic} factor, the best-known {complexity bounds} in stochastic unconstrained optimization.

\item We conduct numerical experiments (\cref{sec:ne}) to compare our SIPMs with existing methods on robust linear regression, multi-task relationship learning, and clustering data streams. The results demonstrate that our SIPMs achieve solution quality comparable to or better than existing methods, and importantly, they consistently solve ML problems with general conic constraints, unlike existing task-specific approaches. 
\end{itemize}


\subsection{Related work}

\paragraph{Interior-point methods} In the deterministic regime, IPMs are recognized as a fundamental and widely used algorithmic approach for solving constrained optimization problems, having been extensively studied for decades \cite{alizadeh1995interior,byrd1999interior,kim2007interior,NN1994IPM,potra2000interior,vanderbei1999interior,wachter2006implementation,wright1997primal,forsgren2002interior}. In particular, primal-dual IPMs are a popular class of methods that iterate towards to an optimal solution by applying Newton's method to solve a system of equations derived from the perturbed first-order necessary conditions of the optimization problem \cite{wachter2006implementation,wright1997primal}. Another early and classical family of IPMs is the affine scaling method, which iteratively improves a feasible point within the relative interior of the feasible region by scaling the search direction to prevent the solution from exiting the boundary \cite{tseng1992convergence,dvurechensky2024hessian}. The SIPMs developed in this paper are more closely aligned with the algorithmic ideas of affine scaling methods. Moreover, IPM-based solvers are widely adopted for large-scale constrained optimization, including Ipopt \cite{wachter2006implementation}, Knitro \cite{byrd1999interior}, and LOQO \cite{vanderbei1999interior} for functional constrained problems, and SDPT3 \cite{toh1999sdpt3}, SeDuMi \cite{sturm1999using}, {and Mosek \cite{mosek}} for conic constrained problems.

Studies on SIPMs have only emerged recently. In particular, \cite{badenbroek2022complexity} and \cite{narayanan2016randomized} propose randomized IPMs for minimizing a linear function over a convex set. In addition, \cite{curtis2023stochastic} introduces an SIPM for bound-constrained optimization by augmenting the objective function with a log-barrier function, and \cite{curtis2024single} generalizes this approach to solve inequality constrained optimization problems. The optimization problems and algorithmic ideas studied in these previous works differ significantly from those in this paper.

\paragraph{Stochastic first-order methods} Recent significant developments in ML, fueled by large-scale data, have made stochastic first-order optimization methods widely popular for driving ML applications. In particular, many stochastic first-order methods have been developed for solving unconstrained and simple constrained problems of the form $\min_{x\in X} f(x)$, where $X \subseteq \mathbb{R}^n$ is a set for which the projection can be computed exactly \cite{cutkosky2020momentum,cutkosky2019momentum,fang2018spider,ghadimi2013stochastic,lan2020first,tran2022hybrid,wang2019spiderboost,xu2023momentum,li2021page}. Assuming that $f$ has a Lipschitz continuous gradient, {iteration complexity} of {$\mathcal{O}(\epsilon^{-4})$} has been established for the methods in \cite{ghadimi2013stochastic,cutkosky2020momentum}, in terms of minimizing the stationary measure $\mathrm{dist}(0, \nabla f(x^k) + \mathcal{N}_X(x^k))$, where $\mathcal{N}_X(x^k)$ denotes the normal cone of $X$ at $x^k$. The {iteration complexity can be improved to $\mathcal{O}(\epsilon^{-7/2})$} by using an implicit gradient transport technique, assuming that $f$ has a Lipschitz continuous Hessian \cite{cutkosky2020momentum}, and to {$\mathcal{O}(\epsilon^{-3})$} by using various variance reduction techniques, assuming that $\nabla f$ has a stochastic estimator satisfying the average smoothness condition \cite{cutkosky2019momentum, fang2018spider, tran2022hybrid, wang2019spiderboost, xu2023momentum,li2021page}. In addition, a number of recent efforts have been devoted to equality constrained optimization in the stochastic regime, including sequential quadratic programming methods \cite{berahas2023stochastic,berahas2021sequential,berahas2023accelerating,curtis2021inexact,fang2024fully,na2023adaptive,na2022asymptotic} and penalty methods \cite{alacaoglu2024complexity,li2024stochastic,lu2024variance,shi2022momentum}.

\section{Notation and preliminaries}\label{sec:not-pre}
Throughout this paper, let $\R^n$ denote the $n$-dimensional Euclidean space and $\langle\cdot,\cdot\rangle$ denote the standard inner product. We use $\|\cdot\|$ to denote the Euclidean norm of a vector or the spectral norm of a matrix. For the closed convex cone $\gK$, its interior and dual cone are denoted by $\rm{int}\gK$ and $\gK^*$, respectively. Define the {perturbed} barrier function of problem \cref{eq:cco} as:
\begin{align}\label{def:feas-r}
\phi_\mu(x)\doteq f(x) + {\mu (f(x)+B(x))}\qquad \forall x\in\Omega^\circ \doteq \{x \in \mathrm{int} \gK : Ax = b\}\ {\text{and}\  \mu\in(0,1]}.
\end{align}
For a finite set $\mathcal{B}$, let $|\mathcal{B}|$ denote its cardinality. For any $t\in\R$, let $[t]_+$ and $[t]_-$ denote its nonnegative and nonpositive parts, respectively (i.e., set to zero if $t$ is {negative} or {positive}, respectively). {Also, let $\lfloor t\rfloor$ denote the largest integer less than or equal to $t$.} We use the standard big-O notation $\gO(\cdot)$ to present {the complexity}, and $\widetilde{\gO}(\cdot)$ to represent the order with a {polylogarithmic} factor omitted. 

For the rest of this section, we review some background on logarithmically homogeneous self-concordant barriers and introduce the approximate optimality conditions.

\subsection{Logarithmically homogeneous self-concordant barrier}

Logarithmically homogeneous self-concordant (LHSC) barrier functions play a crucial role in the development of IPMs for conic optimization (see \cite{NN1994IPM}). in this paper, the design and analysis of SIPMs also heavily rely on the LHSC barrier function. Throughout this paper, assume that $\gK$ is equipped with a {\it $\vartheta$-LHSC} barrier function $B$, {where $\vartheta\ge1$}. Specifically, $B:\mathrm{int}\mathcal{K}\to\R$ satisfies the following conditions: (i) $B$ is convex and three times continuously differentiable in $\mathrm{int} \mathcal{K}$, and moreover, $|\varphi^{\prime\prime\prime}(0)|\le2(\varphi^{\prime\prime}(0))^{3/2}$ holds for all $x\in\mathrm{int} \mathcal{K}$ and $u\in\mathbb{R}^n$, where $\varphi(t)=B(x+tu)$; (ii) $B$ is a {\it barrier function} for $\mathcal{K}$, meaning that $B(x)$ goes to infinity as $x$ approaches the boundary of $\mathcal{K}$; (iii) $B$ satisfies the {\it logarithmically homogeneous property}: $B(tx) = B(x) -\vartheta \ln t$ for all $ x\in\mathrm{int}\mathcal{K},t>0$.

For any $x\in\mathrm{int}\gK$, the function $B$ induces the following local norms for vectors:
\begin{align}\label{def:local-norm}
\|v\|_x\doteq \left(v^T\nabla^2B(x)v\right)^{1/2},\quad \|v\|_x^*\doteq \left(v^T\nabla^2B(x)^{-1}v\right)^{1/2}\qquad \forall v\in\R^n.
\end{align}
The induced local norm for matrices is given by:
\begin{equation}\label{def:local-m-norm}
\|M\|_x^*\doteq \max_{\|v\|_x\le 1}\|Mv\|_x^*\qquad \forall M\in\R^{n\times n}.
\end{equation}

\subsection{Approximate optimality conditions}

Since $f$ is nonconvex, finding a global solution to \cref{eq:cco} is generally impossible. Instead, we aim to find a point that satisfies approximate optimality conditions, as is common in nonconvex optimization. For deterministic IPMs developed to solve \cref{eq:cco}, the following approximate optimality conditions are proposed in \cite{he2023newton,he2023newtonal}: 
\begin{equation}\label{def:1st-sc}
x\in\mathrm{int}\gK,\quad Ax=b,\quad \nabla f(x)+A^T{\tilde{\lambda}}\in\gK^*,\quad \|\nabla f(x)+A^T{\tilde{\lambda}}\|_x^*\le\epsilon,
\end{equation}
{where $\epsilon\in(0,1)$ denotes the tolerance}. 
In addition, stochastic algorithms typically produce solutions that satisfy approximate optimality conditions only in expectation. To facilitate our developments of SIPMs, we next derive an alternative approximate optimality condition for \cref{eq:cco} using $B$, which is a sufficient condition for \cref{def:1st-sc}. Its proof is deferred to \cref{subsec:proof-not-pre}.


\begin{lemma}\label{lem:apr-stat}
Let $\mu>0$ {be given}. Suppose that $(x,\lambda)\in\Omega^\circ\times\R^m$ satisfies $\|\nabla\phi_\mu(x) + A^T\lambda\|_x^* \le \mu$, where $\phi_\mu$ and $\Omega^\circ$ are given in \cref{def:feas-r}. Then, {$x$} also satisfies \cref{def:1st-sc} with {$\tilde{\lambda}=\lambda/(1+\mu)$ and any $\epsilon\ge(1+\sqrt{\vartheta})\mu$}.
\end{lemma}
The above lemma offers an alternative approximate optimality condition for \cref{eq:cco}. We extend this condition to an expectation form and define an {approximate} stochastic stationary point for problem \cref{eq:cco}, which our SIPMs aim to achieve.

\begin{definition}\label{def:ssp}
Let {$\epsilon\in(0,1)$}. We say that $x\in\Omega^\circ$ is {an $\epsilon$}-stochastic stationary point {($\epsilon$-SSP)} of problem \cref{eq:cco} if it, together with some $\lambda\in\R^m$, satisfies $\E[\|\nabla\phi_\mu(x) + A^T\lambda\|_x^*] \le \mu$ {for some $\mu\le\epsilon/(1+\sqrt{\vartheta})$}.
\end{definition}

{One can also define an approximate stochastic stationary point for \cref{eq:cco} that satisfies \cref{def:1st-sc} with high probability, as follows.

\begin{definition}
Let $\epsilon,\delta\in(0,1)$. We say that $x\in\Omega^\circ$ is an $(\epsilon,\delta)$-stochastic stationary point ($(\epsilon,\delta)$-SSP) of problem \cref{eq:cco} if it, together with some $\lambda\in\R^m$, satisfies $\|\nabla\phi_\mu(x) + A^T\lambda\|_x^*\le \mu$ for some $\mu\le\epsilon/(1+\sqrt{\vartheta})$ with probability at least $1-\delta$.    
\end{definition}

Using Markov's inequality, we obtain that
\begin{align*}
\mathbb{P}(\|\nabla\phi_\mu(x) + A^T\lambda\|_x^*> \mu) \le \frac{\E[\|\nabla\phi_\mu(x) + A^T\lambda\|_x^*]}{\mu}\qquad\forall \mu>0.
\end{align*}
Therefore, if $x\in\Omega^\circ$, together with $\lambda\in\R^m$, satisfies $\E[\|\nabla\phi_\mu(x) + A^T\lambda\|_x^*] \le \delta\mu$ for some $\mu\le\epsilon/(1+\sqrt{\vartheta})$ and $\epsilon,\delta\in(0,1)$, one can derive that $x$ is an $(\epsilon,\delta)$-SSP of problem \cref{eq:cco}, which leads to the following lemma.

\begin{lemma}\label{lem:prob-exp}
Let $\epsilon,\delta\in(0,1)$ be given. Suppose that $(x,\lambda)\in\Omega^\circ\times\R^m$ satisfies $\E[\|\nabla\phi_\mu(x) + A^T\lambda\|_x^*] \le \delta\mu$ for some $\mu\le\epsilon/(1+\sqrt{\vartheta})$, where $\phi_\mu$ and $\Omega^\circ$ are given in \cref{def:feas-r}. Then, $x$ is an $(\epsilon,\delta)$-SSP of problem \cref{eq:cco}.    
\end{lemma}

For the remainder of this paper, we focus on establishing the iteration complexity of our proposed SIPMs for finding an $\epsilon$-SSP of problem \cref{eq:cco}. While we do not discuss it in detail, by applying Lemma \ref{lem:prob-exp}, one can also extend our analysis to establish the iteration complexity of our SIPMs for finding an $(\epsilon,\delta)$-SSP of problem \cref{eq:cco}.

}


\section{Stochastic interior-point methods}\label{sec:sasa}

In this section, we propose an SIPM framework for solving \cref{eq:cco} and then analyze the {iteration complexity} of four SIPM variants. We now make the following additional assumptions that will be used throughout this section.

\begin{assumption}\label{asp:basic}
\begin{enumerate}[{\rm (a)}]
\item The Slater's condition holds, that is, $\Omega^\circ$, defined in \cref{def:feas-r}, is nonempty. In addition, there exists a finite $\phi_{\mathrm{low}}$ such that 
\begin{align}\label{def:lwbd-phi}
\inf_{x\in\Omega^\circ,\mu\in(0,1]}\{f(x)+\mu B(x)\}\ge \phi_{\mathrm{low}}.  
\end{align}
\item {For any $s_\eta\in(0,1)$,} there exists {an} $L_1>0$ such that
\begin{align}\label{ineq:1st-Lip}
\|\nabla f(y) - \nabla f(x)\|_x^* \le L_1\|y-x\|_x\qquad \forall x,y\in\Omega^\circ\text{ with } \|y-x\|_x\le s_\eta.
\end{align}
\item {We have access to a} stochastic gradient estimator $G:\Omega^\circ\times\Xi\to\R^n$ {that} satisfies
\begin{align}\label{asp:unbias-boundvar}
\E_\xi[G(x;\xi)] = \nabla f(x),\quad \E_\xi[(\|G(x;\xi)-\nabla f(x)\|_x^*)^2]\le\sigma^2\qquad \forall x\in\Omega^\circ
\end{align} 
for some $\sigma>0$.
\end{enumerate}
\end{assumption}

\begin{remark}\label{rmk:basic-asm}
\begin{enumerate}[{\rm (i)}]
\item {Assumption \ref{asp:basic}(a) is reasonable. Particularly, the assumption in \cref{def:lwbd-phi} means that the barrier function $f(x)+\mu B(x)$ is uniformly bounded below whenever the barrier parameter $\mu$ is no larger than $1$. It usually holds for problems where the barrier method converges, and similar assumptions are also used in \cite{he2023newton,he2023newtonal}. Otherwise, in case that \cref{def:lwbd-phi} fails, one can instead solve a perturbed counterpart of \cref{eq:cco}:
\begin{align}\label{perturbed-cco}
\min_{x} f(x) + \tau \|x\|^2\quad\mathrm{s.t.}\quad Ax=b,x\in\gK
\end{align}
for a sufficiently small $\tau>0$. It can be verified that \cref{def:lwbd-phi} holds for the perturbed problem \cref{perturbed-cco} with $f(\cdot)$ replaced by $f(\cdot)+\tau\|\cdot\|^2$. Indeed, let $f^*$ be the optimal value of \cref{eq:cco}. Then, \begin{align*}
&\inf_{x\in\Omega^\circ,\mu\in(0,1]}\{f(x) + \tau\|x\|^2 + \mu B(x)\}\\
&\ge f^* + \inf_{\mu\in(0,1]}\Big\{\mu\inf_{x\in\Omega^\circ}\{(\tau/\mu)\|x\|^2 + B(x)\}\Big\}\ge f^* + \inf_{\mu\in(0,1]}\Big\{\mu\inf_{x\in\Omega^\circ}\{\tau\|x\|^2 + B(x)\}\Big\}\\     
&\ge f^* - \Big|\inf_{x\in\Omega^\circ}\{\tau\|x\|^2 + B(x)\}\Big| > -\infty,
\end{align*}
where the last inequality is due to the strong convexity of $\tau\|x\|^2 + B(x)$. Hence, the assumption in \cref{def:lwbd-phi} holds for the perturbed problem \cref{perturbed-cco}.}

\item {As a consequence of Assumption \ref{asp:basic}(a), one can see that the perturbed barrier function defined in \cref{def:feas-r} is bounded below:
\begin{align}\label{to-lwbd-phimu}
\phi_\mu(x) = (1+\mu) \Big(f(x) + \frac{\mu}{1+\mu}B(x)\Big) \ge (1+\mu)\phi_{\mathrm{low}}\qquad\forall \mu\in(0,1],x\in\Omega^\circ.   
\end{align}
In addition, for notational convenience, we define
\begin{align}\label{lwbd-fb}
\Delta(x)\doteq f(x)  + [f(x) + B(x)]_+ - 2[\phi_{\mathrm{low}}]_-\qquad \forall x\in\Omega^\circ.
\end{align} 
}

\item \cref{asp:basic}(b) implies that $\nabla f$ is locally Lipschitz continuous on $\Omega^\circ$ with respect to local norms. {Similar local smoothness assumptions are also used for other barrier methods \cite{he2023newton,dvurechensky2024hessian,he2023newtonal}. As will be shown in Lemma \ref{lem:tech-barrier}(iv), $\nabla B$ is locally Lipschitz continuous on $\Omega^\circ$ with respect to local norms, even if it is not well defined on the boundary of $\Omega$. In general, we recall from \cite{he2023newton} that for any bounded $x$, $\nabla^2B(x)^{-1}$ is bounded with respect to the spectral norm. Then, one can show that when $\Omega$ is bounded, the locally Lipschitz condition \cref{ineq:1st-Lip} can be implied by the globally Lipschitz condition (see \cite[Section 5]{he2023newton} for detailed discussions).} For convenience, we define
\begin{align}\label{def:Lphi-1}
{ L_\phi\doteq 2L_1 + 1/(1-s_\eta),}
\end{align}
which denotes the Lipschitz constant of $\nabla\phi_\mu$ for any $\mu\in(0,1]$ (see \cref{lem:Lip-phimu} below).

\item \cref{asp:basic}(c) implies that $G(\cdot;\xi)$ is an unbiased estimator of $\nabla f(\cdot)$ and that its variance, with respect to the local norm, is bounded above. As noted in \cref{rmk:basic-asm}(iii), $\nabla^2 B(x)^{-1}$ is bounded {with respect to the spectral norm for any bounded $x$. Using} this and \cref{def:local-norm}, we have that {when $\Omega$ is bounded,} the second relation in \cref{asp:unbias-boundvar} holds if the variance of $G(\cdot;\xi)$ with respect to the Euclidean norm is bounded.
\end{enumerate}    
\end{remark}

In what follows, we propose an SIPM framework in \cref{alg:unf-sipm} for solving problem~\cref{eq:cco}. Subsequently, we will employ four distinct stochastic estimators to construct $\{\overline{m}^k\}_{k\ge0}$, with specific schemes provided in \cref{sipm-over-mk}, \cref{sipm-pm-over-mk}, \cref{sipm-em-over-mk}, and \cref{sipm-rm-mk-up}, respectively. 

We remark that computations involving $\nabla^2 B(x^k)^{-1}$ can be performed efficiently for common nonlinear cones $\gK$. For example, when $\gK$ is the second-order cone, $\nabla B(x^k)^{-1}$ can be computed analytically (e.g., see \cite{alizadeh2003second}), and $(A\nabla^2 B(x^k)^{-1}A^T)^{-1}v$ for any $v \in \R^m$ can be efficiently evaluated using Cholesky factorization. When $\gK$ is the semidefinite cone, we have that $\nabla^2 B(X^k)^{-1}[V]=X^kVX^k$ holds for any $n\times n$ symmetric matrix $V$, and that $A\nabla^2 B(X^k)^{-1}A^T$ can be efficiently computed by exploiting the sparsity of $A$ (see \cite{toh1999sdpt3} for details).

\begin{algorithm}[!htbp]
\caption{An SIPM framework}
\label{alg:unf-sipm}
\begin{algorithmic}[0]
\State \textbf{Input:} starting point $x^0\in\Omega^\circ$, {nonincreasing step sizes $\{\eta_k\}_{k\ge0}\subset(0,s_\eta]$ with $s_\eta\in(0,1)$}, nonincreasing barrier parameters $\{\mu_k\}_{k\ge0}\subset(0,1]$.
\For{$k=0,1,2,\ldots$}
\State Construct an estimator $\overline{m}^k$ for $\nabla f(x^k)$, and set $m^k = \overline{m}^k + \mu_k {(\overline{m}^k+\nabla B(x^k))}$.
\State Update dual and primal variables as follows:
\begin{align}\label{step-dp-update}
\lambda^k = -(A H_k A^T)^{-1}A H_k m^k,\quad x^{k+1} = x^k - \eta_k\frac{H_k(m^k + A^T\lambda^k)}{\|m^k + A^T\lambda^k\|_{x^k}^*}, 
\end{align}
\State where $H_k=\nabla^2 B(x^k)^{-1}$.
\EndFor
\end{algorithmic}
\end{algorithm}

Our next lemma shows that all iterates $\{x^k\}_{k\ge0}$ generated by \cref{alg:unf-sipm} are strictly feasible, i.e., $x^k\in\Omega^\circ$ holds for all $k\ge0$. Its proof is deferred to \cref{subsec:proof-not-pre}.

\begin{lemma}\label{lem:int-frame}
Suppose that \cref{asp:basic} holds. Let $\{x^k\}_{k\ge0}$ be generated by \cref{alg:unf-sipm}. Then, $\|x^{k+1} - x^k\|_{x^k}= \eta_k$ and $x^k\in\Omega^\circ$ for all $k\ge0$, where $\Omega^\circ$ is defined in \cref{def:feas-r}.
\end{lemma}

In the remainder of this section, we propose four variants of SIPMs in \cref{subsec:ipm-mb,subsec:sipm-pm,subsec:sipm-em,subsec:sipm-rm} and study their {iteration complexity}. The developments and guarantees of these four SIPM variants are independent of each other.




\subsection{An SIPM with mini-batch estimators}\label{subsec:ipm-mb}
We now describe a variant of \cref{alg:unf-sipm}, where $\{\overline{m}^k\}_{k\ge0}$ is constructed using mini-batch estimators:
\begin{align}\label{sipm-over-mk}
\overline{m}^k = \sum\nolimits_{i\in \mathscr{B}_k} G(x^k;\xi_i^k)/|\mathscr{B}_k|\qquad \forall k\ge0,
\end{align}
where $G(\cdot,\cdot)$ satisfies \cref{asp:basic}(c), and the sequence $\{\mathscr{B}_k\}_{k\ge0}$ denotes the sets of sample indices. We refer this variant as SIPM with mini-batch estimators (SIPM-ME).

The following lemma establishes an upper bound for the estimation error of the mini-batch estimators defined in \cref{sipm-over-mk}, and its proof is deferred to \cref{subsec:proof-me}.

\begin{lemma}\label{lem:est-error-me}
Suppose that \cref{asp:basic} holds. Let $\{x^k\}_{k\ge0}$ be generated by \cref{alg:unf-sipm} with $\{\overline{m}^k\}_{k\ge0}$ constructed as in \cref{sipm-over-mk} and input parameters $\{(\eta_k,\mathscr{B}_k)\}_{k\ge0}$. Then,
\begin{align}\label{bd:variance}
\E_{\{\xi_i^k\}_{i\in \mathscr{B}_k}}[(\|\overline{m}^k - \nabla f(x^k)\|_{x^k}^*)^2] \le \sigma^2/|\mathscr{B}_k|\qquad \forall k\ge0,
\end{align}   
where $\sigma$ is given in \cref{asp:basic}(c).
\end{lemma}

We next provide an upper bound for the average expected error of the stationary condition across all iterates generated by SIPM-ME. Its proof is relegated to \cref{subsec:proof-me}.

\begin{theorem}\label{thm:stat-ave-me}
Suppose that \cref{asp:basic} holds. Let $\{(x^k,\lambda^k)\}_{k\ge0}$ be the sequence generated by \cref{alg:unf-sipm} with $\{\overline{m}^k\}_{k\ge0}$ constructed as in \cref{sipm-over-mk} and input parameters $\{(\eta_k,\mathscr{B}_k,\mu_k)\}_{k\ge0}$. Assume that $\{\eta_k\}_{k\ge0}$ is nonincreasing. Then, for any $K\ge1$,
\begin{align}
\sum_{k=0}^{K-1} \E[\|\nabla \phi_{\mu_k}(x^k) + A^T{\lambda}^k\|_{x^k}^*]\le \frac{\Delta(x^0)}{\eta_{K-1}} +\frac{1}{\eta_{K-1}}\sum_{k=0}^{K-1}\eta_k\left(\frac{{4}\sigma}{|\mathscr{B}_k|^{1/2}}+\frac{L_\phi}{2}\eta_k\right),\label{upbd:me-exp-xk}
\end{align}
where $\Delta(\cdot)$, $L_\phi$, and $\sigma$ are given in \cref{lwbd-fb}, \cref{def:Lphi-1}, and \cref{asp:basic}(c), respectively.
\end{theorem}

\subsubsection{Hyperparameters and {iteration complexity}}

{In this subsection, we establish the iteration complexity of SIPM-ME with its input parameters specified $\{(\eta_k,\mathscr{B}_k,\mu_k)\}_{k\ge0}$ as:}
\begin{align}
\eta_k = \frac{s_\eta}{(k+1)^{1/2}},\quad |\mathscr{B}_k| = k+1,\quad \mu_k{=\max\Big\{\frac{1}{(k+1)^{1/2}},\frac{\epsilon}{1+\sqrt{\vartheta}}\Big\}} \qquad \forall k\ge0,\label{etak-Bk-me}
\end{align}
where $s_\eta\in(0,1)$ is a user-defined maximum length for step sizes in \cref{alg:unf-sipm}, {and $\epsilon\in(0,1)$ denotes the tolerance.} It follows that $\{\eta_k\}_{k\ge0}\subset(0,s_\eta]$ and $\{\mu_k\}_{k\ge0}\subset(0,1]$, with both sequences nonincreasing.


The following theorem presents the {iteration complexity} of SIPM-ME with inputs specified in \cref{etak-Bk-me}. Its proof is deferred to \cref{subsec:proof-me}.

\begin{theorem}\label{cor:order-me}
Suppose that \cref{asp:basic} holds. Consider \cref{alg:unf-sipm} with $\{\overline{m}^k\}_{k\ge0}$ constructed as in \cref{sipm-over-mk} and $\{(\eta_k,\mathscr{B}_k,\mu_k)\}_{k\ge0}$ specified as in \cref{etak-Bk-me}. Let {$\kappa(K)$} be uniformly drawn from {$\{\lfloor K/2\rfloor,\ldots,K-1\}$}, and define
\begin{align}\label{def:Kme}
{M_{\mathrm{me}}\doteq 2\Big(\frac{\Delta(x^0)}{s_\eta} + 8\sigma+ s_\eta L_\phi\Big),}
\end{align}
where $\Delta(\cdot)$ is defined in \cref{lwbd-fb}, $L_\phi$ and $\sigma$ are given in \cref{def:Lphi-1} and \cref{asp:basic}(c), respectively, and $s_\eta$ is an input of \cref{alg:unf-sipm}. Then, {
\begin{align}
&\E[\|\nabla \phi_{\mu_{\kappa(K)}}(x^{\kappa(K)}) + A^T\lambda^{\kappa(K)}\|_{x^{\kappa(K)}}^*] \le \mu_{\kappa(K)}\ \text{ with }\ \mu_{\kappa(K)}\le\epsilon/(1+\sqrt{\vartheta})\nonumber\\
&\forall K\ge \max\Big\{2\Big(\frac{1+\sqrt{\vartheta}}{\epsilon}\Big)^2,\Big(\frac{4M_{\mathrm{me}}(1+\sqrt{\vartheta})}{\epsilon}\ln\Big(\frac{4M_{\mathrm{me}}(1+\sqrt{\vartheta})}{\epsilon}\Big)\Big)^2,3\Big\}.  \label{complexity-me}
\end{align}
}

\vspace{-1em}
\end{theorem}
{
\begin{remark}\label{rmk:me-rate}
From Theorem \ref{cor:order-me}, we observe that SIPM-ME returns an $\epsilon$-SSP within $\widetilde{\mathcal{O}}(\epsilon^{-2})$ iterations. In addition, by this and the definition of $|\mathscr{B}_k|$, one can see that the stochastic gradient evaluations per iteration of SIPM-ME is at most $\widetilde{\mathcal{O}}(\epsilon^{-2})$.
\end{remark}
}
\subsection{An SIPM with Polyak momentum}\label{subsec:sipm-pm}
We now propose a variant of \cref{alg:unf-sipm}, in which $\{\overline{m}^k\}_{k\ge0}$ is constructed using Polyak momentum \cite{cutkosky2020momentum} as follows:
\begin{align}
\gamma_{-1}=1,\quad \overline{m}^{-1}=0,\quad \overline{m}^k = (1-\gamma_{k-1}) \overline{m}^{k-1} + \gamma_{k-1}G(x^k,\xi^k)\qquad \forall k\ge0,   \label{sipm-pm-over-mk}
\end{align}
where $G(\cdot,\cdot)$ satisfies \cref{asp:basic}(c), and $\{\gamma_k\}_{k\ge0}{\subset(0,1]}$ denotes the sequence of momentum parameters. We refer this variant as SIPM with Polyak momentum (SIPM-PM).

The next lemma provides the recurrence relation for the estimation error of the gradient estimators based on Polyak momentum defined in \cref{sipm-pm-over-mk}. Its proof is deferred to \cref{subsec:proof-pm}.

\begin{lemma}\label{lem:pm-estimate}
Suppose that \cref{asp:basic} holds. Let $\{x^k\}_{k\ge0}$ be generated by \cref{alg:unf-sipm} with $\{\overline{m}^k\}_{k\ge0}$ constructed as in \cref{sipm-pm-over-mk} and input parameters $\{(\eta_k,\gamma_k)\}_{k\ge0}$. For all $k\ge0$, assume that $\gamma_k>\eta_k$ holds, and define $\alpha_k=1-(1-\gamma_k)/(1-\eta_k)$. Then,
\begin{align}
&\E_{\xi^{k+1}}[(\|\overline{m}^{k+1} - \nabla f(x^{k+1})\|_{x^{k+1}}^*)^2]\nonumber\\ 
&\le (1-\alpha_k)(\|\overline{m}^k - \nabla f(x^k)\|_{x^k}^*)^2 + \frac{L_1^2\eta_k^2}{\alpha_k} + \sigma^2\gamma_k^2\qquad \forall k\ge0,\label{ineq:var-recur}
\end{align}
where $L_1$ and $\sigma$ are given in \cref{asp:basic}.
\end{lemma}

We now derive an upper bound for the average expected error of the stationary condition across all iterates generated by SIPM-PM. Its proof is deferred to \cref{subsec:proof-pm}.

\begin{theorem}\label{thm:converge-rate-ipmpm}
Suppose that \cref{asp:basic} holds. Let $\{(x^k,\lambda^k)\}_{k\ge0}$ be the sequence generated by \cref{alg:unf-sipm} with $\{\overline{m}^k\}_{k\ge0}$ constructed as in \cref{sipm-pm-over-mk} and input parameters $\{(\eta_k,\gamma_k,\mu_k)\}_{k\ge0}$. Assume that $\{\eta_k\}_{k\ge0}$ is nonincreasing and also that $\gamma_k>\eta_k$ holds for all $k\ge0$. Define $\alpha_k=1-(1-\gamma_k)/(1-\eta_k)$ for all $k\ge0$. Then, for all $K\ge1$,
\begin{align}
\sum_{k=0}^{K-1}\E[\|\nabla \phi_{\mu_k}(x^k) + A^T\lambda^k\|^*_{x^k}]\le \frac{\Delta(x^0) + \sigma^2/L_1}{\eta_{K-1}} + \frac{1}{\eta_{K-1}}\sum_{k=0}^{K-1}\left(\frac{L_\phi\eta_k^2}{2} + \frac{{5}L_1\eta_k^2}{\alpha_k} + \frac{\sigma^2\gamma_k^2}{L_{1}}\right),\label{ineq:ave-stat-pm}
\end{align}
where $\Delta(\cdot)$ is defined in \cref{lwbd-fb}, $L_\phi$ is given in \cref{def:Lphi-1}, and $L_1$ and $\sigma$ are given in \cref{asp:basic}.
\end{theorem}

\subsubsection{Hyperparameters and {iteration complexity}}

{In this subsection, we establish the iteration complexity of SIPM-PM with its input parameters $\{(\eta_k,\gamma_k,\mu_k)\}_{k\ge0}$ specified as:}
\begin{align}
{\eta_k= \frac{s_\eta}{(k+1)^{3/4}},\quad \gamma_k = \frac{1}{(k+1)^{1/2}},\quad \mu_k =\max\Big\{\frac{1}{{(k+1)^{1/4}}},\frac{\epsilon}{1+\sqrt{\vartheta}}\Big\} \qquad \forall k\ge0,}\label{etak-Bk-pm}
\end{align}    
where $s_\eta\in(0,1)$ is a user-defined input of \cref{alg:unf-sipm}, {and $\epsilon\in(0,1)$ denotes the tolerance.} It can be verified that $\{\eta_k\}_{k\ge0}\subset(0,s_\eta]$ and $\{\mu_k\}_{k\ge0}\subset(0,1]$, with both sequences nonincreasing. From \cref{etak-Bk-pm} and $s_{\eta}\in(0,1)$, we observe that the sequence $\{\alpha_k\}_{k\ge0}$ defined in \cref{lem:pm-estimate,thm:converge-rate-ipmpm} satisfies:
\begin{align}\label{ineq:lwbd-ak-pm}
{\alpha_{k} = \frac{(k+1)^{1/4}-s_\eta}{(k+1)^{3/4}-s_\eta}> \frac{(k+1)^{1/4}-s_\eta}{(k+1)^{3/4}} = \frac{1-s_\eta/(k+1)^{1/4}}{(k+1)^{1/2}} \ge \frac{1-s_\eta}{(k+1)^{1/2}} \qquad \forall k\ge0.  }  
\end{align}

The following theorem presents the {iteration complexity} of SIPM-PM with its inputs specified in \cref{etak-Bk-pm}. Its proof is relegated to \cref{subsec:proof-pm}.

\begin{theorem}\label{cor:order-pm}
Suppose that \cref{asp:basic} holds. Consider \cref{alg:unf-sipm} with $\{\overline{m}^k\}_{k\ge0}$ constructed as in \cref{sipm-pm-over-mk} and $\{(\eta_k,\gamma_k,\mu_k)\}_{k\ge0}$ specified as in \cref{etak-Bk-pm}. Let {$\kappa(K)$} be uniformly drawn from {$\{\lfloor K/2\rfloor,\ldots,K-1\}$}, and define
\begin{align}\label{def:kpm}
{M_{\rmpm}\doteq 2\Big(\frac{\Delta(x^0) + \sigma^2/L_1}{s_{\eta}}  + \frac{3s_\eta L_\phi}{2} + 2\Big(\frac{5s_\eta L_1}{1-s_\eta}+\frac{\sigma^2}{s_{\eta}L_1}\Big)\Big),}     
\end{align}
where $\Delta(\cdot)$ is defined in \cref{lwbd-fb}, $L_\phi$ is given in \cref{def:Lphi-1}, $L_1$ and $\sigma$ are given in \cref{asp:basic}, and $s_\eta$ is an input of \cref{alg:unf-sipm}. Then,
{
\begin{align}
&\E[\|\nabla \phi_{\mu_{\kappa(K)}}(x^{\kappa(K)}) + A^T\lambda^{\kappa(K)}\|_{x^{\kappa(K)}}^*] \le \mu_{\kappa(K)}\ \text{ with }\ \mu_{\kappa(K)}\le\epsilon/(1+\sqrt{\vartheta})\nonumber\\
&\forall K\ge\max\Big\{2\Big(\frac{1+\sqrt{\vartheta}}{\epsilon}\Big)^4,\Big(\frac{8M_{\mathrm{pm}}(1+\sqrt{\vartheta})}{\epsilon}\ln\Big(\frac{8M_{\mathrm{pm}}(1+\sqrt{\vartheta})}{\epsilon}\Big)\Big)^4,3\Big\}. \label{complexity-pm}
\end{align}
}
\end{theorem}

{
\begin{remark}\label{rmk:pm-rate}
\begin{enumerate}[{\rm (i)}]
    \item From Theorem \ref{cor:order-pm}, we see that SIPM-PM returns an $\epsilon$-SSP within $\widetilde{\mathcal{O}}(\epsilon^{-4})$ iterations. This complexity result matches that of stochastic unconstrained optimization with Lipschitz continuous gradient \cite{ghadimi2013stochastic,cutkosky2020momentum}, up to a polylogarithmic factor.
    \item It is worth mentioning that our analysis implies that only a point selected from a subset of the first $K$ iterates generated by SIPM-PM (and similarly for other variants) is an $\epsilon$-SSP, provided that $K$ is sufficiently large. This guarantee is consistent with those in \cite{cutkosky2019momentum,cutkosky2020momentum} for stochastic first-order methods with momentum in unconstrained optimization. It would be interesting to develop SIPMs with a stronger guarantee that all $x^K$ are $\epsilon$-SSPs when $K$ is larger than a threshold, which we leave as a direction for future research. 
\end{enumerate}
\end{remark}
}

\subsection{An SIPM with extrapolated Polyak momentum}\label{subsec:sipm-em}

We now propose a variant of \cref{alg:unf-sipm}, where $\{\overline{m}^k\}_{k\ge0}$ is constructed based on extrapolated Polyak momentum \cite{cutkosky2020momentum} as follows:
\begin{subequations}\label{sipm-em-over-mk}
\begin{align}
&\hspace{-3mm}\gamma_{-1}=1,\quad x^{-1}=x^0,\quad \overline{m}^{-1}=0,\label{sipm-em-zk}\\
&\hspace{-3mm}z^k = x^k + \frac{1-\gamma_{k-1}}{\gamma_{k-1}}(x^k - x^{k-1}),\quad  \overline{m}^k = (1-\gamma_{k-1}) \overline{m}^{k-1} + \gamma_{k-1}G(z^k,\xi^k)\qquad \forall k\ge0, \label{sipm-em-over-mk-1}  
\end{align}    
\end{subequations}
where $G(\cdot,\cdot)$ satisfies \cref{asp:basic}(c), and $\{\gamma_k\}_{k\ge0}{\subset(0,1]}$ denotes momentum parameters. We refer to this variant as SIPM with extrapolated Polyak momentum (SIPM-EM). 

To analyze SIPM-EM, we make the following additional assumption regarding the local Lipschitz continuity of $\nabla^2 f$.

\begin{assumption}\label{asp:2nd-smooth}
\begin{enumerate}[{\rm (a)}]
\item The function $f$ is twice continuously differentiable on $\Omega^\circ$.
\item {For any $s_\eta\in(0,1)$,} there exists {an} $L_2>0$ such that 
\begin{align}\label{ineq:2nd-Lip}
\|\nabla^2 f(y) - \nabla^2 f(x)\|_x^*  \le L_2\|y-x\|_x\qquad \forall x,y\in\Omega^\circ\text{ with } \|y-x\|_x\le s_{\eta}.
\end{align}
\end{enumerate}
\end{assumption}

The following lemma shows that the iterates $\{z^k\}_{k\ge0}$ generated by SIPM-EM lie in $\Omega^\circ$. Its proof is deferred to \cref{subsec:proof-em}.

\begin{lemma}\label{lem:feas-em}
Suppose that Assumptions \ref{asp:basic} and \ref{asp:2nd-smooth} hold. Let $\{z^k\}_{k\ge0}$ be generated by \cref{alg:unf-sipm} with $\{\overline{m}^k\}_{k\ge0}$ constructed as in \cref{sipm-em-over-mk} and input parameters $\{(\eta_k,\gamma_k)\}_{k\ge0}$. Assume $\eta_k/\gamma_k\le s_{\eta}$ for all $k\ge0$, where $s_{\eta}$ is an input of \cref{alg:unf-sipm}. Then, $z^k\in\Omega^\circ$ for all $k\ge0$, where $\Omega^\circ$ is defined in \cref{def:feas-r}.
\end{lemma}



The next lemma provides the recurrence relation for the estimation error of the gradient estimators based on extrapolated Polyak momentum defined in \cref{sipm-em-over-mk}. Its proof is deferred to \cref{subsec:proof-em}.

\begin{lemma}\label{lem:em-estimate}
Suppose that Assumptions \ref{asp:basic} and \ref{asp:2nd-smooth} hold. Let $\{x^k\}_{k\ge0}$ be generated by \cref{alg:unf-sipm} with $\{\overline{m}^k\}_{k\ge0}$ constructed as in \cref{sipm-em-over-mk} and input parameters $\{(\eta_k,\gamma_k)\}_{k\ge0}$. Assume that $\{\eta_k\}_{k\ge0}$ is nonincreasing, {$\{\gamma_k\}_{k\ge0}\subset(0,1]$,} and that $\eta_k/\gamma_k\le s_{\eta}$ holds for all $k\ge0$, where $s_{\eta}$ is an input of \cref{alg:unf-sipm}. Define $\alpha_k=1-(1-\gamma_k)/(1-\eta_k)$ for all $k\ge0$. Then,
\begin{align}\label{ineq:var-recur-em}
&\E_{\xi^{k+1}}[(\|\overline{m}^{k+1} - \nabla f (x^{k+1})\|_{x^{k+1}}^*)^2] \nonumber\\
&\le (1-\alpha_k)(\|\overline{m}^k - \nabla f(x^k)\|_{x^k}^*)^2 + \frac{L_2^2\eta_k^4}{(1-\eta_0)^2{\gamma_k^2}\alpha_k} + \frac{\sigma^2\gamma_k^2}{(1-\eta_0)^2(1-\eta_k/\gamma_k)^2}\qquad \forall k\ge0,
\end{align}
where $\sigma$ and $L_2$ are given in Assumptions \ref{asp:basic}(c) and \ref{asp:2nd-smooth}(b), respectively.
\end{lemma}

We next derive an upper bound for the average expected error of the stationary condition across all iterates generated by SIPM-EM. Its proof is relegated to \cref{subsec:proof-em}.

\begin{theorem}\label{thm:em-stat-ave}
Suppose that Assumptions \ref{asp:basic} and \ref{asp:2nd-smooth} hold. Let $\{(x^k,\lambda^k)\}_{k\ge0}$ be generated by \cref{alg:unf-sipm} with $\{\overline{m}^k\}_{k\ge0}$ constructed as in \cref{sipm-em-over-mk} and input parameters $\{(\eta_k,\gamma_k,\mu_k)\}_{k\ge0}$. Assume that $\{\eta_k\}_{k\ge0}$ is nonincreasing, {$\{\gamma_k\}_{k\ge0}\subset(0,1]$,} and that $\eta_k/\gamma_k\le s_{\eta}$ holds for all $k\ge0$, where $s_{\eta}$ is an input of \cref{alg:unf-sipm}. Define $\alpha_k=1-(1-\gamma_k)/(1-\eta_k)$ for all $k\ge0$. Let $\{p_k\}_{k\ge0}$ be a nondecreasing sequence satisfying $(1-\alpha_k)p_{k+1}\le (1-\alpha_k/2)p_k$ for all $k\ge0$. Then, for all $K\ge1$,
\begin{align}\label{ineq:ave-stat-em}
&\sum_{k=0}^{K-1}\E[\|\nabla \phi_{\mu_k}(x^k) + A^T {\lambda}^k\|^*_{x^k}] \nonumber\\
&\le \frac{\Delta(x^0) + p_0\sigma^2}{\eta_{K-1}} + \frac{1}{\eta_{K-1}}\sum_{k=0}^{K-1}\left(\frac{L_\phi}{2}\eta_k^2 + \frac{{8}\eta_k^2}{p_k\alpha_k}  + \frac{L_2^2\eta_k^4p_{k+1}}{(1-\eta_0)^2{\gamma_k^2}\alpha_k} + \frac{\sigma^2\gamma_k^2p_{k+1}}{(1-\eta_0)^2(1-\eta_k/\gamma_k)^2}\right),
\end{align}
where $\Delta(\cdot)$ and $L_\phi$ are defined in \cref{lwbd-fb} and \cref{def:Lphi-1}, respectively, and $\sigma$ and $L_2$ are given in Assumptions~\ref{asp:basic}(c) and \ref{asp:2nd-smooth}(b), respectively.
\end{theorem}

\subsubsection{Hyperparameters and iteration complexity}

In this subsection, we establish the iteration complexity of SIPM-EM with its input parameters $\{(\eta_k,\gamma_k,\mu_k)\}_{k\ge0}$ specified as:
\begin{align}
{\eta_k= \frac{5s_\eta}{7(k+1)^{5/7}},\quad \gamma_{k} = \frac{1}{(k+1)^{4/7}},\quad \mu_{k} = \max\Big\{\frac{1}{(k+1)^{2/7}}, \frac{\epsilon}{1+\sqrt{\vartheta}}\Big\} \qquad \forall k\ge0,}\label{etak-Bk-em}
\end{align}
where $s_\eta\in(0,1)$ is a user-defined input of \cref{alg:unf-sipm}, {and $\epsilon\in(0,1)$ denotes the tolerance.} It then follows that $\{\eta_k\}_{k\ge0}\subset(0,s_\eta]$ and $\{\mu_k\}_{k\ge0}\subset(0,1]$, with both sequences nonincreasing. 
We also define 
\begin{equation}\label{def:ap-em}
p_k = (k+{1})^{1/7}\qquad \forall k\ge0.
\end{equation}

The next lemma provides some useful properties of the sequences $\{\alpha_{k}\}_{k\ge0}$ and $\{p_{k}\}_{k\ge0}$ defined in \cref{thm:em-stat-ave} and \cref{def:ap-em}, respectively. 
These properties will be used to establish the {iteration complexity} of SIPM-EM, and the proof is deferred to \cref{subsec:proof-em}.

\begin{lemma}\label{lem:tech-ap-use}
Let $\{(\eta_k,\gamma_k)\}_{k\ge0}$ be defined in \cref{etak-Bk-em}, and $\{\alpha_{k}\}_{k\ge0}$ and $\{p_{k}\}_{k\ge0}$ be defined in \cref{thm:em-stat-ave} and \cref{def:ap-em}, respectively. Then, $\alpha_k \ge {(1-5s_\eta/7)/(k+1)^{4/7}}$ and $(1-\alpha_k)p_{k+1} \le (1-\alpha_k/2)p_k$ hold for all $k\ge0$.
\end{lemma}

The next theorem establishes the {iteration complexity} of SIPM-EM with its inputs specified in \cref{etak-Bk-em}. Its proof is relegated to \cref{subsec:proof-em}.

\begin{theorem}\label{cor:order-em}
Suppose that Assumptions \ref{asp:basic} and \ref{asp:2nd-smooth} hold. Consider \cref{alg:unf-sipm} with $\{\overline{m}^k\}_{k\ge0}$ constructed as in \cref{sipm-em-over-mk} and $\{(\eta_k,\gamma_k,\mu_k)\}_{k\ge0}$ specified in \cref{etak-Bk-em}. Let {$\kappa(K)$} be uniformly drawn from {$\{\lfloor K/2\rfloor,\ldots,K-1\}$}, and define
\begin{align}\label{def:kem}  
{
M_{\rmem}\doteq \frac{14}{5}\bigg(\frac{\Delta(x^0) + \sigma^2}{s_\eta} + \frac{40s_\eta L_\phi}{49} + 2\Big(\frac{200s_\eta}{7(7-5s_\eta)} +\frac{1250 L_2^2s_\eta^3}{7(7-5s_\eta)^3} + \frac{2\sigma^2}{s_\eta(1-5s_\eta/7)^4}\Big)\bigg),}
\end{align}
where $\Delta(\cdot)$ and $L_\phi$ are defined in \cref{lwbd-fb} and \cref{def:Lphi-1}, respectively, $\sigma$ and $L_2$ are given in Assumptions~\ref{asp:basic}(c) and \ref{asp:2nd-smooth}(b), respectively, and $s_\eta$ is an input of \cref{alg:unf-sipm}. Then, 
{
\begin{align}
&\E[\|\nabla \phi_{\mu_{\kappa(K)}}(x^{\kappa(K)}) + A^T\lambda^{\kappa(K)}\|_{x^{\kappa(K)}}^*] \le \mu_{\kappa(K)}\ \text{ with }\ \mu_{\kappa(K)}\le\epsilon/(1+\sqrt{\vartheta})\nonumber\\
&\forall K\ge\max\Big\{2\Big(\frac{1+\sqrt{\vartheta}}{\epsilon}\Big)^{7/2},\Big(\frac{7M_{\mathrm{em}}(1+\sqrt{\vartheta})}{\epsilon}\ln\Big(\frac{7M_{\mathrm{em}}(1+\sqrt{\vartheta})}{\epsilon}\Big)\Big)^{7/2},3\Big\} . \label{complexity-em}
\end{align}
}
\end{theorem}

\vspace{-1em}

{
\begin{remark}\label{rmk:em-rate}
From Theorem \ref{cor:order-em}, we observe that SIPM-EM returns an $\epsilon$-SSP within $\widetilde{\mathcal{O}}(\epsilon^{-7/2})$ iterations. This iteration complexity matches that of stochastic unconstrained optimization with Lipschitz continuous Hessian \cite{cutkosky2020momentum}, up to a polylogarithmic factor.
\end{remark}
}


\subsection{An SIPM with recursive momentum}\label{subsec:sipm-rm}

This subsection incorporates recursive momentum into \cref{alg:unf-sipm}. We make the following additional assumptions throughout this subsection.
{
\begin{assumption}\label{asp:ave-smth}
We have access to a stochastic gradient estimator $G:\Omega^\circ\times\Xi\to\R^n$ such that \cref{asp:unbias-boundvar} holds for some $\sigma>0$, and for any $s_\eta\in(0,1)$,
\begin{align}
\E_\xi[(\|G(y,\xi) - G(x,\xi)\|_x^*)^2] \le L^2 \|y-x\|_x^2\qquad \forall x,y\in\Omega^\circ\text{ with } \|y-x\|_x\le s_{\eta}\label{ineq:Lip-average}
\end{align}
holds for some $L>0$.
\end{assumption}
}

{
\begin{remark}\label{rmk:asp-ave-smth}
Assumption \ref{asp:ave-smth} can be seen as a local-norm variant of the average smoothness condition \cite{cutkosky2019momentum,fang2018spider,li2021page}. It is generally stronger than Assumption \ref{asp:basic}(b), as implied by the following:
\begin{align*}
(\|\nabla f(y)-\nabla f(x)\|_x^*)^2 = (\|\E_\xi[G(y,\xi) - G(x,\xi)]\|_x^*)^2  \le\E_\xi[(\|G(y,\xi) - G(x,\xi)\|_x^*)^2],
\end{align*}
where the equality follows from the unbiasedness of $G(\cdot,\xi)$, and the inequality follows from Jensen's inequality and the convexity of $\|\cdot\|_x^*$.
\end{remark}}

We next describe a variant of \cref{alg:unf-sipm} with recursive momentum, in which $\{\overline{m}^k\}_{k\ge0}$ is constructed based on recursive momentum \cite{cutkosky2019momentum}:
\begin{subequations}\label{sipm-rm-mk-up}
\begin{align}
&\gamma_{-1}=1,\quad x^{-1}=x^0,\quad \overline{m}^{-1}=0,\\
&{\overline{m}^k = G(x^k,\xi^k) + (1-\gamma_{k-1}) (\overline{m}^{k-1} - G(x^{k-1},\xi^k))\qquad \forall k\ge0,}
\end{align}    
\end{subequations}
where $G(\cdot,\cdot)$ satisfies Assumption \ref{asp:ave-smth}, and $\{\gamma_k\}_{k\ge0}{\subset(0,1]}$ is the sequence of momentum parameters. For ease of reference, we refer to this variant of \cref{alg:unf-sipm} as SIPM with recursive momentum (SIPM-RM).

The next lemma provides the recurrence for the estimation error of the gradient estimators based on the recursive momentum described in \cref{sipm-rm-mk-up}. Its proof is deferred to \cref{subsec:proof-rm}.

\begin{lemma}\label{lem:rm-var-err}
Suppose that Assumptions \ref{asp:basic} and \ref{asp:ave-smth} hold. Let $\{x^k\}_{k\ge0}$ be all iterates generated by \cref{alg:unf-sipm} with $\{\overline{m}^k\}_{k\ge0}$ updated according to \cref{sipm-rm-mk-up} and input parameters {$\{(\eta_k,\gamma_k)\}_{k\ge0}$}. {Assume that $\{\eta_k\}_{k\ge0}$ is nonincreasing. For all $k\ge0$, assume that $\gamma_k>\eta_k$ and define $\alpha_k=1-(1-\gamma_k)/(1-\eta_k)$.} Then,
{
\begin{align}
&\E_{\xi^{k+1}}[(\|\overline{m}^{k+1} - \nabla f(x^{k+1})\|_{x^{k+1}}^*)^2]\nonumber\\ 
&\le (1-\alpha_k)(\|\overline{m}^k-\nabla f(x^k)\|^*_{x^{k}})^2 + \frac{3(L_1^2+L^2)\eta_k^2 + 3\sigma^2\gamma_k^2}{(1-\eta_0)^2}\qquad \forall k\ge0,\label{ineq:var-recur-rm}
\end{align}
where $L_1$ and $\sigma$ are given in Assumption \ref{asp:basic}, and $L$ is given in Assumption \ref{asp:ave-smth}.
}
\end{lemma}

We next provide an upper bound for the average expected error of the stationary condition among all iterates generated by SIPM-RM. Its proof is relegated to \cref{subsec:proof-rm}.

\begin{theorem}\label{thm:converge-rate-ipmrm}
Suppose that Assumptions~\ref{asp:basic} and \ref{asp:ave-smth} hold. Let $\{(x^k,\lambda^k)\}_{k\ge0}$ be generated by \cref{alg:unf-sipm} with $\{\overline{m}^k\}_{k\ge0}$ updated according to \cref{sipm-rm-mk-up} and input parameters $\{(\eta_k,\gamma_k,\mu_k)\}_{k\ge0}$. Assume that {$\{\eta_k\}_{k\ge0}$ is nonincreasing and also that $\gamma_k> \eta_k$ for all $k\ge0$. Define $\alpha_k=1-(1-\gamma_k)/(1-\eta_k)$ for all $k\ge0$. Let $\{p_k\}_{k\ge0}$ be a nondecreasing sequence satisfying $(1-\alpha_k)p_{k+1}\le (1-\alpha_k/2)p_k$ for all $k\ge0$.} Then, for all $K\ge1$,
{
\begin{align}
&\sum_{k=0}^{K-1}\E[\|\nabla \phi_{\mu_k}(x^k) +A^T {\lambda}^k\|^*_{x^k}]\nonumber\\
&\le \frac{\Delta(x^0) + p_0\sigma^2}{\eta_{K-1}} + \frac{1}{\eta_{K-1}}\sum_{k=0}^{K-1}\Big(\frac{L_{\phi}}{2}\eta_k^2 + \frac{8\eta_k^2}{p_k\alpha_k}+ \frac{3(L_1^2+L^2)\eta_k^2p_{k+1} + 3\sigma^2\gamma_k^2p_{k+1}}{(1-\eta_0)^2}\Big),\label{ineq:ave-stat-rm}
\end{align}
where $\Delta(\cdot)$ and $L_\phi$ are defined in \cref{lwbd-fb} and \cref{def:Lphi-1}, respectively, $L_1$ and $\sigma$ are given in Assumption \ref{asp:basic}, and $L$ is given in Assumption \ref{asp:ave-smth}.}
\end{theorem}

\subsubsection{Hyperparameters and {iteration complexity}}

{In this subsection, we establish the iteration complexity of SIPM-RM with its input parameters $\{(\eta_k,\gamma_k,\mu_k)\}_{k\ge0}$ specified as:
\begin{align}
\eta_{k}= \frac{s_{\eta}}{3(k+1)^{2/3}},\quad \gamma_{k} = \frac{1}{(k+1)^{2/3}},\quad \mu_{k} =\max\Big\{\frac{1}{(k+1)^{1/3}},\frac{\epsilon}{1+\sqrt{\vartheta}}\Big\} \qquad \forall k\ge0,\label{etak-Bk-rm}
\end{align}
where $s_\eta\in(0,1)$ is a user-defined input of Algorithm \ref{alg:unf-sipm}, {and $\epsilon\in(0,1)$ denotes the tolerance.} It then follows that $\{\eta_k\}_{k\ge0}\subset(0,s_{\eta}]$ and $\{\mu_k\}_{k\ge0}\subset(0,1]$, with both sequences nonincreasing. We also define 
\begin{align}\label{def:ap-rm}
p_k=(k+1)^{1/3}\qquad k\ge0.
\end{align}

The following lemma provides some useful properties of the sequences $\{\alpha_k\}_{k\ge0}$ and $\{p_k\}_{k\ge0}$ defined in Theorem \ref{thm:converge-rate-ipmrm} and \cref{def:ap-rm}, respectively. Its proof is deferred to Section \ref{subsec:proof-rm}.

\begin{lemma}\label{lem:tech-eta-gma-rm}
Let $\{(\eta_k,\gamma_k)\}_{k\ge0}$ be defined in \cref{etak-Bk-rm}, and $\{\alpha_k\}_{k\ge0}$ and $\{p_k\}_{k\ge0}$ be defined in Theorem \ref{thm:converge-rate-ipmrm} and \cref{def:ap-rm}, respectively. Then, $\alpha_k\ge(1-s_\eta/3)/(k+1)^{2/3}$ and $(1-\alpha_k)p_{k+1}\le(1-\alpha_k/2)p_k$ hold for all $k\ge0$.
\end{lemma}
}

The following theorem presents the {iteration complexity} of SIPM-RM with its inputs specified in \cref{etak-Bk-rm}. Its proof is relegated to \cref{subsec:proof-rm}.

\begin{theorem}\label{cor:order-rm}
Suppose that Assumptions \ref{asp:basic} and \ref{asp:ave-smth} hold. Consider \cref{alg:unf-sipm} with $\{\overline{m}^k\}_{k\ge0}$ updated using \cref{sipm-rm-mk-up} and input parameters $\{(\eta_k,\gamma_k,\mu_k)\}_{k\ge0}$ specified in \cref{etak-Bk-rm}. Let {$\kappa(K)$} be uniformly drawn from {$\{\lfloor K/2\rfloor,\ldots,K-1\}$}, and define
{
\begin{align}
M_{\rmrm}\doteq 6\bigg(\frac{\Delta(x^0) + \sigma^2}{s_{\eta}}  + \frac{2s_\eta L_\phi}{9} + 4\Big(\frac{4s_\eta}{3(3-s_\eta)} + \frac{3(L_1^2+L^2)s_\eta}{(3-s_\eta)^2} + \frac{3\sigma^2}{s_\eta(1-s_\eta/3)^2}\Big) \bigg),  \label{def:krm} 
\end{align}
where $\Delta(\cdot)$ and $L_\phi$ are defined in \cref{lwbd-fb} and \cref{def:Lphi-1}, respectively, $L_1$ and $\sigma$ are given in Assumption \ref{asp:basic}, and $L$ is given in Assumption \ref{asp:ave-smth}. Then, 
\begin{align}
&\E[\|\nabla \phi_{\mu_{\kappa(K)}}(x^{\kappa(K)}) + A^T\lambda^{\kappa(K)}\|_{x^{\kappa(K)}}^*] \le \mu_{\kappa(K)}\ \text{ with }\ \mu_{\kappa(K)}\le\epsilon/(1+\sqrt{\vartheta})\nonumber\\
&\forall K\ge\max\Big\{2\Big(\frac{1+\sqrt{\vartheta}}{\epsilon}\Big)^3,\Big(\frac{6M_{\mathrm{rm}}(1+\sqrt{\vartheta})}{\epsilon}\ln\Big(\frac{6M_{\mathrm{rm}}(1+\sqrt{\vartheta})}{\epsilon}\Big)\Big)^3,3\Big\}. \label{complexity-rm}
\end{align}
}
\end{theorem}

\vspace{-1em}

\begin{remark}\label{rmk:rm-rate}
{From Theorem \ref{cor:order-rm}, we see that SIPM-RM returns an $\epsilon$-SSP within $\widetilde{\mathcal{O}}(\epsilon^{-3})$ iterations. This iteration complexity bound matches the best-known results for stochastic unconstrained optimization under average smoothness condition \cite{cutkosky2019momentum,fang2018spider,li2021page}, up to a polylogarithmic factor.}
\end{remark}

\section{Numerical experiments}\label{sec:ne}
We now conduct numerical experiments to evaluate the performance of our proposed SIPMs. For SIPM-ME, we consider two versions: SIPM-ME$^+$ with increasing batch sizes and SIPM-ME$^1$ with a fixed batch size. We compare our SIPMs against a deterministic variant of \cref{alg:unf-sipm} with full-batch gradients (IPM-FG) and other popular methods on robust linear regression (\cref{subsec:rlr}), multi-task relationship learning (\cref{subsec:mrl}), and clustering data streams (\cref{subsec:cds}).

We evaluate the approximate solutions found by our SIPMs using two measures: relative objective value: $f(x^k)/f(x^0)$; and  relative estimated stationary error: $\|m^k + A^T\lambda^k\|_{x^k}^*/\|m^0 + A^T\lambda^0\|_{x^0}^*$. All experiments are carried out using Matlab 2024b on a standard PC with 3.20 GHz AMD R7 5800H microprocessor with 16GB of memory. The code to reproduce our numerical results is available at \url{https://github.com/ChuanH6/SIPM}.

\begin{table}[H]
\centering
\caption{\small The relative objective value and relative estimated stationary error for all methods applied to solve problem \cref{rbst-socp}.}
\label{table:rbst}
\resizebox{\textwidth}{!}{
\begin{tabular}{l|cc|cc|cc|cc}
\hline
\multirow{3}{*}{} &   \multicolumn{4}{c|}{wine-quality} &  \multicolumn{4}{c}{energy-efficiency} 
\\
&   \multicolumn{2}{c}{$(d,p)=(10,2000)$} &  \multicolumn{2}{c|}{$(d,p)=(20,4000)$} &  \multicolumn{2}{c}{$(d,p)=(10,5000)$} &  \multicolumn{2}{c}{$(d,p)=(20,10000)$}
\\ 
& objective & \multicolumn{1}{c}{stationary} & objective & stationary & objective & \multicolumn{1}{c}{stationary} & objective & stationary  \\ \hline
SIPM-ME$^1$  & 0.2860  &1.550e-2 & 0.3170 & 2.500e-2 & 0.8243 & 3.303e-3 & 0.8406  & 6.500e-3 \\
SIPM-ME$^+$  & 0.2376  &3.611e-3 & 0.2220 & 3.813e-3 & 0.8128 & 5.982e-5 & 0.8126  & 5.493e-5\\
SIPM-PM  & 0.2307  &2.002e-3  & 0.2372 & 3.703e-3 & 0.8131 & 8.123e-5 &  0.8133 & 8.051e-5 \\ 
SIPM-EM  & 0.2307 &2.002e-3 & 0.2322 & 3.702e-3 & 0.8128 & 6.954e-6&  0.8127 & 7.527e-6\\
SIPM-RM  & 0.2306  &1.903e-3  & 0.2302 & 4.303e-3 & 0.8128 & 2.167e-6 &  0.8127  & 2.444e-6\\
IPM-FG  & 0.2359  &4.652e-3  & 0.2350 & 4.503e-3 & 0.8128 & 2.974e-6 &  0.8127  & 1.453e-5\\
SALM-RM  & 0.2583  & -- & 0.2570& -- & 0.8532 & -- &  0.8241  & -- \\ \hline
\end{tabular}}
\end{table}
\begin{figure}[ht]
\centering
\begin{minipage}[b]{0.24\linewidth}
\centering
\includegraphics[width=\linewidth]{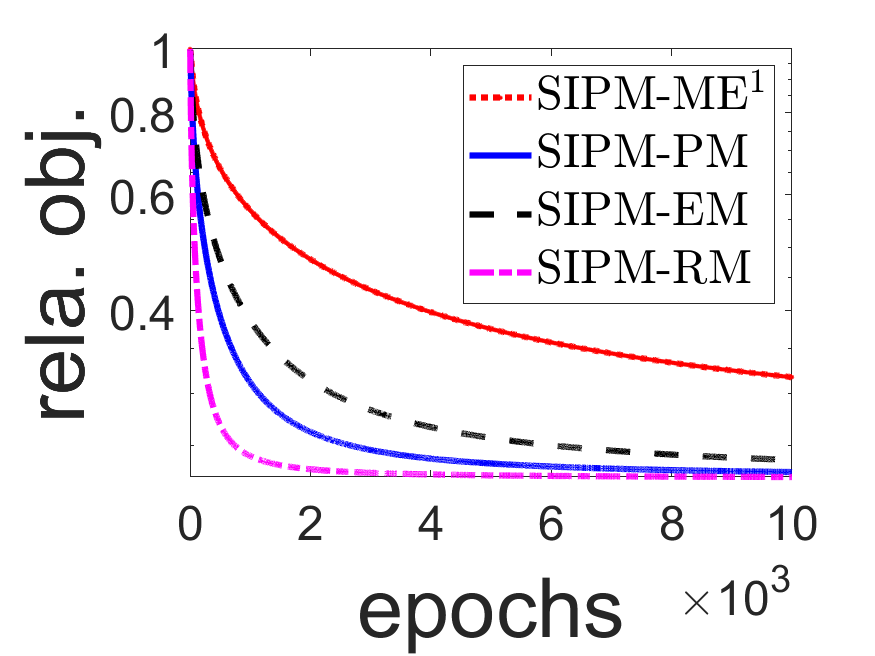}
\end{minipage}
\hfill
\begin{minipage}[b]{0.24\linewidth}
\centering
 \hfill\includegraphics[width=\linewidth]{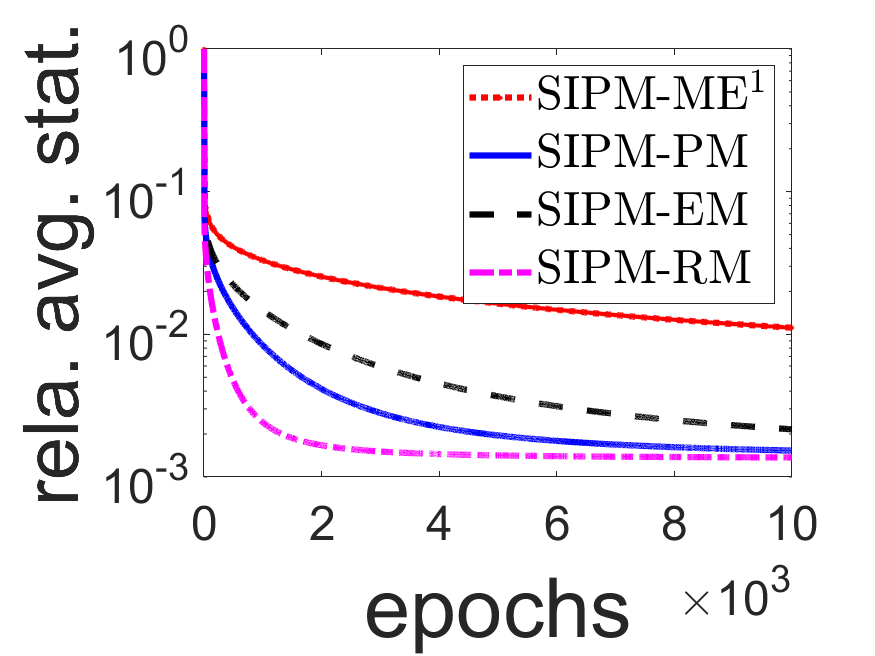}
\end{minipage}
\begin{minipage}[b]{0.24\linewidth}
\centering
\includegraphics[width=\linewidth]{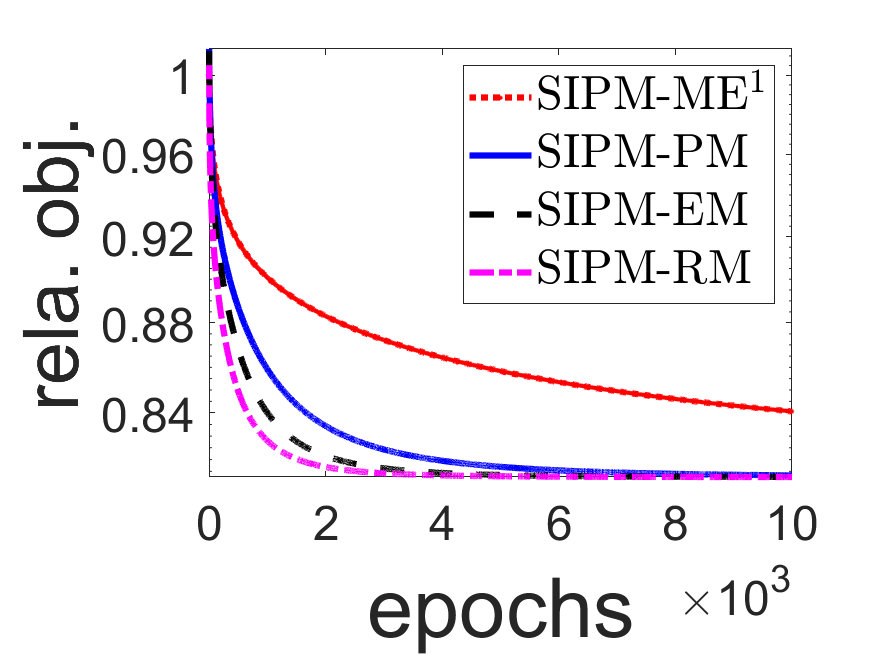}
\end{minipage}
\hfill
\begin{minipage}[b]{0.24\linewidth}
\centering
\includegraphics[width=\linewidth]{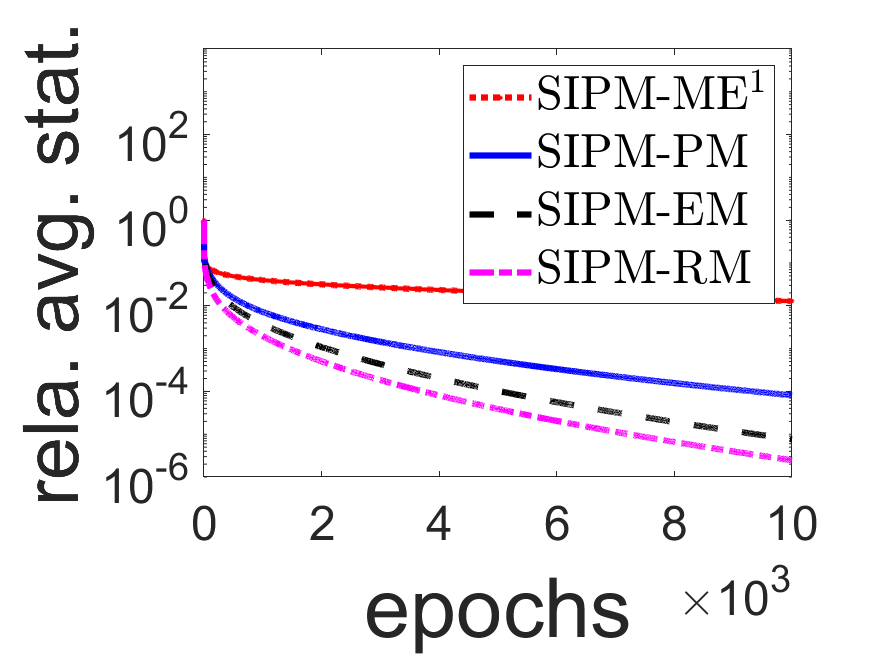}
\end{minipage}
\caption{\small Convergence behavior of the relative objective value and average relative stationary error for each epoch. The first two and last two plots correspond to the `wine-quality' and `energy-efficiency' datasets, respectively.}
\label{fig:socp}
\end{figure}

\subsection{Robust linear regression with chance constraints}\label{subsec:rlr}
In this subsection, we consider the second-order cone constrained regression problem in \cref{rbst-socp}, which is derived from robust regression with chance constraints. We let $\phi(t)=t^2/(1+t^2)$ as a robust modification of the quadratic loss \cite{carmon2017convex}. We consider two {typical regression} datasets, namely, `wine-quality' and `energy-efficiency', from the UCI repository.\footnote{see \url{archive.ics.uci.edu/datasets}} 
Following \cite{shivaswamy2006second}, we simulate missing values by randomly deleting 25\% of the features in 50\% of the training samples, and filling the missing values using linear regression on the remaining features.

We apply our SIPMs, IPM-FG, and a stochastic augmented Lagrangian method with recursive momentum in \cite{alacaoglu2024complexity,lu2024variance} (SALM-RM) to solve \cref{rbst-socp}. {For SIPMs, we set the barrier function as $B(u,t)=-\ln(t^2-\|u\|^2)$, associated with the second-order cone $\mathbb{Q}^{d+1}\doteq\{(u,t)\in\R^d\times\R_+:\|u\|\le t\}$.} For SIPM-ME$^1$, SIPM-PM, SIPM-EM, SIPM-RM, and SALM-RM, we set the maximum number of epochs as $10,000$, {and the batch size as 200 and 500 for the `wine-quality' and `energy-efficiency' datasets, respectively. For SIPM-ME$^+$, we initialize the batch size as 1 and increase it by 1 per iteration.} For a fair comparison, we set the maximum number of iterations for SIPM-ME$^+$ and IPM-FG such that the total number of data points used during training equals that of the other methods in $10,000$ epochs. For all methods, we set the initial point $(w^0,v^0,\theta^0)$ to $(0,1,1)$. {We set the other hyperparameters---including step sizes, momentum parameters, and barrier parameters---as diminishing sequences of the form $\{(k+1)^{-\alpha}\}_{k\ge0}$, with the exponent $\alpha$ tuned via grid search to best suit each method in terms of computational performance.}

Comparing the relative objective values and the relative estimated stationary errors in \cref{table:rbst}, we see that our SIPM-ME$^+$, SIPM-PM, SIPM-EM, and SIPM-RM yield solutions of similar quality to IPM-FG. In addition, SIPM-ME$^1$ converges slowly and returns suboptimal solutions, which corroborates the theoretical results that incorporating momentum facilitates gradient estimation and improves solution quality. We also observe that the solution quality of SALM-RM is generally worse than that of our SIPMs, except SIPM-ME$^1$. From \cref{fig:socp}, we observe that SIPM-ME$^1$ converges much slower than the other three variants, and SIPM-RM is slightly faster than SIPM-PM and SIPM-EM, which corroborates our established {iteration complexity} for these methods.

\subsection{Multi-task relationship learning}\label{subsec:mrl}
In this subsection, we consider the problem of multi-task relationship learning in \cref{multi-task}, where $a_{ij}=(p_{ij},q_{ij})\in\R^d\times\R$ for all $1\le i\le p$ and $1\le j\le m_i$, $\ell(w_i,a_{ij})=\phi(w_i^Tp_{ij} - q_{ij})$, and $\phi(t)=t^2/(1+t^2)$. We consider five {typical regression} datasets from the UCI repository: `wine-quality-red', `wine-quality-white', `energy-efficiency', `air-quality', and `abalone'. 
In our multi-task learning, we consider problems with five tasks and ten tasks, where, in the former case, we treat a subset of each of the five datasets as a separate task, and in the latter case, we take two subsets from each dataset, treating each subset as a separate task.

We apply our SIPMs, IPM-FG, and the alternating minimization method in \cite{argyriou2008convex} (AM) to solve \cref{multi-task}. {For SIPMs, we choose the barrier function to be $B(\Sigma)=-\ln(\mathrm{det}(\Sigma))$, associated with the positive semidefinite cone $\mathbb{S}_+^{p}\doteq\{\Sigma\succeq0\}$.} For SIPM-ME$^1$, SIPM-PM, SIPM-EM, and SIPM-RM, we {choose the batch size as $200$ and $500$,} and set the maximum number of epochs as {$250$. For SIPM-ME$^+$, we initialize the batch size as 10 and increase it by 10 per iteration.}
For a fair comparison, we set the maximum number of iterations for SIPM-ME$^+$, IPM-FG and AM such that the total number of data points used during training equals that {of the other methods in $250$ epochs}. For all methods, we set the initial point $W^0$ as the all-zero matrix and $\Sigma^0$ as a diagonal matrix with diagonal elements equal to $1/p$. {We set the other hyperparameters---including step sizes, momentum parameters, and barrier parameters---according to the strategy described in Section \ref{subsec:rlr}.}

\begin{figure}[ht]
\begin{minipage}[t]{\linewidth}
\raggedright 
\includegraphics[width=0.88\linewidth]{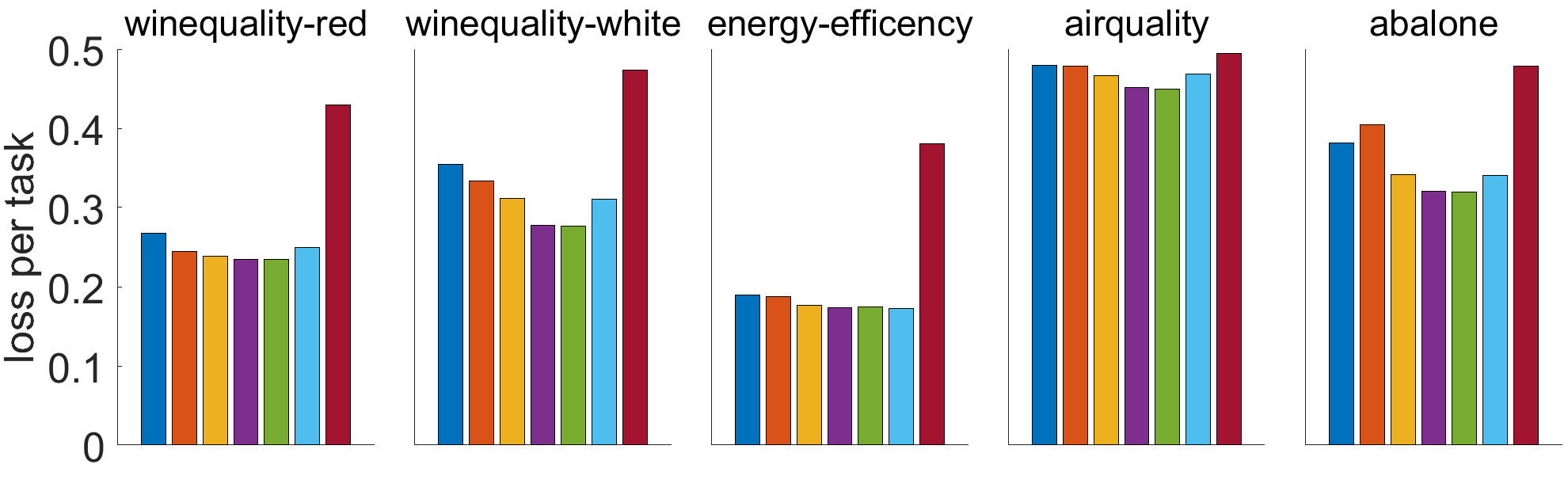}
\par\vspace{0.2em}
\end{minipage}

\begin{minipage}[t]{\linewidth}
\centering
\includegraphics[width=1\linewidth]{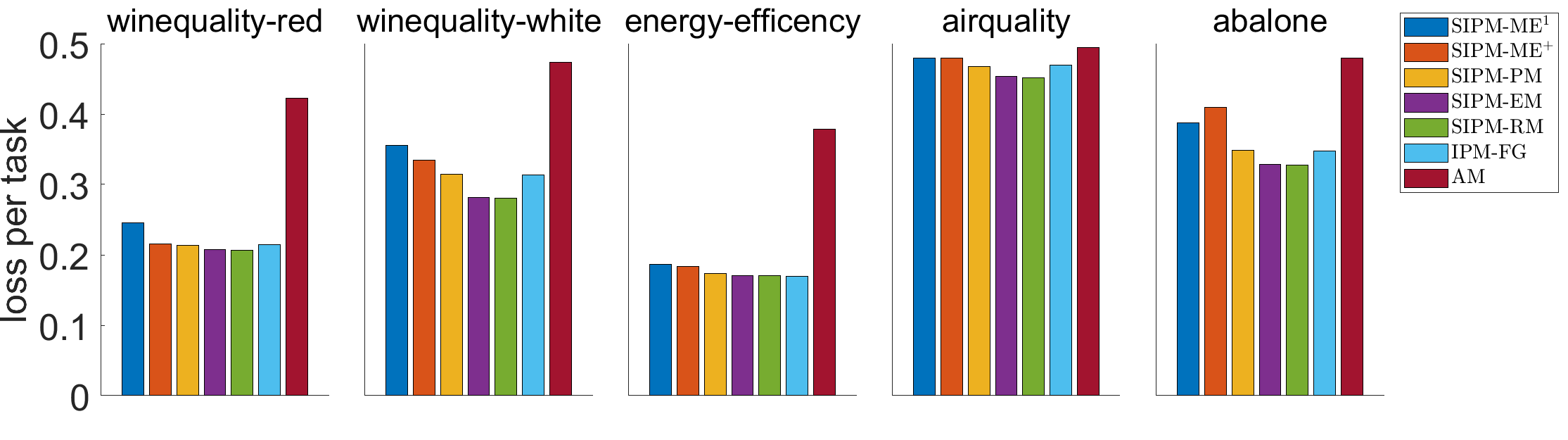}
\end{minipage}

\caption{\small Loss per task for the training ({top}) and validation ({bottom}).}
\label{fig:multi-gen}
\end{figure}

\begin{table}[h]
\centering
\caption{\small The relative objective value and relative estimated stationary error for all methods applied to solve problem \cref{multi-task}.}
\label{table:mtl}
\resizebox{\textwidth}{!}{
\begin{tabular}{l|cc|cc|cc|cc}
\hline
&   \multicolumn{4}{c|}{five tasks} &  \multicolumn{4}{c}{ten tasks}
\\
&   \multicolumn{2}{c}{$m=200$} &  \multicolumn{2}{c|}{$m=500$} &  \multicolumn{2}{c}{$m=200$} &  \multicolumn{2}{c}{$m=500$}
\\ 
& objective & \multicolumn{1}{c}{stationary} & objective & stationary & objective & \multicolumn{1}{c}{stationary} & objective & stationary  \\ \hline
SIPM-ME$^1$  & 0.3260  &2.985e-2 & 0.3137  &3.002e-2 & 0.3985 & 8.985e-3 & 0.4245  & 3.587e-3 \\ 
SIPM-ME$^+$  & 0.3254  &8.763e-3 & 0.2854  &8.733e-3 & 0.2879 & 4.803e-3 & 0.2987  & 9.564e-3 \\
SIPM-PM  & 0.3215  & 8.654e-3  &0.2868  & 1.383e-2 & 0.2875 & 5.123e-3 &  0.2884 & 1.251e-2 \\ 
SIPM-EM  & 0.3203 & 8.685e-3 & 0.2850 & 8.254e-3 & 0.2865 & 4.954e-3&  0.2764 & 8.527e-3\\ 
SIPM-RM  & 0.3198  & 8.534e-3   & 0.2843  & 8.234e-3 & 0.2765 & 4.547e-3 &  0.2774   &8.436e-3\\ 
IPM-FG  &  0.3243   & 8.778e-3  & 0.2975 & 8.447e-3  &  0.2909 &  4.987e-3 &  0.2894 & 8.554e-3 \\ 
AM  &  0.3535  & -- &  0.3743  & -- & 0.4043  & -- &  0.4253 & --  \\
\hline
\end{tabular}}
\end{table}
\begin{figure}[ht]
\centering
\begin{minipage}[b]{0.24\linewidth}
\centering
\includegraphics[width=\linewidth]{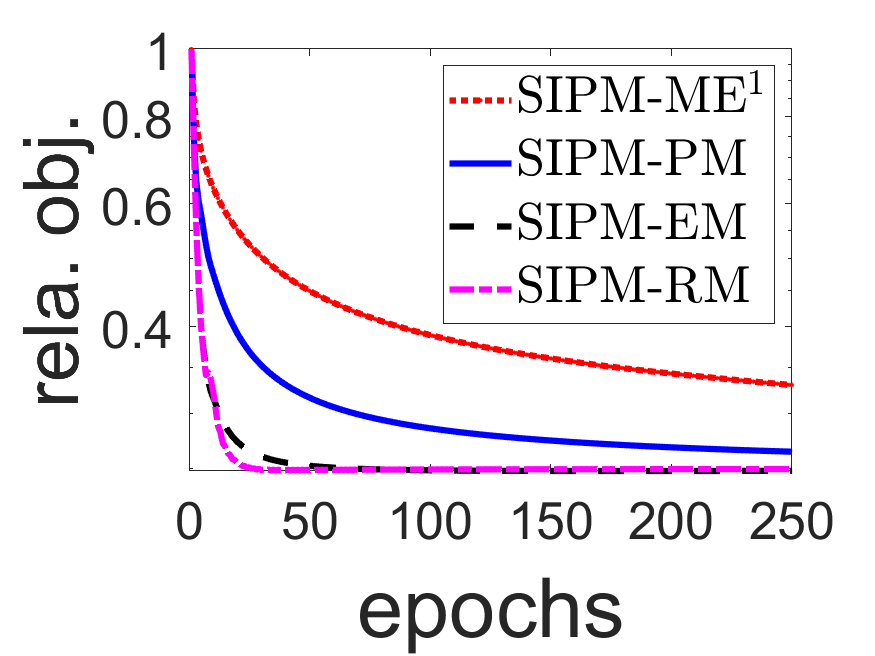}
\end{minipage}
\hfill
\begin{minipage}[b]{0.24\linewidth}
\centering
 \hfill\includegraphics[width=\linewidth]{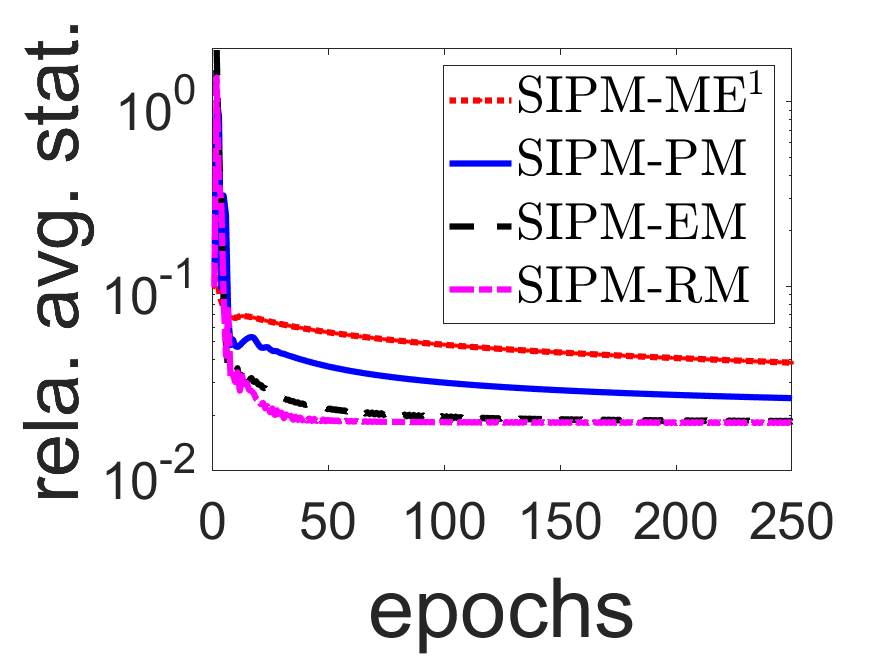}
\end{minipage}
\begin{minipage}[b]{0.24\linewidth}
\centering
\includegraphics[width=\linewidth]{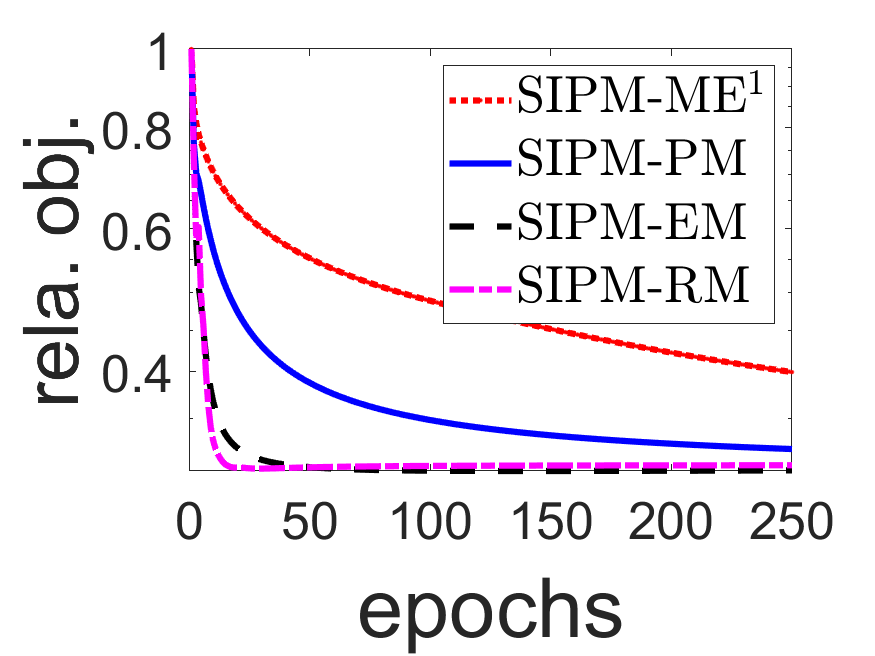}
\end{minipage}
\hfill
\begin{minipage}[b]{0.24\linewidth}
\centering
\includegraphics[width=\linewidth]{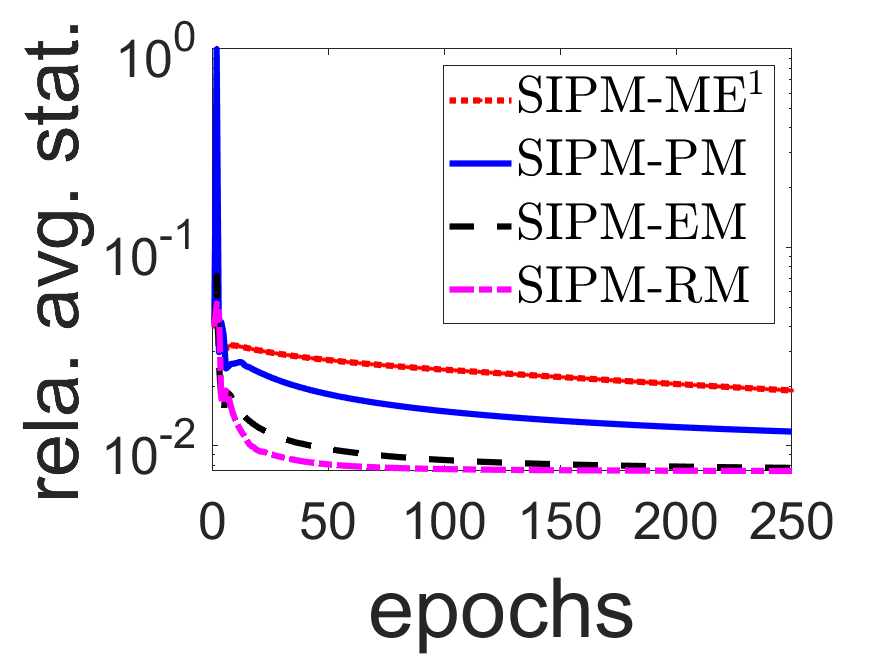}
\end{minipage}
\caption{\small Convergence behavior of the relative objective value and average relative stationary error for each epoch. The first two figures correspond to the problem with five tasks, and the last two correspond to the problem with ten tasks.}
\label{fig:multi-sipms}
\end{figure}

From \cref{fig:multi-gen}, our SIPMs and IPM-FG outperform AM in loss for each task on both training and validation datasets, improving multi-task learning performance. By comparing the relative objective values and the relative estimated stationary errors in \cref{table:mtl}, we observe that our SIPM-ME$^+$, SIPM-PM, SIPM-EM, and SIPM-RM yield {solutions} of similar quality compared to IPM-FG. In addition, SIPM-ME$^1$ converges slowly and returns suboptimal solutions, which corroborates the theoretical results that incorporating momentum aids gradient estimation and improves solution quality. We also observe that the solution quality, in terms of relative objective value, of AM is generally worse than that of our SIPM variants. From \cref{fig:multi-sipms}, we observe that SIPM-ME$^1$ converges more slowly than the other three variants, and SIPM-RM is faster than SIPM-PM and SIPM-EM, which corroborates our established {iteration complexity} for these methods.


\subsection{Clustering data streams}\label{subsec:cds}
In this subsection, we consider the problem of clustering data streams \cite{bidaurrazaga2021k}, which aims at assigning $d$ points to $k$ clusters, with each point having $p$ data observations that arrive continuously. Applying the semidefinite relaxation from \cite{peng2007approximating} to this clustering problem results in:
\begin{align}\label{pb:ceds}
\min_{W\in\R^{d\times d}} \frac{1}{p}\sum_{i=1}^p\langle A_i,W\rangle + \tau\sum_{i=1}^d\ln(\gamma +\lambda_i(W))\ \ \mathrm{s.t.}\ \ W\in\mathbb{S}^{d}_+,\ \ We_d = e_d,\ \ \langle I_d,W\rangle=k,
\end{align}
where $\{A_i\}_{i=1}^p$ are computed from data streams, $\tau\sum_{i=1}^d\ln(\gamma +\lambda_i(W))$ is a nonconvex regularizer that imposes low rankness, with {$\lambda_i(W)$ being the $i$th largest eigenvalue of $W$, and} $\tau$ and $\gamma$ being tuning parameters \cite{lu2014generalized}, and $e_d$ and $I_d$ denote the $d$-dimensional  all-one vector and the $d\times d$ identity matrix, respectively. We consider two {typical classification} datasets, namely, `spam-base' and `cover-type', from the UCI repository. 
To simulate stream scenarios, we apply $(1+\varepsilon)$-drifts adapted from \cite{bidaurrazaga2021k} to each dataset, where $\varepsilon$ is randomly drawn from a standard Gaussian distribution.
\begin{figure}[h]
\centering
\begin{minipage}[b]{0.24\linewidth}
\centering
\includegraphics[width=\linewidth]{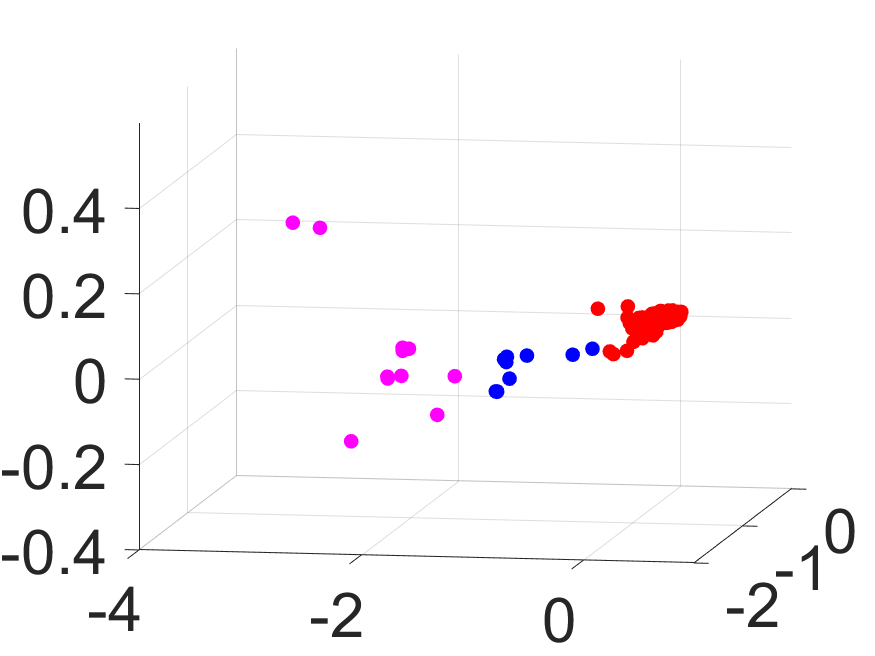}
\end{minipage}
\hfill
\begin{minipage}[b]{0.24\linewidth}
\centering
\includegraphics[width=\linewidth]{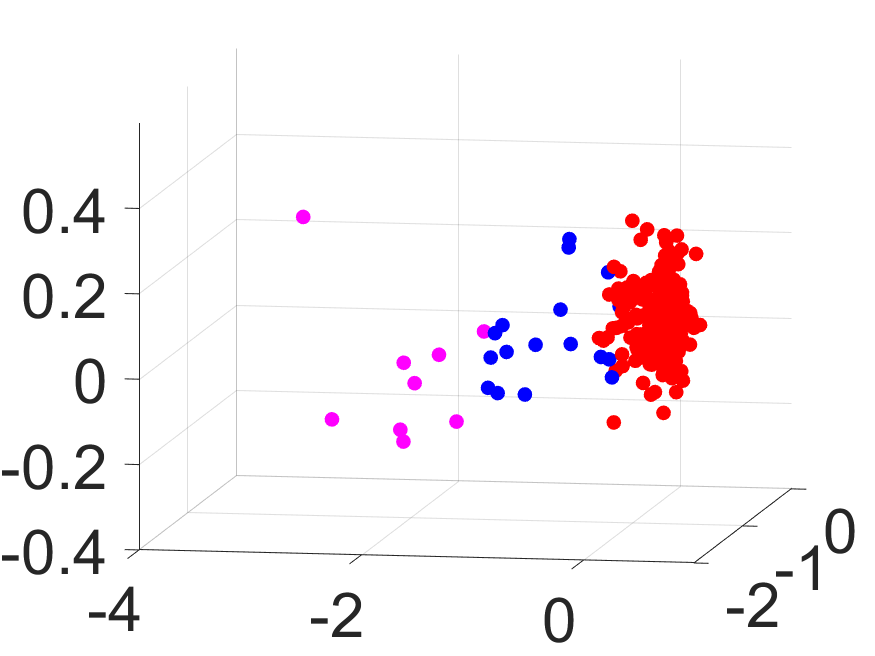}
\end{minipage}
\hfill
\begin{minipage}[b]{0.24\linewidth}
\centering
\includegraphics[width=\linewidth]{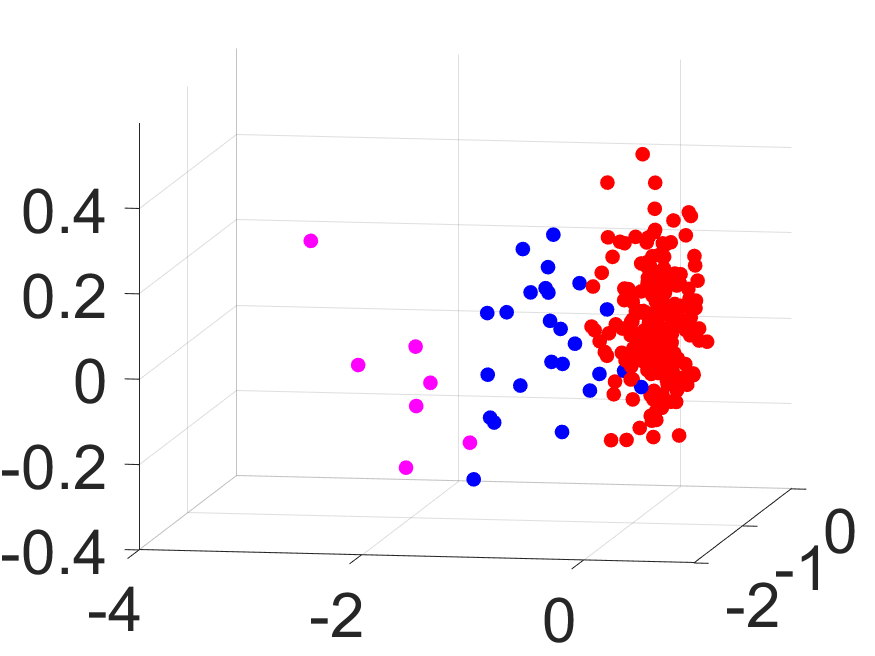}
 \end{minipage}
\hfill
\begin{minipage}[b]{0.24\linewidth}
\centering
\includegraphics[width=\linewidth]{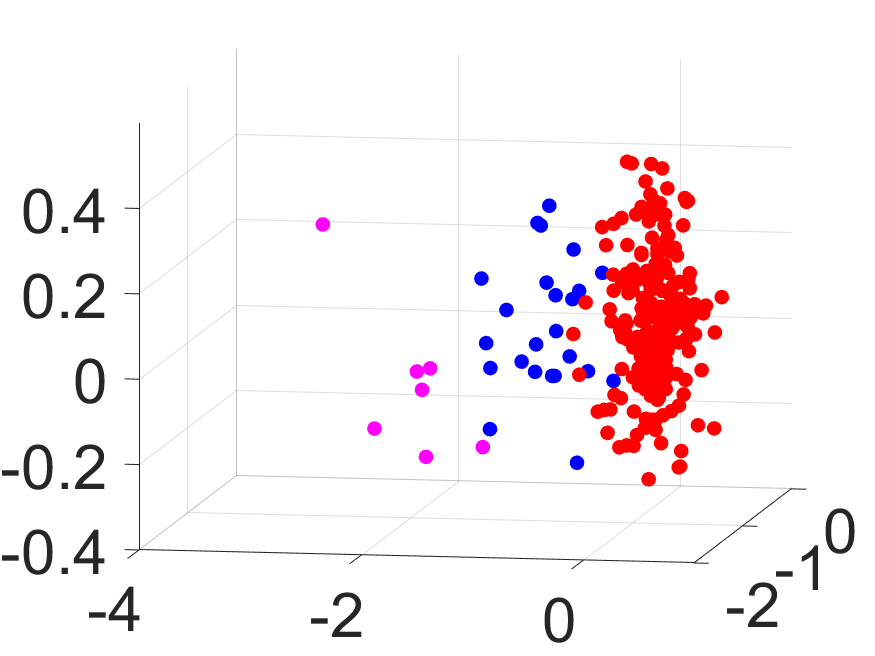}
\end{minipage}\\
\begin{minipage}[b]{0.24\linewidth}
\centering
\includegraphics[width=\linewidth]{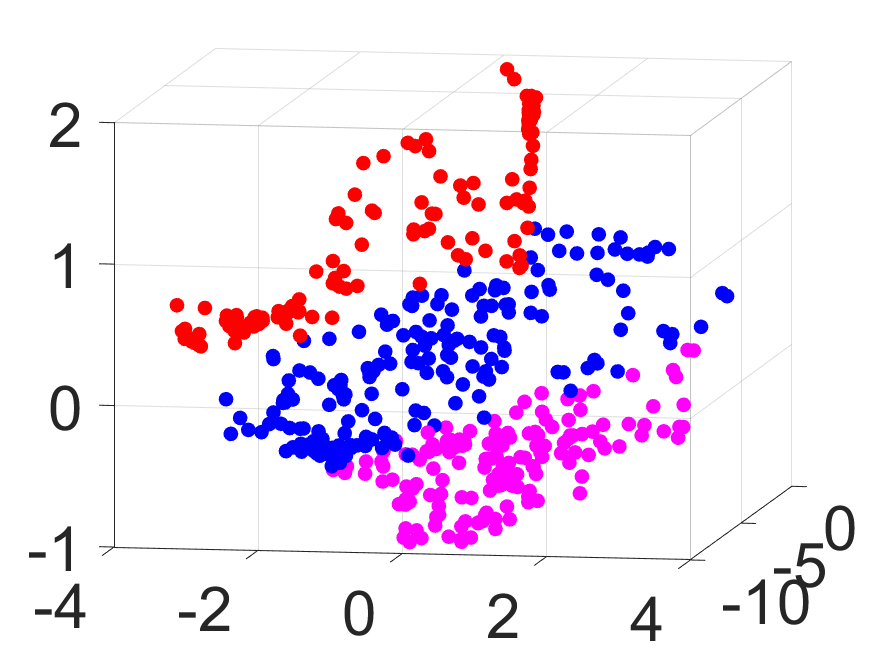}
\end{minipage}
\hfill
\begin{minipage}[b]{0.24\linewidth}
\centering
\includegraphics[width=\linewidth]{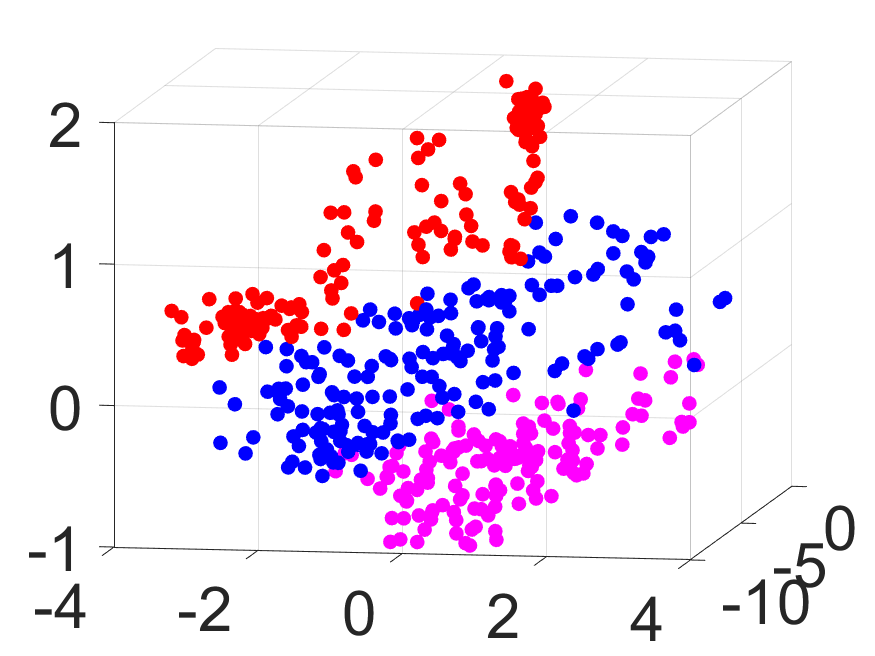}
\end{minipage}
\hfill
\begin{minipage}[b]{0.24\linewidth}
\centering
\includegraphics[width=\linewidth]{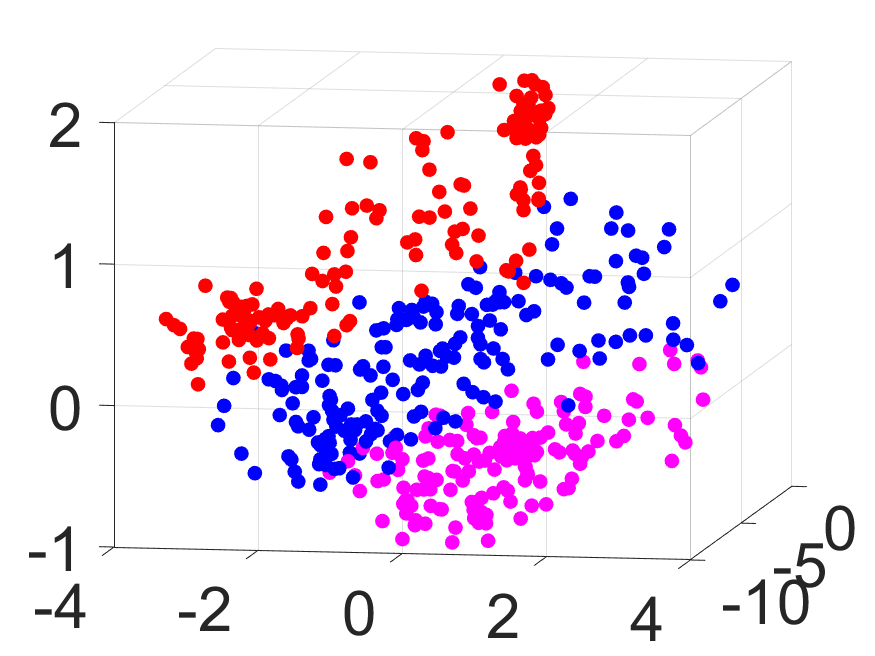}
\end{minipage}
\hfill
\begin{minipage}[b]{0.24\linewidth}
\centering
\includegraphics[width=\linewidth]{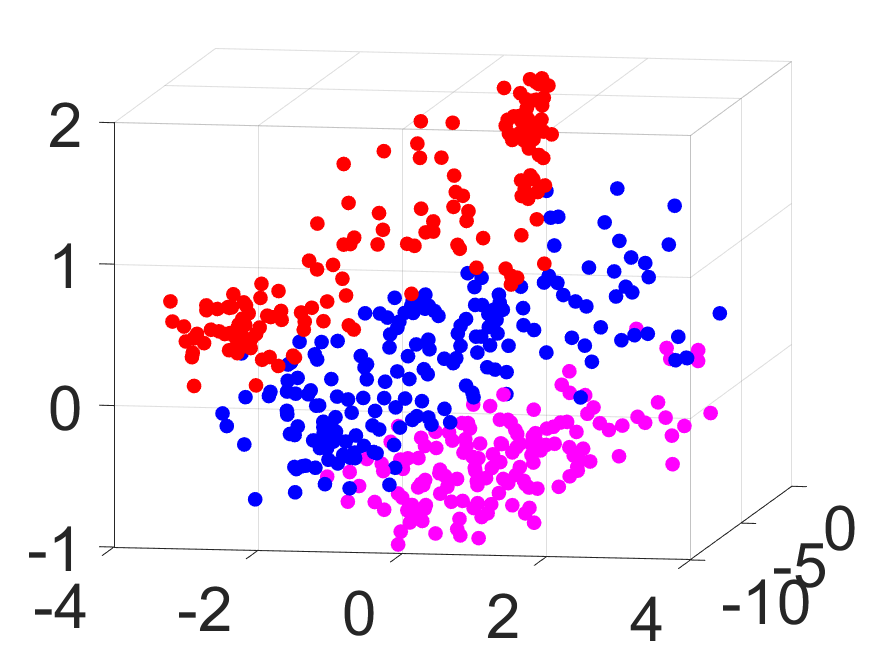}
\end{minipage}
\caption{\small Visualization of the clustering results obtained by solving \cref{pb:ceds} using our SIPMs at the 1st, 333rd, 666th, and 1000th data observations in the stream (from left to right). The first and second rows display the clustering results for the `spam-base' and `cover-type' datasets, respectively.}
\label{fig:sdp}
\end{figure}
\begin{table}[H]
\centering
\caption{\small The relative objective value and relative estimated stationary error for all methods applied to solve problem \cref{pb:ceds}.}
\label{table:cluster}
\resizebox{\textwidth}{!}{
\begin{tabular}{l|cc|cc|cc|cc}
\hline
\multirow{3}{*}{} &   \multicolumn{4}{c|}{spam-base} &  \multicolumn{4}{c}{cover-type} 
\\
&   \multicolumn{2}{c}{$(d,p)=(100,100)$} &  \multicolumn{2}{c|}{$(d,p)=(500,100)$} &  \multicolumn{2}{c}{$(d,p)=(100,500)$} &  \multicolumn{2}{c}{$(d,p)=(500,500)$}
\\ 
& objective & \multicolumn{1}{c}{stationary} & objective & stationary & objective & \multicolumn{1}{c}{stationary} & objective & stationary  \\ \hline
SIPM-ME$^1$  & -19.14  &3.108e-1 & -17.21 & 1.544e-1 &  0.2909 & 6.822e-2 & 0.2302  & 3.996e-2\\ 
SIPM-ME$^+$  &  -19.15   &2.950e-2  & -17.24 & 3.789e-2 &  0.2908 & 6.165e-3  & 0.2299 & 9.096e-3\\
SIPM-PM  &  -19.15   & 2.952e-2  & -17.23 & 3.790e-2 & 0.2908 & 6.165e-3 & 0.2297 & 1.108e-2\\ 
SIPM-EM  &  -19.15 & 2.948e-2 & -17.23 & 3.611e-2 & 0.2908 & 6.163e-3 & 0.2297  & 9.179e-3\\ 
SIPM-RM  &  -19.15  & 2.947e-2  & -17.24 & 3.609e-2& 0.2908 & 6.163e-3 & 0.2296  & 9.096e-3\\
IPM-FG  &  -19.15 & 2.950e-2  & -17.24 & 3.623e-2  & 0.2909  & 6.235e-3  & 0.2297 & 9.026e-3 \\
\hline
\end{tabular}}
\end{table}

\begin{figure}[ht]
\centering
\begin{minipage}[b]{0.24\linewidth}
\centering
\includegraphics[width=\linewidth]{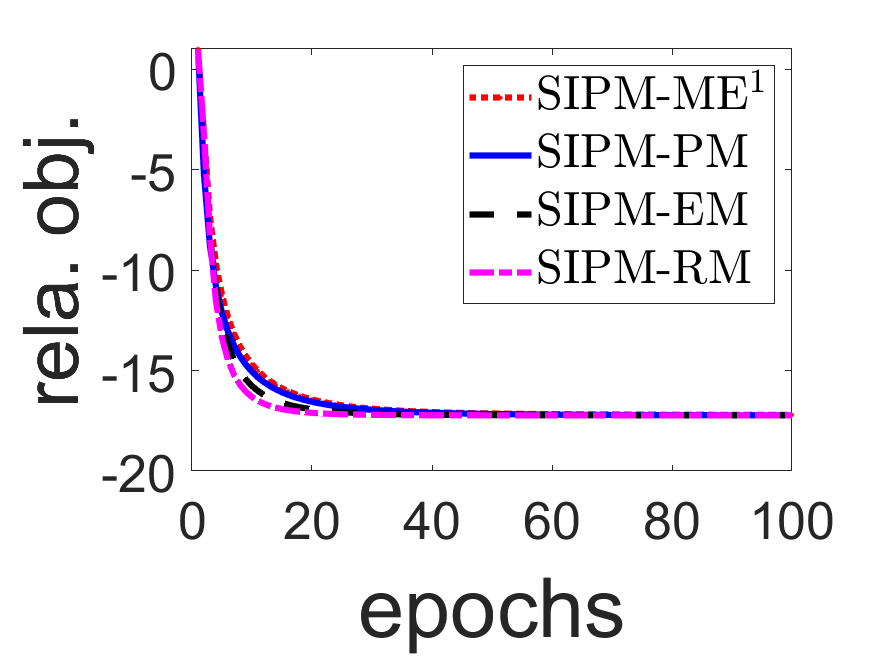}
\end{minipage}
\begin{minipage}[b]{0.24\linewidth}
\centering
\includegraphics[width=\linewidth]{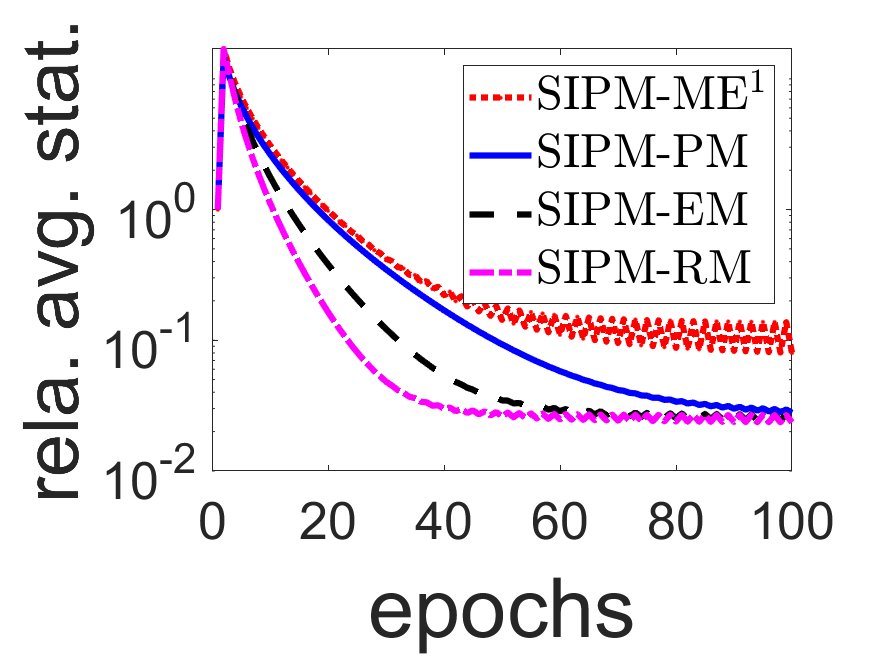}
\end{minipage}
\hfill
\begin{minipage}[b]{0.24\linewidth}
\centering
\includegraphics[width=\linewidth]{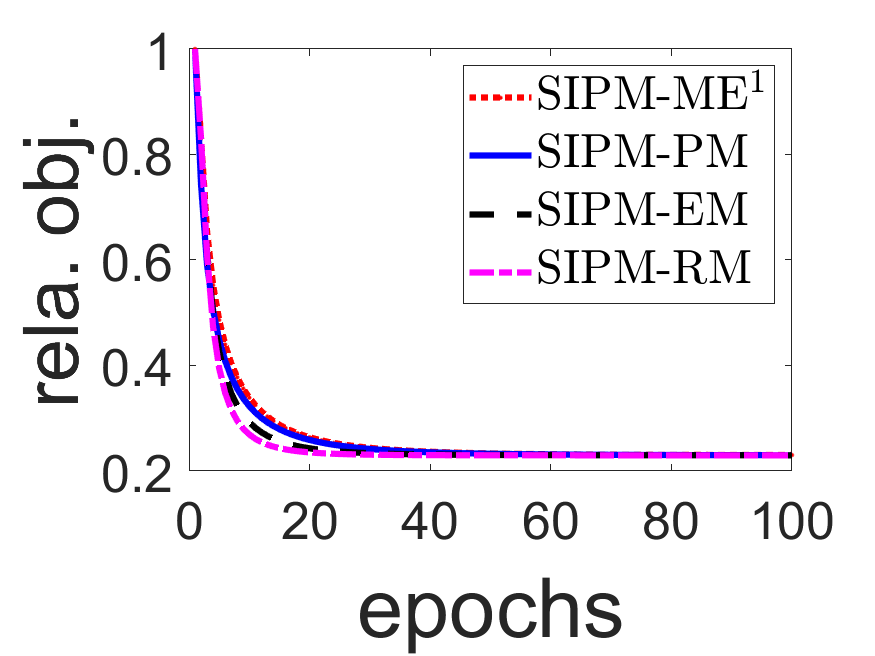}
\end{minipage}
\begin{minipage}[b]{0.24\linewidth}
\centering
\includegraphics[width=\linewidth]{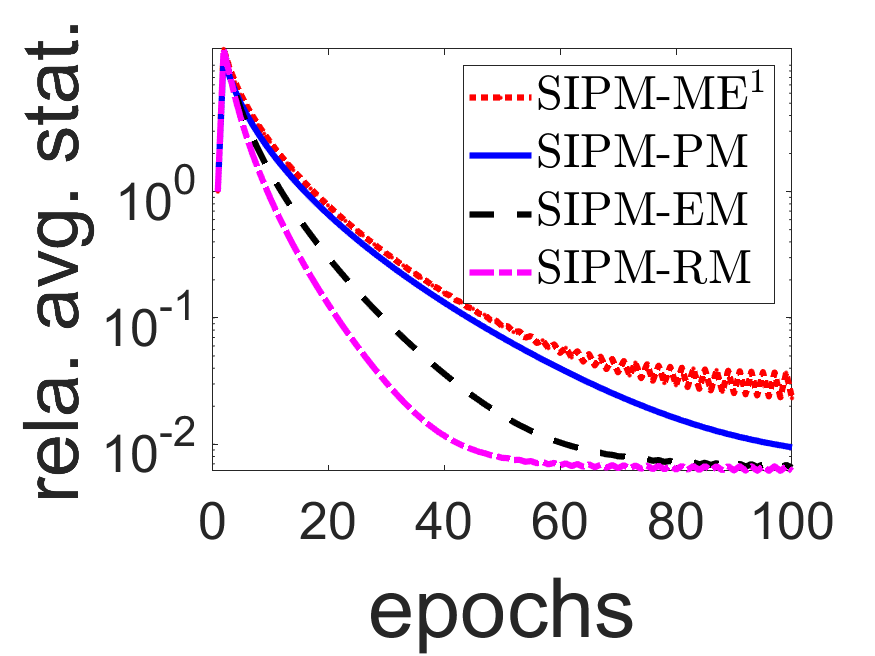}
\end{minipage}
\caption{\small Convergence behavior of the relative objective value and average relative stationary error for each epoch. The first two and last two plots correspond to the `spam-base' and `cover-type' datasets, respectively.}
\label{fig:sdp-error}
\end{figure}

We apply our SIPMs and IPM-FG to solve \cref{pb:ceds}. {For SIPMs, we set the barrier function as $B(\Sigma)=-\ln(\mathrm{det}(\Sigma))$, associated with the positive semidefinite cone $\mathbb{S}_+^{p}\doteq\{\Sigma\succeq0\}$.} For SIPM-ME$^1$, SIPM-PM, SIPM-EM, and SIPM-RM, we {set the batch size as 10 and 50 for the case when $p$ is 100 and 500, respectively, and} set the maximum number of epochs as $100$. {For SIPM-ME$^+$, we initialize the batch size as 1 and increase it by 1 per iteration.} For a fair comparison, we set the maximum number of iterations for SIPM-ME$^+$ and IPM-FG such that the total number of data points used during training equals that of the other methods in $100$ epochs. For all methods, we set the initial point $W^0$ with all diagonal elements equal to $k/d$ and all other elements equal to $(d-k)/(d(d-1))$. {We set the other hyperparameters---including step sizes, momentum parameters, and barrier parameters---according to the strategy described in Section \ref{subsec:rlr}.}

From \cref{fig:sdp}, we observe that the solutions obtained using our SIPMs on \cref{pb:ceds} effectively cluster the data streams. By comparing the relative objective values and the relative estimated stationary errors in \cref{table:cluster}, we observe that our proposed SIPM-ME$^+$, SIPM-PM, SIPM-EM, and SIPM-RM yield solutions of similar quality to the deterministic variant IPM-FG. In addition, SIPM-ME$^1$ converges slowly and returns suboptimal solutions, which corroborates the theoretical results that incorporating momentum greatly facilitates gradient estimation, thereby improving solution quality. From \cref{fig:sdp-error}, we see that SIPM-ME$^1$ converges much more slowly than the other three variants, while SIPM-RM is slightly faster than SIPM-PM and SIPM-EM. This observation corroborates the {iteration complexity} we established for these methods.


\section{Proof of the main results}

In this section, we provide proofs of our main results presented in \cref{sec:not-pre,sec:sasa}, which are particularly, \cref{lem:apr-stat,lem:int-frame,lem:est-error-me,lem:pm-estimate,lem:feas-em,lem:em-estimate,lem:tech-ap-use,lem:rm-var-err,lem:tech-eta-gma-rm}, and \cref{thm:stat-ave-me,thm:converge-rate-ipmpm,thm:em-stat-ave,thm:converge-rate-ipmrm,cor:order-me,cor:order-pm,cor:order-em,cor:order-rm}.

We start with the next lemma regarding the properties of the $\vartheta$-LHSC barrier function.

\begin{lemma}\label{lem:tech-barrier}
Let $x\in\mathrm{int}\mathcal{K}$ and $s_\eta\in(0,1)$ be given. Then the following statements hold for the $\vartheta$-LHSC barrier function $B$.
\begin{enumerate}[{\rm (i)}]
\item $(\|\nabla B(x)\|_x^*)^2=-x^T\nabla B(x) = \|x\|_x^2 = \vartheta$.
\item $\{y:\|y-x\|_x<1\}\subset\mathrm{int}\gK$ and $\{s:\|s+\nabla B(x)\|_{x}^*\le 1\}\subseteq \gK^*$.
\item For any $y$ satisfying $\|y-x\|_x<1$, the relation
$(1-\|y-x\|_x)\|v\|_x^*\le \|v\|_y^* \le (1-\|y-x\|_x)^{-1}\|v\|_x^*$ holds for all $v\in\R^n$.
\item $\|\nabla B(y) - \nabla B(x)\|_x^* \le \|y-x\|_x/(1-s_\eta)$ holds for all $y$ with $\|y-x\|_x\le s_\eta$.
\end{enumerate}
\end{lemma}

\begin{proof}
The proof of statements (i)-(iii) can be found in Lemma 1 in \cite{he2023newton}.

We next prove statement (iv). Let $y$ satisfy $\|y-x\|_x\le s_\eta$. Using this and Eq.(2.3) in \cite{nemirovski2004interior}, we obtain that $|z^T(\nabla B(y) - \nabla B(x))|\le\|y-x\|_x/(1- s_\eta)$ holds for all $z$ satisfying $\|z\|_x\le 1$. In addition, notice that $\max_{\|z\|_x\le 1}|z^T(\nabla B(y) - \nabla B(x))| = \| \nabla B(y) - \nabla B(x)\|_x^*$. Combining these, we conclude that statement (iv) holds as desired.
\end{proof}

\subsection{Proof of \cref{lem:apr-stat,lem:int-frame}}\label{subsec:proof-not-pre}

In this subsection, we prove \cref{lem:apr-stat,lem:int-frame}.

\begin{proof}[\textbf{Proof of \cref{lem:apr-stat}}]
It suffices to prove that the last two relations in \cref{def:1st-sc} hold. Using $\mu>0$, $\|\nabla \phi_\mu(x) + A^T\lambda \|_x^* \le \mu$, and the definition of $\phi_\mu$ in \cref{def:feas-r}, we have $\|((1+\mu)\nabla f(x) + A^T\lambda)/\mu + \nabla B(x)\|_x^* \le 1$, which together with \cref{lem:tech-barrier}(ii) implies that $((1+\mu)\nabla f(x) + A^T\lambda)/\mu\in\gK^*$. Thus, we observe that $x$ satisfies the third relation in \cref{def:1st-sc} with $\tilde\lambda=\lambda/(1+\mu)$. 

We next show that $x$ satisfies the fourth relation in \cref{def:1st-sc} with $\tilde\lambda=\lambda/(1+\mu)$ and $\epsilon\ge(1+\sqrt{\vartheta})\mu$. Using $\|\nabla\phi_\mu(x) + A^T\lambda \|_x^* \le \mu$ and $\|\nabla B(x)\|_x^*=\sqrt{\vartheta}$ from \cref{lem:tech-barrier}(i), we have $\mu\ge \|(1+\mu)\nabla f(x) + A^T\lambda\|_x^* - \mu\|\nabla B(x)\|_x^* = \|(1+\mu)\nabla f(x) + A^T\lambda\|_x^* - \mu\sqrt{\vartheta}$. Therefore, it follows that $\|\nabla f(x) + A^T\lambda/(1+\mu)\|_x^*\le(1+\sqrt{\vartheta})\mu/(1+\mu)<(1+\sqrt{\vartheta})\mu$, which implies that $x$ satisfies the fourth relation in \cref{def:1st-sc} with $\tilde\lambda=\lambda/(1+\mu)$ and any $\epsilon\ge(1+\sqrt{\vartheta})\mu$. Hence, the proof is complete.    
\end{proof}

\begin{proof}[\textbf{Proof of \cref{lem:int-frame}}]
Notice from the update of $x^{k+1}$ in \cref{step-dp-update} that 
\begin{align*}
\|x^{k+1}-x^k\|_{x^k}\overset{\cref{step-dp-update}}{=} \eta_k\frac{\|H_k(m^k+A^T\lambda^k)\|_{x^k}}{\|m^k+A^T\lambda^k\|_{x^k}^*} \overset{\cref{def:local-norm}}{=}\eta_k\le s_\eta <1.
\end{align*}
Hence, $\|x^{k+1}-x^k\|_{x^k} = \eta_k$ for all $k\ge0$. We now prove $x^k\in\Omega^\circ$ for all $k\ge0$ by induction. Note from \cref{alg:unf-sipm} that $x^0\in\Omega^\circ$. Suppose $x^k\in\Omega^\circ$ for some $k\ge0$. We next prove $x^{k+1}\in\Omega^\circ$. Since $x^k\in\rmint\gK$ and $\|x^{k+1}-x^k\|_{x^k}<1$, it then follows from \cref{lem:tech-barrier}(ii) that $x^{k+1}\in\mathrm{int}\gK$. In addition, by $Ax^k=b$ and \cref{step-dp-update}, one has that
\begin{align*}
Ax^{k+1} 
\overset{\cref{step-dp-update}}{=} Ax^k - \eta_k\frac{AH_k(m^k-A^T(AH_kA^T)^{-1}AH_km^k)}{\|m^k+A^T\lambda^k\|_{x^k}^*} = Ax^k=b.
\end{align*}
This together with $x^{k+1}\in\mathrm{int}\gK$ implies that $x^{k+1}\in\Omega^\circ$, which completes the induction.
\end{proof}

\subsection{Proof of some auxiliary lemmas}\label{subsec:proof-aux}

In this subsection, we prove several auxiliary lemmas. The following lemma shows that $\phi_\mu$ is locally Lipschitz continuous under \cref{asp:basic}(b) and a useful descent inequality holds.

\begin{lemma}\label{lem:Lip-phimu}
Suppose that \cref{asp:basic}(b) holds. Let $\mu\in(0,1]$. Then,
\begin{align}
&\|\nabla\phi_\mu(y)-\nabla\phi_\mu(x)\|_x^* \le L_{\phi}\|y-x\|_x\qquad  \forall x,y\in\Omega^\circ\text{ with } \|y-x\|_x\le s_\eta,\label{Lip-phimu}\\
&\phi_\mu(y) \le \phi_\mu(x) + \nabla\phi_\mu(x)^T(y-x) + L_{\phi}\|y-x\|_x^2/2\qquad  \forall x,y\in\Omega^\circ\text{ with } \|y-x\|_x\le s_\eta,\label{ineq:desc}
\end{align}
where $\phi_\mu$ and $\Omega^\circ$ are defined in \cref{def:feas-r}, and $L_\phi$ is defined in \cref{def:Lphi-1}.
\end{lemma}

\begin{proof}
Fix any $x,y\in\Omega^\circ$ satisfying $y\in\{y:\|y-x\|_x\le s_\eta\}$. Using the definition of $\phi_\mu$ in \cref{def:feas-r}, $\mu\in(0,1]$, \cref{asp:basic}(b), and \cref{lem:tech-barrier}(iv), we obtain that \cref{Lip-phimu} holds. 

We next prove \cref{ineq:desc}. Indeed, one has
\begin{align*}
\phi_\mu(y)  - \phi_\mu(x) - \nabla\phi_\mu(x)^T(y-x) & = \int^1_0 (\nabla\phi_\mu(x+t(y-x))-\nabla\phi_\mu(x))^T(y-x) \mathrm{d}t\\
&\overset{\cref{def:local-norm}}{\le} \int^1_0 \|\nabla\phi_\mu(x+t(y-x))-\nabla\phi_\mu(x)\|_x^*\mathrm{d}t\|y-x\|_x \overset{\cref{Lip-phimu}}{\le} \frac{L_{\phi}}{2}\|y-x\|_x^2.
\end{align*}
Hence, \cref{ineq:desc} holds as desired.
\end{proof}

The following lemma establishes a key inequality under \cref{asp:2nd-smooth}, whose proof is identical to that of Lemma 3 in \cite{he2023newton} and is therefore omitted here.

\begin{lemma}\label{lem:2nd-smth-desc}
Suppose that \cref{asp:2nd-smooth} holds. Then,
\begin{align*}
\|\nabla f(y) - \nabla f(x) - \nabla^2 f(x)(y-x)\|_x^*  \le \frac{L_2}{2}\|y-x\|_x^2\qquad \forall x,y\in\Omega^\circ\text{ with } \|y-x\|_x\le s_{\eta},
\end{align*}
where $\Omega^\circ$ is defined in \cref{def:feas-r}, $L_2$ is given in \cref{asp:2nd-smooth}(b), and $s_\eta$ is an input of \cref{alg:unf-sipm}.
\end{lemma}


The following lemma concerns the descent of $\phi_\mu$ for iterates generated by \cref{alg:unf-sipm}. 

\begin{lemma}\label{lem:general-tech}
Suppose that \cref{asp:basic} holds. Let $\{(x^k,\lambda^k)\}_{k\ge0}$ be the sequence generated by \cref{alg:unf-sipm} with input parameters $\{(\eta_k,\mu_k)\}$. Then, for all $k\ge0$, it holds that
\begin{align}
\phi_{\mu_k}(x^{k+1}) \le \phi_{\mu_k}(x^k) - \eta_k \|\nabla\phi_{\mu_k}(x^k)+A^T\lambda^k\|^*_{x^k} + 4\eta_k \|\nabla f(x^k)-\overline{m}^k\|_{x^k}^* + \frac{L_{\phi}}{2}\eta_k^2,\label{ineq:general-descent}
\end{align}
where $\phi_{\mu_k}$ and $L_{\phi}$ are defined in \cref{def:feas-r} and \cref{def:Lphi-1}, respectively.
\end{lemma}

\begin{proof}
We fix an arbitrary $k\ge0$. Recall from \cref{lem:int-frame} that $\|x^{k+1}-x^k\|_{x^k}=\eta_k$ and $x^k,x^{k+1}\in\Omega^\circ$. It then follows that $A(x^{k+1}-x^k)=0$, and from \cref{step-dp-update} that
\begin{align}\label{eq:square-unit}
\langle m^k + A^T\lambda^k, x^{k+1} - x^k\rangle \overset{\cref{step-dp-update}}{=} -\eta_k\frac{(m^k+A^T\lambda^k)^TH_k(m^k+A^T\lambda^k)}{\|m^k + A^T\lambda^k\|^*_{{x^k}}}\overset{\cref{def:local-norm}}{=}-\eta_k\|m^k + A^T\lambda^k\|_{x^k}^*.
\end{align}
In view of these and \cref{ineq:desc} with $(x,y,\mu,\eta)=(x^k,x^{k+1},\mu_k,\eta_k)$, one can see that
\begin{align}
\phi_{\mu_k}(x^{k+1})& \le \phi_{\mu_k}(x^k) + \langle\nabla\phi_{\mu_k}(x^k),x^{k+1}-x^k\rangle + \frac{L_{\phi}}{2}\|x^{k+1} - x^k\|^2_{x^k}\nonumber\\
& = \phi_{\mu_k}(x^k) + \langle m^k + A^T\lambda^k, x^{k+1} - x^k\rangle + \langle\nabla\phi_{\mu_k}(x^k)-m^k, x^{k+1}-x^k\rangle + \frac{L_{\phi}}{2}\eta_k^2\nonumber\\
& \overset{ \cref{eq:square-unit}}{=} \phi_{\mu_k}(x^k) - \eta_k \|m^k+A^T\lambda^k\|_{x^k}^* + \langle\nabla\phi_{\mu_k}(x^k)-m^k, x^{k+1} - x^k\rangle + \frac{L_{\phi}}{2}\eta_k^2\nonumber\\
&\le \phi_{\mu_k}(x^k) - \eta_k \|m^k + A^T\lambda^k\|_{x^k}^* + \eta_k\|\nabla\phi_{\mu_k}(x^k)-m^k\|_{x^k}^* + \frac{L_{\phi}}{2}\eta_k^2\nonumber\\
&\le\phi_{\mu_k}(x^k) - \eta_k\|\nabla\phi_{\mu_k}(x^k)+A^T\lambda^k\|_{x^k}^* + 2\eta_k\|\nabla\phi_{\mu_k}(x^k)-m^k\|_{x^k}^* + \frac{L_{\phi}}{2}\eta_k^2,\nonumber\\
&=\phi_{\mu_k}(x^k) - \eta_k\|\nabla\phi_{\mu_k}(x^k)+A^T\lambda^k\|_{x^k}^* + 2\eta_k{(1+\mu_k)}\|\nabla f(x^k)- \overline{m}^k\|_{x^k}^* + \frac{L_{\phi}}{2}\eta_k^2,\nonumber
\end{align}
where the first equality is due to $\|x^{k+1}-x^k\|_{x^k}=\eta_k$ and $A(x^{k+1}-x^k)=0$, the second inequality follows from the Cauchy-Schwarz inequality and $\|x^{k+1}-x^k\|_{x^k}=\eta_k$, the last inequality is due to the triangular inequality, and the last equality is due to the definition of $\phi_\mu$ and $m^k=\overline{m}^k+\mu_k(\overline{m}^k+\nabla B(x^k))$. This together with $\mu_k\le 1$ proves this lemma as desired.
\end{proof}

We next present a lemma regarding the estimation of the partial sums of series.

\begin{lemma}\label{lem:series}
Let $\zeta(\cdot)$ be a convex univariate function. Then, for any integers $a,b$ satisfying $[a-1/2, b+1/2]\subset\mathrm{dom}\zeta$, it holds that $\sum_{p=a}^{b} \zeta(p) \le \int_{a-1/2}^{b+1/2}\zeta(\tau) \mathrm{d}\tau$.
\end{lemma}

\begin{proof}
Since $\zeta$ is convex, one has $\sum_{p=a}^{b} \zeta(p) \le \sum_{p=a}^{b} \int_{p-1/2}^{p+1/2}\zeta(\tau)\mathrm{d}\tau = \int_{a-1/2}^{b+1/2}\zeta(\tau)\mathrm{d}\tau$.    
\end{proof}

As a consequence of \cref{lem:series}, we consider $\zeta(\tau)=1/\tau^\alpha$ for some $\alpha\in(0,\infty]$, where $\tau\in(0,\infty)$. Then, for any positive integers $a,b$, one has
\begin{align}
\sum_{p=a}^b 1/p^\alpha \le \left\{\begin{array}{ll}
\ln(b+1/2) - \ln(a-1/2)&\text{if }\alpha=1,\\[6pt]
\frac{1}{1-\alpha}((b+1/2)^{1-\alpha} - (a-1/2)^{1-\alpha})&\text{if }\alpha\in(0,1)\cup(1,+\infty).
\end{array}\right.
\label{upbd:series-ka}
\end{align}

{
We next provide an auxiliary lemma that will be used to analyze the complexity involving polylogarithmic terms for our methods.

\begin{lemma}\label{lem:rate-complexity}
Let $\beta\in(0,1)$ and $u\in(0,1/e)$ be given. Then, $1/v^\beta\ln v\le 2u/\beta$ holds for all $v\ge(1/u\ln(1/u))^{1/\beta}$.
\end{lemma}

\begin{proof}
Fix any $v$ satisfying $v\ge(1/u\ln(1/u))^{1/\beta}$. It then follows from $u\in(0,1/e)$ that
\begin{align}\label{aux-uv-ln}
v\ge(1/u\ln(1/u))^{1/\beta}> e^{1/\beta}.    
\end{align}
Denote $\phi(v)\doteq 1/v^{\beta}\ln v$. One can verify that $\phi$ is decreasing over $(e^{1/\beta},\infty)$. Using this and \cref{aux-uv-ln}, we obtain that 
\begin{align*}
1/v^{\beta}\ln v = \phi(v) \le  \phi((1/u\ln(1/u))^{1/\beta}) = \frac{u}{\beta} \Big(1+\frac{\ln\ln(1/u)}{\ln(1/u)}\Big)\le \frac{2u}{\beta},
\end{align*}
where the last inequality is due to $\ln\ln(1/u)\le\ln(1/u)$ for all $u\in(0,1/e)$. Hence, the conclusion of this lemma holds as desired.
\end{proof}
}

\subsection{Proof of the main results in \cref{subsec:ipm-mb}}\label{subsec:proof-me}

\begin{proof}[\textbf{Proof of \cref{lem:est-error-me}}]
Fix any $k\ge0$. It follows from \cref{asp:unbias-boundvar} and \cref{sipm-over-mk} that
\begin{align*}
\E_{\{\xi_i^k\}_{i\in \mathscr{B}_k}}[(\|\overline{m}^k - \nabla f(x^k)\|_{x^k}^*)^2]
= \frac{1}{|\mathscr{B}_k|^2} \sum_{i\in \mathscr{B}_k}\E_{\xi_i^k}[(\|G(x^k;\xi_i^k) - \nabla f(x^k)\|_{x^k}^*)^2] \le \frac{\sigma^2}{|\mathscr{B}_k|},
\end{align*}
where the first equality follows from \cref{sipm-over-mk} and the first relation in \cref{asp:unbias-boundvar}, and the last inequality follows from the second relation in \cref{asp:unbias-boundvar}. Hence, \cref{bd:variance} holds as desired.    
\end{proof}

\begin{proof}[\textbf{Proof of \cref{thm:stat-ave-me}}]
By summing \cref{ineq:general-descent} over $k=0,\ldots,K-1$ and taking expectation on $\{\xi_i^k\}_{i\in \mathscr{B}_k,1\le k\le K-1}$, we have
\begin{align*}
&\sum_{k=0}^{K-1}\eta_k \E[\|\nabla\phi_{\mu_k}(x^k) + A^T{\lambda}^k\|_{x^k}^*]\le \phi_{\mu_0}(x^0) - \E[\phi_{\mu_{K}}(x^{K})] + \sum_{k=0}^{K-1}(\mu_{k+1}-\mu_k)\E[f(x^{k+1}) + B(x^{k+1})]\nonumber\\
&\qquad  + 4\sum_{k=0}^{K-1}\eta_k\E[\|\nabla f(x^k) - \overline{m}^k\|_{x^k}^*] + \frac{L_{\phi}}{2}\sum_{k=0}^{K-1}\eta_k^2\nonumber\\
&\le\phi_{\mu_0}(x^0) - \E[\phi_{\mu_{K}}(x^{K})] + (\mu_{K}-\mu_0)\phi_{\mathrm{low}} +  4\sum_{k=0}^{K-1}\eta_k\E[\|\nabla f(x^k) - \overline{m}^k\|_{x^k}^*] + \frac{L_{\phi}}{2}\sum_{k=0}^{K-1}\eta_k^2\nonumber\\
&\le \phi_{\mu_0}(x^0) - \E[\phi_{\mu_{K}}(x^{K})] + (\mu_{K}-\mu_0)\phi_{\mathrm{low}} + 4\sigma\sum_{k=0}^{K-1}\frac{\eta_k}{|\mathscr{B}_k|^{1/2}} + \frac{L_{\phi}}{2}\sum_{k=0}^{K-1}\eta_k^2,\nonumber\\
&\le f(x^0) + \mu_0 (f(x^0) + B(x^0)) - (1+\mu_0)\phi_{\mathrm{low}} + 4\sigma\sum_{k=0}^{K-1}\frac{\eta_k}{|\mathscr{B}_k|^{1/2}} + \frac{L_{\phi}}{2}\sum_{k=0}^{K-1}\eta_k^2\nonumber\\
&\le f(x^0) + [f(x^0) + B(x^0)]_+ - 2[\phi_{\mathrm{low}}]_- + 4\sigma\sum_{k=0}^{K-1}\frac{\eta_k}{|\mathscr{B}_k|^{1/2}} + \frac{L_{\phi}}{2}\sum_{k=0}^{K-1}\eta_k^2, 
\end{align*}
where the first inequality follows from  \cref{ineq:general-descent}, the second inequality follows from $\mu_k-\mu_{k+1}\ge0$ and $f(x^{k+1}) + B(x^{k+1})\ge \phi_{\mathrm{low}}$ (due to \cref{def:lwbd-phi} and $x^{k+1}\in\Omega^\circ$), the third inequality follows from $\E[X]^2\le \E[X^2]$ for any random variable $X$ and \cref{bd:variance}, the fourth inequality is due to the definitions of $\phi_\mu$ and \cref{to-lwbd-phimu}, and the last inequality is due to $\mu_0\in(0,1]$. Using the above inequality, \cref{lwbd-fb}, and the fact that $\{\eta_k\}_{k\ge0}$ is nonincreasing, we have
\begin{align*}
\sum_{k=0}^{K-1} \E[\|\nabla\phi_{\mu_k}(x^k) + A^T{\lambda}^k\|_{x^k}^*]
\le\frac{\Delta(x^0)}{\eta_{K-1}} + \frac{1}{\eta_{K-1}} \sum_{k=0}^{K-1}\eta_k(4\sigma/|\mathscr{B}_k|^{1/2}+L_\phi\eta_k/2).
\end{align*}
Hence, \cref{upbd:me-exp-xk} holds as desired.   
\end{proof}

\begin{proof}[\textbf{Proof of \cref{cor:order-me}}]
Substituting \cref{etak-Bk-me} into \cref{upbd:me-exp-xk}, we obtain that for all $K\ge3$,
\begin{align}
\sum_{k=0}^{K-1}\E[\|\nabla \phi_{\mu_k}(x^k) + A^T{\lambda}^k\|^*_{x^k}] & \overset{\cref{upbd:me-exp-xk}\cref{etak-Bk-me}}{\le} \frac{K^{1/2}\Delta(x^0)}{s_\eta} + K^{1/2} \sum_{k=0}^{K-1}\frac{8\sigma + s_\eta L_\phi}{2(k+1)}\nonumber\\
&\le \frac{K^{1/2}\Delta(x^0)}{s_\eta} +  \Big(4\sigma + \frac{s_\eta L_\phi}{2}\Big)K^{1/2}\ln(2K+1) \nonumber \\
& \le \Big(\frac{\Delta(x^0)}{s_\eta} + 8\sigma+ s_\eta L_\phi\Big)K^{1/2}\ln K\nonumber\\
&\overset{\cref{def:Kme}}{=}M_{\mathrm{me}}K^{1/2}\ln K/2,\label{upbd:me-exp-xk-proof}
\end{align}
where the second inequality follows from $\sum_{k=0}^{K-1}1/(k+1)\le\ln(2K+1)$ due to \cref{upbd:series-ka} with $(a,b,\alpha)=(1,K,1)$, and the last inequality follows from $1\le \ln K$ and $\ln(2K+1)\le 2\ln K$ given that $K\ge3$. Since $\kappa(K)$ is uniformly drawn from $\{\lfloor K/2\rfloor,\ldots,K-1\}$, we have that for all $K\ge3$,
\begin{align}
\E[\|\nabla \phi_{\mu_{\kappa(K)}}(x^{\kappa(K)}) + A^T\lambda^{\kappa(K)}\|_{x^{\kappa(K)}}^*] & =\frac{1}{K - \lfloor K/2\rfloor}\sum_{k=\lfloor K/2\rfloor}^{K-1}\E[\|\nabla \phi_{\mu_k}(x^k) + A^T{\lambda}^k\|^*_{x^k}] \nonumber\\
&\le\frac{2}{K}\sum_{k=0}^{K-1}\E[\|\nabla \phi_{\mu_k}(x^k) + A^T{\lambda}^k\|^*_{x^k}] \overset{\cref{upbd:me-exp-xk-proof}}{\le} M_{\mathrm{me}}K^{-1/2}\ln K.\label{upbd-me-spec}
\end{align}
By \cref{lem:rate-complexity} with $(\beta,u,v)=(1/2,\epsilon/(4M_{\mathrm{me}}(1+\sqrt{\vartheta})),K)$, one can see that 
\begin{align*}
K^{-1/2}\ln K\le \frac{\epsilon}{M_{\mathrm{me}}(1+\sqrt{\vartheta})}\qquad\forall K\ge \Big(\frac{4M_{\mathrm{me}}(1+\sqrt{\vartheta})}{\epsilon}\ln\Big(\frac{4M_{\mathrm{me}}(1+\sqrt{\vartheta})}{\epsilon}\Big)\Big)^2,   
\end{align*}
which together with \cref{upbd-me-spec} implies that
\begin{align}
&\E[\|\nabla \phi_{\mu_{\kappa(K)}}(x^{\kappa(K)}) + A^T\lambda^{\kappa(K)}\|_{x^{\kappa(K)}}^*]\le \frac{\epsilon}{1+\sqrt{\vartheta}}\nonumber\\
&\forall K\ge \max\Big\{\Big(\frac{4M_{\mathrm{me}}(1+\sqrt{\vartheta})}{\epsilon}\ln\Big(\frac{4M_{\mathrm{me}}(1+\sqrt{\vartheta})}{\epsilon}\Big)\Big)^2,3\Big\}.    \label{inter-cmplx-me}
\end{align}
On the other hand, when $K\ge2((1+\sqrt{\vartheta})/\epsilon)^2$, by the definition of $\{\mu_k\}_{k\ge0}$ in \cref{etak-Bk-me} and the fact that $\kappa(K)$ is uniformly selected from $\{\lfloor K/2\rfloor,\ldots,K-1\}$, one has that $\mu_{\kappa(K)}=\mu_{\lfloor K/2\rfloor}=\epsilon/(1+\sqrt{\vartheta})$. Combining this with \cref{inter-cmplx-me}, we obtain that \cref{complexity-me} holds as desired, and the proof of this theorem is complete.
\end{proof}

\subsection{Proof of the main results in \cref{subsec:sipm-pm}}\label{subsec:proof-pm}

\begin{proof}[\textbf{Proof of \cref{lem:pm-estimate}}]
Fix any $k\ge0$. Recall from \cref{lem:int-frame} that $\|x^{k+1}-x^k\|_{x^k}=\eta_k$. Using this and \cref{sipm-pm-over-mk}, we have that
\begin{align}
&\E_{\xi^{k+1}}[(\|\overline{m}^{k+1} - \nabla f(x^{k+1})\|_{x^{k+1}}^*)^2] \nonumber\\ 
&\overset{\cref{sipm-pm-over-mk}}{=} \E_{\xi^{k+1}}[(\|(1-\gamma_k)(\overline{m}^k - \nabla f(x^{k+1})) + \gamma_k(G(x^{k+1},\xi^{k+1}) - \nabla f(x^{k+1}))\|_{x^{k+1}}^*)^2]\nonumber\\
&\overset{\cref{asp:unbias-boundvar}}{\le} (1-\gamma_k)^2(\|\overline{m}^k - \nabla f(x^{k+1})\|_{x^{k+1}}^*)^2 + \sigma^2\gamma_k^2\nonumber\\
&\le(1-\alpha_k)^2(\|\overline{m}^{k} - \nabla f(x^{k+1})\|_{x^{k}}^*)^2 + \sigma^2\gamma_k^2,\label{ineq:grad-error-pm}
\end{align}
where the last inequality follows from \cref{lem:tech-barrier}(iii) and the definition of $\alpha_k$. Also, notice that for all $a>0$,
\begin{align}
&(1-\alpha_k)^2(\|\overline{m}^k - \nabla f(x^{k+1})\|_{x^k}^*)^2\nonumber\\
&\le (1-\alpha_k)^2(1+a)(\|\overline{m}^k - \nabla f(x^k)\|_{x^k}^*)^2 +  (1-\alpha_k)^2(1+1/a)(\|\nabla f(x^{k+1}) - \nabla f(x^k)\|_{x^k}^*)^2\nonumber\\
&\le (1-\alpha_k)^2(1+a)(\|\overline{m}^k - \nabla f(x^k)\|_{x^k}^*)^2 + (1-\alpha_k)^2(1+1/a)L_1^2\eta_k^2,\nonumber
\end{align}
where the first inequality is due to {$\|u+v\|^2\le(1+a)\|u\|^2+(1+1/a)\|v\|^2$ for all $u,v\in\R^n$ and $a>0$}, and the last inequality follows from \cref{ineq:1st-Lip} and $\|x^{k+1}-x^k\|_{x^k}=\eta_k$. Letting $a=\alpha_k/(1-\alpha_k)$ and combining this inequality with \cref{ineq:grad-error-pm}, we obtain that
\begin{align*}
\E_{\xi^{k+1}}[(\|\overline{m}^{k+1} - \nabla f(x^{k+1})\|_{x^{k+1}}^*)^2]\le (1-\alpha_k)(\|\overline{m}^k - \nabla f(x^k)\|_{x^k}^*)^2 + \frac{(1-\alpha_k)^2L_1^2\eta_k^2}{\alpha_k} +\sigma^2 \gamma_k^2.
\end{align*}
Since $\gamma_k>\eta_k$, we have $\alpha_k\in(0,1)$. This along with the above inequality implies that \cref{ineq:var-recur} holds as desired.    
\end{proof}

\begin{proof}[\textbf{Proof of \cref{thm:converge-rate-ipmpm}}]
For convenience, we define the following potentials: 
\begin{align}\label{P-pm-k}
P_k \doteq \phi_{\mu_k}(x^k) + (\|\overline{m}^k - \nabla f(x^k)\|_{x^k}^*)^2/L_1\qquad \forall k\ge0.
\end{align}
Recall from Algorithm \ref{alg:unf-sipm} that $\{\mu_k\}$ is nonincreasing. By these, \cref{def:lwbd-phi}, \cref{ineq:var-recur}, and \cref{ineq:general-descent}, one has that for all $k\ge0$,
\begin{align}
&\E_{\xi^{k+1}}[P_{k+1}]\overset{\cref{P-pm-k}}{=}\E_{\xi^{k+1}}[\phi_{\mu_{k+1}}(x^{k+1}) + (\|\overline{m}^{k+1} - \nabla f(x^{k+1})\|_{x^{k+1}}^*)^2/L_1]\nonumber\\
&=  (\mu_{k+1}-\mu_k)(f(x^{k+1})+B(x^{k+1})) + \E_{\xi^{k+1}}[\phi_{\mu_{k}}(x^{k+1}) + (\|\overline{m}^{k+1} - \nabla f(x^{k+1})\|_{x^{k+1}}^*)^2/L_1]\nonumber\\
&\overset{\cref{def:lwbd-phi}\cref{ineq:var-recur}\cref{ineq:general-descent}}{\le} (\mu_{k+1}-\mu_k)\phi_{\mathrm{low}} + \phi_{\mu_k}(x^k)- \eta_k\|\nabla\phi_{\mu_k}(x^k)+A^T\lambda^k\|^*_{x^k} + 4\eta_k \|\overline{m}^k - \nabla f(x^k)\|_{x^k}^* \nonumber\\
&\qquad + \frac{L_{\phi}}{2}\eta_k^2 + (1-\alpha_k)(\|\overline{m}^k- \nabla f(x^k)\|_{x^k}^*)^2/L_1 + \frac{L_1\eta_k^2}{\alpha_k} +\frac{\sigma^2\gamma_k^2}{L_1}.\label{upbd-potent-pm}
\end{align}
In addition, notice that $4\eta_k \|\overline{m}^k - \nabla f(x^k)\|_{x^k}^* \le 4L_1\eta_k^2/\alpha_k +  \alpha_k(\|\overline{m}^k - \nabla f(x^k)\|_{x^k}^*)^2/L_{1}$, which together with \cref{P-pm-k} and \cref{upbd-potent-pm} implies that for all $k\ge0$,
\begin{align}
\E_{\xi^{k+1}}[P_{k+1}]
\le(\mu_{k+1}-\mu_k)\phi_{\mathrm{low}} + P_k - \eta_k\|\nabla\phi_{\mu_k}(x^k) + A^T{\lambda}^k\|^*_{x^k} + \frac{L_{\phi}}{2}\eta_k^2 + \frac{5L_1\eta_k^2}{\alpha_k} + \frac{\sigma^2\gamma_k^2}{L_{1}}.\label{desc-pm-oneiter}
\end{align}
On the other hand, by \cref{asp:unbias-boundvar}, \cref{to-lwbd-phimu}, \cref{P-pm-k}, $\overline{m}^{0}=G(x^0,\xi^0)$, and $\mu_0\le 1$, one has
\begin{align}
&\E_{\xi^0}[P_0] = \phi_{\mu_0}(x^0) + \E_{\xi^0}[(\|\overline{m}^0 - \nabla f(x^0)\|_{x^0}^*)^2]/L_1 \le f(x^0) +  [f(x^0) + B(x^0)]_+ + \sigma^2/L_1, \nonumber\\
&\E_{\{\xi^k\}_{k=0}^K}[P_K] = \phi_{\mu_K}(x^K) + \E_{\{\xi^k\}_{k=0}^K}[(\|\overline{m}^K- \nabla f(x^K)\|_{x^K}^*)^2]/L_1\overset{\cref{to-lwbd-phimu}}{\ge} (1+\mu_K)\phi_{\mathrm{low}}.\nonumber
\end{align}
By summing \cref{desc-pm-oneiter} over $k=0,\ldots,K-1$, and using the above two inequalities, \cref{lwbd-fb}, and the fact that $\{\eta_k\}_{k\ge0}$ is nonincreasing, we obtain that
\begin{align*}
\sum_{k=0}^{K-1} \E[\|\nabla\phi_{\mu_k}(x^k) + A^T{\lambda}^k\|_{x^k}^*] 
\le \frac{\Delta(x^0) + \sigma^2/L_1}{\eta_{K-1}} + \frac{1}{\eta_{K-1}}\sum_{k=0}^{K-1}\left(\frac{L_\phi}{2}\eta_k^2 + \frac{5L_1\eta_k^2}{\alpha_k} + \frac{\sigma^2\gamma_k^2}{L_{1}}\right).
\end{align*}
Hence, the conclusion of this theorem holds.    
\end{proof}

\begin{proof}[\textbf{Proof of \cref{cor:order-pm}}]
Observe from \cref{etak-Bk-pm} that $\gamma_k>\eta_k$ for all $k\ge0$. Thus, $\{(\eta_k,\gamma_k)\}_{k\ge0}$ defined as in  \cref{etak-Bk-pm} satisfies the assumption on $\{(\eta_k,\gamma_k)\}_{k\ge0}$ in \cref{thm:converge-rate-ipmpm}. Substituting \cref{etak-Bk-pm} and \cref{ineq:lwbd-ak-pm} into \cref{ineq:ave-stat-pm}, we obtain for all $K\ge3$,
\begin{align}
&\sum_{k=0}^{K-1}\E[\|\nabla \phi_{\mu_k}(x^k) + A^T\lambda^k\|^*_{x^k}]\nonumber\\
&\overset{\cref{ineq:ave-stat-pm}\cref{etak-Bk-pm} \cref{ineq:lwbd-ak-pm}}{\le} K^{3/4} \Big(\frac{\Delta(x^0) + \sigma^2/L_1}{s_\eta} + \sum_{k=0}^{K-1}\Big(\frac{s_\eta L_\phi}{2(k+1)^{3/2}} + \frac{5s_\eta L_1/(1-s_\eta) + \sigma^2/(s_\eta L_{1})}{k+1}\Big)\Big)\nonumber\\
&< K^{3/4}\Big(\frac{\Delta(x^0) + \sigma^2/L_1}{s_{\eta}}  + \frac{3s_\eta L_\phi}{2} + \Big(\frac{5s_\eta L_1}{1-s_\eta}+\frac{\sigma^2}{s_{\eta}L_1}\Big)\ln(2K+1)\Big)\nonumber\\
&\le \Big(\frac{\Delta(x^0) + \sigma^2/L_1}{s_{\eta}}  + \frac{3s_\eta L_\phi}{2} + 2\Big(\frac{5s_\eta L_1}{1-s_\eta}+\frac{\sigma^2}{s_{\eta}L_1}\Big)\Big){K^{3/4}}\ln K \overset{\cref{def:kpm}}{=} M_{\rmpm}K^{3/4}\ln K/2,
\label{ineq:ave-stat-pm-proof}
\end{align}
where the second inequality is because $\sum_{k=0}^{K-1}1/(k+1)^{3/2}\le2\sqrt{2}<3$ and $\sum_{k=0}^{K-1}1/(k+1)\le\ln(2K+1)$  due to \cref{upbd:series-ka} with $(a,b)=(1,K)$ and $\alpha=3/2,1$, and the last inequality is due to $1\le\ln K$ and $\ln(2K+1)\le 2\ln K$ given that $K\ge3$. Since $\kappa(K)$ is uniformly drawn from $\{\lfloor K/2\rfloor,\ldots,K-1\}$, we have that for all $K\ge3$,
\begin{align}
&\E[\|\nabla \phi_{\mu_{\kappa(K)}}(x^{\kappa(K)}) + A^T\lambda^{\kappa(K)}\|_{x^{\kappa(K)}}^*]=\frac{1}{K-\lfloor K/2\rfloor}\sum_{k=\lfloor K/2\rfloor}^{K-1}\E[\|\nabla \phi_{\mu_k}(x^k) + A^T\lambda^k\|_{x^k}^*]\nonumber\\
&\le \frac{2}{K}\sum_{k=0}^{K-1}\E[\|\nabla \phi_{\mu_k}(x^k) + A^T\lambda^k\|_{x^k}^*]\overset{\cref{ineq:ave-stat-pm-proof}}{\le}M_{\rmpm}K^{-1/4}\ln K.\label{upbd-rd-para-pm}
\end{align}
By \cref{lem:rate-complexity} with $(\beta,u,v)=(1/4,\epsilon/(8M_{\mathrm{pm}}(1+\sqrt{\vartheta})),K)$, one can see that 
\begin{align*}
K^{-1/4}\ln K\le \frac{\epsilon}{M_{\mathrm{pm}}(1+\sqrt{\vartheta})}\qquad\forall K\ge \Big(\frac{8M_{\mathrm{pm}}(1+\sqrt{\vartheta})}{\epsilon}\ln\Big(\frac{8M_{\mathrm{pm}}(1+\sqrt{\vartheta})}{\epsilon}\Big)\Big)^4,   
\end{align*}
which together with \cref{upbd-rd-para-pm} implies that
\begin{align}
&\E[\|\nabla \phi_{\mu_{\kappa(K)}}(x^{\kappa(K)}) + A^T\lambda^{\kappa(K)}\|_{x^{\kappa(K)}}^*]\le \frac{\epsilon}{1+\sqrt{\vartheta}}\nonumber\\
&\forall K\ge \max\Big\{\Big(\frac{8M_{\mathrm{pm}}(1+\sqrt{\vartheta})}{\epsilon}\ln\Big(\frac{8M_{\mathrm{pm}}(1+\sqrt{\vartheta})}{\epsilon}\Big)\Big)^4,3\Big\}.    \label{inter-cmplx-pm}
\end{align}
On the other hand, when $K\ge2((1+\sqrt{\vartheta})/\epsilon)^4$, by the definition of $\{\mu_k\}_{k\ge0}$ in \cref{etak-Bk-pm} and the fact that $\kappa(K)$ is uniformly selected from $\{\lfloor K/2\rfloor,\ldots,K-1\}$, one has that $\mu_{\kappa(K)}=\mu_{\lfloor K/2\rfloor}=\epsilon/(1+\sqrt{\vartheta})$. Combining this with \cref{inter-cmplx-pm}, we obtain that \cref{complexity-pm} holds as desired, and the proof of this theorem is complete.
\end{proof}

\subsection{Proof of the main results in \cref{subsec:sipm-em}}\label{subsec:proof-em}

\begin{proof}[\textbf{Proof of \cref{lem:feas-em}}]
Recall from \cref{alg:unf-sipm} and \cref{sipm-em-zk} that $x^0=x^{-1}\in\Omega^\circ$. Let $\eta_{-1}=0$. By these and \cref{lem:int-frame}, we have that $\|x^{k+1}-x^k\|_{x^k}=\eta_k$ and $x^k\in\Omega^\circ$ hold for all $k\ge-1$. Fix an arbitrary $k\ge-1$, we next prove $z^{k+1}\in\Omega^\circ$. In view of $\|x^{k+1}-x^k\|_{x^k}=\eta_k$ and \cref{sipm-em-over-mk-1}, we observe that $\|z^{k+1} - x^k\|_{x^k} = \eta_k/\gamma_k \le s_\eta<1$, which together with \cref{lem:tech-barrier}(ii) implies that $z^{k+1}\in\rmint\gK$. In addition, by \cref{sipm-em-over-mk-1} and $x^{k+1},x^k\in\Omega^\circ$, one has that $Az^{k+1} = A x^{k+1} + [(1-\gamma_k)/\gamma_k]A(x^{k+1}-x^k) = b$. Hence, $z^{k+1}\in\Omega^\circ$, which completes the proof.
\end{proof}

\begin{proof}[\textbf{Proof of \cref{lem:em-estimate}}]
Fix any $k\ge0$. Observe from \cref{sipm-em-over-mk-1} that $z^{k+1}-x^k = (x^{k+1} - x^k)/\gamma_k$. Recall from \cref{lem:int-frame} that $\|x^{k+1}-x^k\|_{x^k}=\eta_k$. It then follows that $\|z^{k+1}-x^k\|_{x^k}=\eta_k/\gamma_k$. In addition, using the update of $\overline{m}^{k+1}$ in \cref{sipm-em-over-mk-1}, we obtain that
\begin{align}
\overline{m}^{k+1} - \nabla f (x^{k+1}) &= (1-\gamma_k)(\overline{m}^k - \nabla f (x^k)) + \gamma_k (G(z^{k+1},\xi^{k+1}) - \nabla f(z^{k+1}))\nonumber\\
&\qquad - (\nabla f(x^{k+1})-\nabla f(x^k) - \nabla^2 f(x^k)(x^{k+1}-x^k))\nonumber\\
&\qquad + \gamma_k(\nabla f(z^{k+1})-\nabla f(x^k) - \nabla^2 f(x^k)(z^{k+1} -x^k)).\label{ob-em-m}
\end{align}
It then follows that
\begin{align}\label{long-derive}
&\E_{\xi^{k+1}}[(\|\overline{m}^{k+1} - \nabla f(x^{k+1})\|_{x^{k+1}}^*)^2] \le \frac{1}{(1-\eta_k)^2} \E_{\xi^{k+1}}[(\|\overline{m}^{k+1} - \nabla f(x^{k+1})\|_{x^k}^*)^2]\nonumber\\
&\overset{\cref{ob-em-m}}{=} \frac{\gamma_k^2\E_{\xi^{k+1}}[(\|G(z^{k+1},\xi^{k+1}) - \nabla f(z^{k+1})\|_{x^k}^*)^2]}{(1-\eta_k)^2} \nonumber\\
&\qquad + \frac{1}{(1-\eta_k)^2}(\|(1-\gamma_k)(\overline{m}^k- \nabla f(x^k))- (\nabla f(x^{k+1})-\nabla f(x^k) - \nabla^2 f(x^k)(x^{k+1}-x^k))\nonumber\\
&\qquad\qquad\qquad\qquad\qquad + \gamma_k(\nabla f(z^{k+1})-\nabla f(x^k) - \nabla^2 f(x^k)(z^{k+1} -x^k))\|_{x^{k}}^*)^2 \nonumber\\
&\le \frac{\gamma_k^2\E_{\xi^{k+1}}[(\|G(z^{k+1},\xi^{k+1}) - \nabla f(z^{k+1})\|_{z^{k+1}}^*)^2]}{(1-\eta_k)^2(1-\eta_k/\gamma_k)^2}+ (1-\alpha_k)^2(1+a)(\|\overline{m}^k - \nabla f(x^k)\|_{x^k}^*)^2 \nonumber\\
&\qquad + \frac{2(1+1/a)}{(1-\eta_k)^2}(\|\nabla f(x^{k+1})-\nabla f(x^k) - \nabla^2 f(x^k)(x^{k+1}-x^k)\|_{x^{k}}^*)^2\nonumber\\
&\qquad + \frac{2(1+1/a)\gamma_k^2}{(1-\eta_k)^2}(\|\nabla f(z^{k+1})-\nabla f(x^k) - \nabla^2 f(x^k)(z^{k+1} -x^k)\|_{x^{k}}^*)^2,
\end{align}
where the first inequality follows from \cref{lem:tech-barrier}(iii) and $\|x^{k+1} - x^k\|_{x^k}= \eta_k$, the equality follows from \cref{ob-em-m} and the first relation in \cref{asp:unbias-boundvar}, and the second inequality is due to $\|u+v\|^2\le(1+a)\|u\|^2+(1+1/a)\|v\|^2$ for all $u,v\in\R^n$ and $a>0$, \cref{lem:tech-barrier}(iii), and $\|z^{k+1} - x^k\|_{x^k}=\eta_k/\gamma_k$. In addition, using \cref{lem:2nd-smth-desc}, $\|x^{k+1} - x^k\|_{x^k}=\eta_k$, and $\|z^{k+1} - x^k\|_{x^k}=\eta_k/\gamma_k\le s_\eta$, we obtain that
\begin{align*}
&\|\nabla f(x^{k+1})-\nabla f(x^k) - \nabla^2 f(x^k)(x^{k+1}-x^k)\|_{x^{k}}^*\le L_2\|x^{k+1} - x^k\|_{x^k}^2/2\le L_2\eta_k^2/{2},\\
&\|\nabla f(z^{k+1})-\nabla f(x^k) - \nabla^2 f(x^k)(z^{k+1} -x^k)\|_{x^{k}}^*\le L_2\|z^{k+1} - x^k\|_{x^k}^2/2 \le L_2\eta_k^2/(2\gamma_k^2).
\end{align*}
In view of \cref{long-derive}, the above two inequalities, and the second relation in \cref{asp:unbias-boundvar}, and letting $a=\alpha_k/(1-\alpha_k)$, we obtain that
\begin{align*}
&\E_{\xi^{k+1}}[(\|\overline{m}^{k+1} - \nabla f (x^{k+1})\|_{x^{k+1}}^*)^2] \\
&\le \frac{\sigma^2\gamma_k^2}{(1-\eta_k)^2(1-\eta_k/\gamma_k)^2} + (1-\alpha_k)(\|\overline{m}^k - \nabla f(x^k)\|_{x^k}^*)^2  + \frac{L_2^2\eta_k^4}{2(1-\eta_k)^2\alpha_k} + \frac{L_2^2\eta_k^4}{2(1-\eta_k)^2\gamma_k^2\alpha_k}.
\end{align*}
This along with $\gamma_k\le 1$ and
the fact that $\{\eta_k\}_{k\ge0}$ is nonincreasing implies that this lemma holds.
\end{proof}

\begin{proof}[\textbf{Proof of \cref{thm:em-stat-ave}}]
For convenience, we define the following potentials:
\begin{align}\label{P-em-k}
P_k\doteq \phi_{\mu_k}(x^k) + p_k(\|\overline{m}^k -\nabla f(x^k)\|_{x^k}^*)^2\qquad \forall k\ge0.
\end{align}
Recall from Algorithm \ref{alg:unf-sipm} that $\{\mu_k\}$ is nonincreasing. By these, \cref{def:lwbd-phi}, \cref{ineq:var-recur-em}, and \cref{ineq:general-descent}, one has that for all $k\ge0$,
\begin{align}
&\E_{\xi^{k+1}}[P_{k+1}] \overset{\cref{P-em-k}}{=} \E_{\xi^{k+1}}[\phi_{\mu_{k+1}}(x^{k+1})+p_{k+1}(\|\overline{m}^{k+1} - \nabla f(x^{k+1})\|_{x^{k+1}}^*)^2]\nonumber\\ 
&= (\mu_{k+1}-\mu_k)(f(x^{k+1})+B(x^{k+1})) + \E_{\xi^{k+1}}[\phi_{\mu_k}(x^{k+1})+p_{k+1}(\|\overline{m}^{k+1} - \nabla f(x^{k+1})\|_{x^{k+1}}^*)^2]\nonumber\\
&\overset{\cref{def:lwbd-phi}\cref{ineq:var-recur-em}\cref{ineq:general-descent}}{\le} (\mu_{k+1}-\mu_k)\phi_{\mathrm{low}} + \phi_{\mu_k}(x^k) - \eta_k \|\nabla\phi_{\mu_k}(x^k) + A^T{\lambda}^k\|_{x^k}^* + 4\eta_k\|\nabla f(x^k) - \overline{m}^k\|_{x^k}^* + \frac{L_{\phi}}{2}\eta_k^2\nonumber\\
& +(1-\alpha_k)p_{k+1}(\|\overline{m}^k - \nabla f(x^k)\|_{x^k}^*)^2 +  \frac{L_2^2\eta_k^4p_{k+1}}{(1-\eta_0)^2\gamma_k^2\alpha_k} + \frac{\sigma^2\gamma_k^2p_{k+1}}{(1-\eta_0)^2(1-\eta_k/\gamma_k)^2}\nonumber\\
&\le (\mu_{k+1}-\mu_k)\phi_{\mathrm{low}} + \phi_{\mu_k}(x^k) - \eta_k \|\nabla\phi_{\mu_k}(x^k) + A^T{\lambda}^k\|_{x^k}^* + 4\eta_k\|\nabla f(x^k) - \overline{m}^k\|_{x^k}^* + \frac{L_{\phi}}{2}\eta_k^2\nonumber\\
&+ (1-\alpha_k/2)p_{k}(\|\overline{m}^k - \nabla f(x^k)\|_{x^k}^*)^2 +  \frac{L_2^2\eta_k^4p_{k+1}}{(1-\eta_0)^2\gamma_k^2\alpha_k} + \frac{\sigma^2\gamma_k^2p_{k+1}}{(1-\eta_0)^2(1-\eta_k/\gamma_k)^2},\nonumber
\end{align}
where the second inequality follows from $(1-\alpha_k)p_{k+1}\le (1-\alpha_k/2)p_{k}$. In addition, note that $4\eta_k\|\nabla f(x^k) - \overline{m}^k\|_{x^k}^* \le \alpha_kp_{k}(\|\overline{m}^k - \nabla f(x^k)\|_{x^k}^*)^2/2 + 8\eta_k^2/(p_k\alpha_k)$, which together with \cref{P-em-k} and the above inequality implies that for all $k\ge0$,
\begin{align}
\E_{\xi^{k+1}}[P_{k+1}] 
&\overset{\cref{P-em-k}}{\le}(\mu_{k+1}-\mu_k)\phi_{\mathrm{low}} + P_k - \eta_k \|\nabla\phi_{\mu_k}(x^k) + A^T {\lambda}^k\|_{x^k}^* \nonumber\\
&\qquad + \frac{L_{\phi}}{2}\eta_k^2 + \frac{8\eta_k^2}{p_k\alpha_k} +  \frac{L_2^2\eta_k^4p_{k+1}}{(1-\eta_0)^2 \gamma_k^2\alpha_k} + \frac{\sigma^2\gamma_k^2p_{k+1}}{(1-\eta_0)^2(1-\eta_k/\gamma_k)^2}.\label{desc-em-oneiter}
\end{align}
On the other hand, by \cref{asp:unbias-boundvar}, \cref{to-lwbd-phimu}, \cref{P-em-k}, $\overline{m}^0=G(x^0,\xi^0)$, and the fact that $\mu_0\le 1$, one has 
\begin{align*}
&\E_{\xi^0}[P_0] = \phi_{\mu_0}(x^0) + p_0\E_{\xi^0}[(\|\overline{m}^0 - \nabla f(x^0)\|_{x^0}^*)^2] \le f(x^0) + [f(x^0)+B(x^0)]_+ + p_0\sigma^2,\\
&\E_{\{\xi^k\}_{k=0}^K}[P_K] = \phi_{\mu_K}(x^K) + p_K\E_{\{\xi^k\}_{k=0}^K}[(\|\overline{m}^K- \nabla f(x^K)\|_{x^K}^*)^2]\overset{\cref{to-lwbd-phimu}}{\ge} (1+\mu_K)\phi_{\mathrm{low}}.
\end{align*}
By summing \cref{desc-em-oneiter} over $k=0,\ldots,K-1$, and using the above two inequalities, \cref{lwbd-fb}, and the fact that $\{\eta_k\}_{k\ge0}$ is nonincreasing, we obtain that
\begin{align*}
&\sum_{k=0}^{K-1} \E[\|\nabla\phi_{\mu_k}(x^k) + A^T{\lambda}^k\|_{x^k}^*]\le\frac{\Delta(x^0) + p_0\sigma^2}{\eta_{K-1}}\\
&\qquad + \frac{1}{\eta_{K-1}} \sum_{k=0}^{K-1}\left(\frac{L_\phi}{2}\eta_k^2 + \frac{8\eta_k^2}{p_k\alpha_k} +  \frac{L_2^2\eta_k^4p_{k+1}}{(1-\eta_0)^2\gamma_k^2\alpha_k} + \frac{\sigma^2\gamma_k^2p_{k+1}}{(1-\eta_0)^2(1-\eta_k/\gamma_k)^2}\right).
\end{align*}
Hence, the conclusion of this theorem holds as desired.
\end{proof}

\begin{proof}[\textbf{Proof of \cref{lem:tech-ap-use}}]
It follows from the definition of $\{(\eta_k,\gamma_k,\alpha_k)\}_{k\ge0}$ that
\begin{align*}
\alpha_k = \frac{(k+1)^{1/7} - 5s_\eta/7}{(k+1)^{5/7} - 5s_\eta/7} >  \frac{1-5s_\eta/(7(k+1)^{1/7})}{(k+1)^{4/7}} \ge \frac{1- 5s_\eta/7}{(k+1)^{4/7}}\qquad\forall k\ge0,
\end{align*}
where the first inequality is due to $s_\eta\in(0,1)$. We next prove $(1 - \alpha_k) p_{k+1}\le (1-\alpha_k/2)p_k$ for all $k\ge0$. By the definition of $\{(\eta_k,\gamma_k,\alpha_k)\}_{k\ge0}$, one has for all $k\ge0$ that
\begin{align*}
&\frac{1-\alpha_k/2}{1-\alpha_k}= 1 + \frac{(1 - 5s_\eta/(7(k+1)^{1/7}))/2}{(k+1)^{4/7}-1} > 1 + \frac{(1 - 5s_\eta/7)/2}{(k+1)^{4/7}} > 1 + \frac{1}{7(k+1)^{4/7}},
\end{align*}
where the inequalities are due to $s_\eta\in(0,1)$ and $k\ge0$. In addition, recall from \cref{def:ap-em} that ${p_{k+1}}/{p_{k}} = (1+1/(k+1))^{1/7} \le 1 + 1/(7(k+1))$ for all $k\ge0$, where the second inequality is due to $(1+a)^r\le 1+ar$ for all $a,r\in[0,1]$. Combining the above two inequalities with the fact that $k+1\ge(k+1)^{4/7}$ for all $k\ge0$, we obtain that $(1 - \alpha_k) p_{k+1}\le (1-\alpha_k/2)p_k$ holds for all $k\ge0$. Hence, this lemma holds. 
\end{proof}

\begin{proof}[\textbf{Proof of \cref{cor:order-em}}]
Recall from \cref{etak-Bk-em} that $\eta_k/\gamma_k< s_\eta$ for all $k\ge0$. Thus, $\{(\eta_k,\gamma_k)\}_{k\ge0}$ defined in \cref{etak-Bk-em} satisfies the assumption on $\{(\eta_k,\gamma_k)\}_{k\ge0}$ in \cref{thm:em-stat-ave}. Notice from \cref{lem:tech-ap-use} that $\{p_{k}\}_{k\ge0}$ defined in \cref{def:ap-em} and $\{\alpha_k\}_{k\ge0}$ defined in \cref{thm:em-stat-ave} satisfy the assumption on $\{(\alpha_k,p_k)\}_{k\ge0}$ in \cref{thm:em-stat-ave}. Substituting \cref{etak-Bk-em}, \cref{def:ap-em}, and $\alpha_k\ge (1-5s_\eta/7)/(k+1)^{4/7}$ (see \cref{lem:tech-ap-use}) into \cref{ineq:ave-stat-em}, we obtain for all $K\ge3$,
\begin{align}
&\sum_{k=0}^{K-1}\E[\|\nabla \phi_{\mu_k}(x^k) + A^T\lambda^k\|^*_{x^k}]\nonumber\\
&\le \frac{\Delta(x^0) + p_0\sigma^2}{\eta_{K-1}} + \frac{1}{\eta_{K-1}} \sum_{k=0}^{K-1}\left(\frac{L_\phi}{2}\eta_k^2 + \frac{8\eta_k^2}{p_k\alpha_k} +  \frac{2L_2^2\eta_k^4p_k}{(1-\eta_0)^2\gamma_k^2\alpha_k} + \frac{2\sigma^2\gamma_k^2p_k}{(1-\eta_0)^2(1-\eta_k/\gamma_k)^2}\right)\nonumber\\
&{\le}\frac{7K^{5/7}}{5}\bigg(\frac{\Delta(x^0) + \sigma^2}{s_\eta} + \sum_{k=0}^{K-1}\Big(\frac{25s_\eta L_\phi}{98(k+1)^{10/7}} \nonumber\\
&\qquad +\frac{200s_\eta/(7(7-5s_\eta)) + 1250 L_2^2s_\eta^3/(7(7-5s_\eta)^3) + 2\sigma^2/(s_\eta(1-5s_\eta/7)^4)}{k+1}\Big)\bigg)\nonumber\\
&\le \frac{7K^{5/7}}{5}\bigg(\frac{\Delta(x^0) + \sigma^2}{s_\eta} + \frac{40s_\eta L_\phi}{49} + \Big(\frac{200s_\eta}{7(7-5s_\eta)} +\frac{1250 L_2^2s_\eta^3}{7(7-5s_\eta)^3} + \frac{2\sigma^2}{s_\eta(1-5s_\eta/7)^4}\Big)\ln(2K+1)\bigg)\nonumber\\
&\le \frac{7}{5}\bigg(\frac{\Delta(x^0) + \sigma^2}{s_\eta} + \frac{40s_\eta L_\phi}{49} + 2\Big(\frac{200s_\eta}{7(7-5s_\eta)} +\frac{1250 L_2^2s_\eta^3}{7(7-5s_\eta)^3} + \frac{2\sigma^2}{s_\eta(1-5s_\eta/7)^4}\Big)\bigg)K^{5/7}\ln K\nonumber\\
&\overset{\cref{def:kem}}{=}M_{\mathrm{em}}K^{5/7}\ln K/2,\label{ineq:ave-stat-em-proof}
\end{align}
where the first inequality follows from \cref{ineq:ave-stat-em} and the fact that $p_{k+1}\le 2p_k$ for any $k\ge0$, the second inequality is due to \cref{etak-Bk-em}, \cref{def:ap-em}, and $\alpha_k\ge (1-5s_\eta/7)/(k+1)^{4/7}$ for all $k\ge0$, the third inequality follows from $\sum_{k=0}^{K-1}1/(k+1)^{10/7}\le(7/3)2^{3/7}<16/5$ and $\sum_{k=0}^{K-1}1/(k+1)\le\ln(2K+1)$ due to \cref{upbd:series-ka} with $(a,b)=(1,K)$ and $\alpha=10/7,1$, and the last inequality is due to $1\le\ln K$ and $\ln(2K+1)\le 2\ln K$ given that $K\ge3$. Since $\kappa(K)$ is uniformly drawn from $\{\lfloor K/2\rfloor,\ldots,K-1\}$, we have that for all $K\ge3$,
\begin{align}
&\E[\|\nabla \phi_{\mu_{\kappa(K)}}(x^{\kappa(K)}) + A^T\lambda^{\kappa(K)}\|_{x^{\kappa(K)}}^*]=\frac{1}{K-\lfloor K/2\rfloor}\sum_{k=\lfloor K/2\rfloor}^{K-1}\E[\|\nabla \phi_{\mu_k}(x^k) + A^T\lambda^k\|_{x^k}^*]\nonumber\\
&\le\frac{2}{K}\sum_{k=0}^{K-1}\E[\|\nabla \phi_{\mu_k}(x^k) + A^T\lambda^{k}\|_{x^k}^*] \overset{\cref{ineq:ave-stat-em-proof}}{\le} M_{\mathrm{em}}{K^{-2/7}}\ln K.
\label{upbd-rd-para-em}
\end{align}
By \cref{lem:rate-complexity} with $(\beta,u,v)=(2/7,\epsilon/(7M_{\mathrm{em}}(1+\sqrt{\vartheta})),K)$, one can see that 
\begin{align*}
K^{-2/7}\ln K\le \frac{\epsilon}{M_{\mathrm{em}}(1+\sqrt{\vartheta})}\qquad\forall K\ge \Big(\frac{7M_{\mathrm{em}}(1+\sqrt{\vartheta})}{\epsilon}\ln\Big(\frac{7M_{\mathrm{em}}(1+\sqrt{\vartheta})}{\epsilon}\Big)\Big)^{7/2},   
\end{align*}
which together with \cref{upbd-rd-para-em} implies that
\begin{align}
&\E[\|\nabla \phi_{\mu_{\kappa(K)}}(x^{\kappa(K)}) + A^T\lambda^{\kappa(K)}\|_{x^{\kappa(K)}}^*]\le \frac{\epsilon}{1+\sqrt{\vartheta}}\nonumber\\
&\forall K\ge \max\Big\{\Big(\frac{7M_{\mathrm{em}}(1+\sqrt{\vartheta})}{\epsilon}\ln\Big(\frac{7M_{\mathrm{em}}(1+\sqrt{\vartheta})}{\epsilon}\Big)\Big)^{7/2},3\Big\}.    \label{inter-cmplx-em}
\end{align}
On the other hand, when $K\ge2((1+\sqrt{\vartheta})/\epsilon)^{7/2}$, by the definition of $\{\mu_k\}_{k\ge0}$ in \cref{etak-Bk-em} and the fact that $\kappa(K)$ is uniformly selected from $\{\lfloor K/2\rfloor,\ldots,K-1\}$, one has that $\mu_{\kappa(K)}=\mu_{\lfloor K/2\rfloor}=\epsilon/(1+\sqrt{\vartheta})$. Combining this with \cref{inter-cmplx-em}, we obtain that \cref{complexity-em} holds as desired, and the proof of this theorem is complete. 
\end{proof}

\subsection{Proof of the main results in \cref{subsec:sipm-rm}}\label{subsec:proof-rm}

\begin{proof}[\textbf{Proof of \cref{lem:rm-var-err}}]
Fix any $k\ge0$. Recall from \cref{lem:int-frame} that $\|x^{k+1}-x^k\|_{x^k}=\eta_k$. Using this and \cref{sipm-rm-mk-up}, we have that
\begin{align}
&\E_{\xi^{k+1}}[(\|\overline{m}^{k+1} - \nabla f(x^{k+1})\|_{x^{k+1}}^*)^2] \le \frac{1}{(1-\eta_k)^2}\E_{\xi^{k+1}}[(\|\overline{m}^{k+1} - \nabla f(x^{k+1})\|_{x^{k}}^*)^2]\nonumber\\
& \overset{\cref{sipm-rm-mk-up}}{=} \frac{1}{(1-\eta_k)^2}\E_{\xi^{k+1}}[(\|G(x^{k+1},\xi^{k+1}) + (1-\gamma_k) (\overline{m}^k - G(x^k,\xi^{k+1})) - \nabla f(x^{k+1})\|_{x^{k}}^*)^2]\nonumber\\
& = (1-\alpha_k)^2(\|\overline{m}^k-\nabla f(x^k)\|^*_{x^{k}})^2\nonumber\\
&\qquad + \frac{1}{(1-\eta_k)^2} \E_{\xi^{k+1}}[(\|G(x^{k+1},\xi^{k+1}) - \nabla f(x^{k+1}) + (1-\gamma_k) (\nabla f(x^k) - G(x^k,\xi^{k+1}))\|_{x^{k}}^*)^2],\label{ineq:rm-error-upbd-1} 
\end{align}
where the first inequality is due to \cref{lem:tech-barrier}(iii) and $\|x^{k+1}-x^k\|_{x^k}=\eta_k$, and the second equality follows from the first relation in \cref{asp:unbias-boundvar} and the definition of $\alpha_k$. Also, observe that
\begin{align*}
&\E_{\xi^{k+1}}[(\|G(x^{k+1},\xi^{k+1}) - \nabla f(x^{k+1}) + (1-\gamma_k) (\nabla f(x^k) - G(x^k,\xi^{k+1}))\|_{x^k}^*)^2]\\
&\le 3(\|\nabla f(x^{k+1}) - \nabla f(x^k)\|_{x^{k}}^*)^2 + 3\E_{\xi^{k+1}}[(\|G(x^{k+1},\xi^{k+1}) - G(x^k,\xi^{k+1})\|_{x^{k}}^*)^2] \\
&\qquad +  3\gamma_k^2\E_{\xi^{k+1}}[(\|\nabla f(x^k) - G(x^k,\xi^{k+1})\|_{x^{k}}^*)^2]\Big)\\
&\le 3(L_1^2+L^2)\|x^{k+1}-x^k\|_{x^k}^2 + 3\gamma_k^2\E_{\xi^{k+1}}[(\|\nabla f(x^k) - G(x^k,\xi^{k+1})\|_{x^{k}}^*)^2]\\
&= 3(L_1^2+L^2)\eta_k^2 + 3\gamma_k^2\E_{\xi^{k+1}}[(\|\nabla f(x^k) - G(x^k,\xi^{k+1})\|_{x^{k}}^*)^2] \le 3(L_1^2+L^2)\eta_k^2 + 3\sigma^2\gamma_k^2,
\end{align*}
where {the first inequality is due to $\|a+b+c\|^2\le3\|a\|^2+3\|b\|^2+3\|c\|^2$ for all $a,b,c\in\R^n$,} the second inequality is due to \cref{ineq:1st-Lip} and \cref{ineq:Lip-average}, the first equality is due to $\|x^{k+1}-x^k\|_{x^k}=\eta_k$, and the last inequality follows from the second relation in \cref{asp:unbias-boundvar}. Combining \cref{ineq:rm-error-upbd-1} with the above inequality, we obtain that
\[
\E_{\xi^{k+1}}[(\|\overline{m}^{k+1} - \nabla f(x^{k+1})\|_{x^{k+1}}^*)^2] \le (1-\alpha_k)^2(\|\overline{m}^k-\nabla f(x^k)\|^*_{x^{k}})^2 + \frac{3(L_1^2+L^2)\eta_k^2 + 3\sigma^2\gamma_k^2}{(1-\eta_k)^2}.
\]
Notice from $\gamma_k>\eta_k$ and the definition of $\alpha_k$ that $\alpha_k\in(0,1)$. Using this, the above relation, and the fact that $\{\eta_k\}_{k\ge0}$ is nonincreasing, we obtain that this lemma holds as desired.    
\end{proof}


\begin{proof}[\textbf{Proof of \cref{thm:converge-rate-ipmrm}}]
For convenience, we construct potentials as
\begin{align}\label{def:potent-rm}
P_k = \phi_{\mu_k}(x^k) + p_k(\|\overline{m}^k-\nabla f(x^k)\|_{x^k}^*)^2 \qquad \forall k\ge0.
\end{align}
Recall from Algorithm \ref{alg:unf-sipm} that $\{\mu_k\}$ is nonincreasing. Using these, \cref{def:lwbd-phi}, \cref{ineq:var-recur-rm}, and \cref{ineq:general-descent}, we obtain that for all $k\ge0$,
\begin{align*}
&\E_{\xi^{k+1}}[P_{k+1}] \overset{\cref{def:potent-rm}}{=} \E_{\xi^{k+1}}[\phi_{\mu_{k+1}}(x^{k+1}) + p_{k+1}(\|\overline{m}^{k+1} - \nabla f(x^{k+1})\|_{x^{k+1}}^*)^2]\nonumber\\
&= (\mu_{k+1}-\mu_k)(f(x^{k+1})+B(x^{k+1})) + \E_{\xi^{k+1}}[\phi_{\mu_k}(x^{k+1}) + p_{k+1}(\|\overline{m}^{k+1} - \nabla f(x^{k+1})\|_{x^{k+1}}^*)^2]\nonumber\\
&\overset{\cref{def:lwbd-phi}\cref{ineq:var-recur-rm}\cref{ineq:general-descent}}{\le} (\mu_{k+1}-\mu_k)\phi_{\mathrm{low}} +  \phi_{\mu_k}(x^k) - \eta_k \|\nabla\phi_{\mu_k}(x^k) + A^T{\lambda}^k\|_{x^k}^* + 4\eta_k\|\nabla f(x^k) - \overline{m}^k\|_{x^k}^* + \frac{L_{\phi}}{2}\eta_k^2\\
&\qquad +  (1-\alpha_k)p_{k+1}(\|\nabla f(x^k) - \overline{m}^k\|_{x^k}^*)^2 + \frac{3(L_1^2+L^2)\eta_k^2p_{k+1} + 3\sigma^2\gamma_k^2p_{k+1}}{(1-\eta_0)^2}\nonumber\\
&\le (\mu_{k+1}-\mu_k)\phi_{\mathrm{low}} +  \phi_{\mu_k}(x^k) - \eta_k \|\nabla\phi_{\mu_k}(x^k) + A^T{\lambda}^k\|_{x^k}^* + 4\eta_k\|\nabla f(x^k) - \overline{m}^k\|_{x^k}^* + \frac{L_{\phi}}{2}\eta_k^2\\
&\qquad + (1-\alpha_k/2)p_k(\|\nabla f(x^k) - \overline{m}^k\|_{x^k}^*)^2 + \frac{3(L_1^2+L^2)\eta_k^2p_{k+1} + 3\sigma^2\gamma_k^2p_{k+1}}{(1-\eta_0)^2},
\end{align*}
where the second inequality follows from $(1-\alpha_k)p_{k+1}\le (1-\alpha_k/2)p_k$. In addition, notice that $4\eta_k\|\nabla f(x^k) - \overline{m}^k\|_{x^k}^*\le \alpha_kp_k(\|\nabla f(x^k) - \overline{m}^k\|_{x^k}^*)^2/2 + 8\eta_k^2/(p_k\alpha_k)$, which together with \cref{def:potent-rm} and the above inequality implies that for all $k\ge0$, 
\begin{align}
&\E_{\xi^{k+1}}[P_{k+1}]\overset{\cref{def:potent-rm}}{\le} (\mu_{k+1}-\mu_k)\phi_{\mathrm{low}} + P_k - \eta_k \|\nabla\phi_{\mu_k}(x^k) + A^T{\lambda}^k\|_{x^k}^*\nonumber\\
&\qquad\qquad\qquad\qquad + \frac{L_{\phi}}{2}\eta_k^2 + \frac{8\eta_k^2}{p_k\alpha_k}+ \frac{3(L_1^2+L^2)\eta_k^2p_{k+1} + 3\sigma^2\gamma_k^2p_{k+1}}{(1-\eta_0)^2}.\label{desc-rm-oneiter}
\end{align}
On the other hand, by \cref{asp:unbias-boundvar}, \cref{to-lwbd-phimu}, \cref{def:potent-rm}, $\overline{m}^0=G(x^0,\xi^0)$, and $\mu_0\le1$, one has
\begin{align*}
&\E_{\xi^0}[P_0] = \phi_{\mu_0}(x^0) + p_0\E[(\|\overline{m}^0 - \nabla f(x^0)\|_{x^0}^*)^2] \le f(x^0) + [f(x^0) + B(x^0)]_+ + p_0\sigma^2,\\
&\E_{\{\xi^k\}_{k=0}^K}[P_K] = \phi_{\mu_K}(x^K) + p_K\E[(\|\overline{m}^K - \nabla f(x^K)\|_{x^K}^*)^2] \overset{\cref{to-lwbd-phimu}}{\ge} (1+\mu_K)\phi_{\mathrm{low}}.
\end{align*}
By summing \cref{desc-rm-oneiter} over $k=0,\ldots,K-1$, and using the above two inequalities, \cref{lwbd-fb}, and the fact that $\{\eta_k\}_{k\ge0}$ is nonincreasing, we obtain that
\begin{align*}
&\sum_{k=0}^{K-1}\E[\|\nabla \phi_{\mu_k}(x^k) +A^T {\lambda}^k\|^*_{x^k}]
\le \frac{\Delta(x^0) + p_0\sigma^2}{\eta_{K-1}}\\
&\qquad + \frac{1}{\eta_{K-1}}\sum_{k=0}^{K-1}\Big(\frac{L_{\phi}}{2}\eta_k^2 + \frac{8\eta_k^2}{p_k\alpha_k}+ \frac{3(L_1^2+L^2)\eta_k^2p_{k+1} + 3\sigma^2\gamma_k^2p_{k+1}}{(1-\eta_0)^2}\Big).
\end{align*}
Hence, the conclusion of this theorem holds.  
\end{proof}

\begin{proof}[\textbf{Proof of \cref{lem:tech-eta-gma-rm}}]
It follows from the definition of $\{(\eta_k,\gamma_k,\alpha_k)\}_{k\ge0}$ that 
\begin{align*}
\alpha_k= \frac{1-s_\eta/3}{(k+1)^{2/3}-s_\eta/3} > \frac{1-s_\eta/3}{(k+1)^{2/3}} \qquad\forall k\ge0,
\end{align*}
where the inequality is due to $s_\eta\in(0,1)$. We next prove $(1-\alpha_k)p_{k+1}\le(1-\alpha_k/2)p_k$ for all $k\ge0$. By the definition of $\{(\eta_k,\gamma_k,\alpha_k)\}_{k\ge0}$, one has for all $k\ge0$ that 
\[
\frac{1-\alpha_k/2}{1-\alpha_k} = 1 + \frac{(1-s_\eta/3)/2}{(k+1)^{2/3}-1} > 1 + \frac{1}{3(k+1)^{2/3}},
\]
where the first inequality is due to $s_\eta\in(0,1)$ and $k\ge0$. Also, we recall from \cref{def:ap-rm} that $p_{k+1}/p_k=(1+1/(k+1))^{1/3}\le 1+ 1/(3(k+1))$ for all $k\ge0$, where the second inequality is due to $(1+a)^r\le 1+ar$ for all $a,r\in[0,1]$. Combining the above two inequalities with the fact that $k+1\ge(k+1)^{2/3}$, we obtain that $(1-\alpha_k)p_{k+1}\le(1-\alpha_k/2)p_k$ for all $k\ge0$. Hence, this lemma holds as desired.
\end{proof}

\begin{proof}[\textbf{Proof of \cref{cor:order-rm}}]
Recall from \cref{etak-Bk-rm} that $\eta_k< \gamma_k$ for all $k\ge0$. Therefore, $\{(\eta_k,\gamma_k)\}_{k\ge0}$ defined in \cref{etak-Bk-rm} satisfies the assumption on $\{(\eta_k,\gamma_k)\}_{k\ge0}$ in \cref{thm:converge-rate-ipmrm}. Notice from \cref{lem:tech-eta-gma-rm} that $\{\alpha_k\}_{k\ge0}$ defined in \cref{thm:converge-rate-ipmrm} and $\{p_k\}_{k\ge0}$ defined in \cref{def:ap-rm} satisfy the assumption on $\{(\alpha_k,p_k)\}_{k\ge0}$ in \cref{thm:converge-rate-ipmrm}. By substituting \cref{etak-Bk-rm}, \cref{def:ap-rm}, and $\alpha_k\ge(1-s_\eta/3)/(k+1)^{2/3}$ (see \cref{lem:tech-eta-gma-rm}) into \cref{ineq:ave-stat-rm}, one can obtain that for all $K\ge3$,
\begin{align}
&\sum_{k=0}^{K-1}\E[(\|\nabla \phi_{\mu_k}(x^k) + A^T\lambda^k\|^*_{x^k})^2]\nonumber\\
&\le \frac{\Delta(x^0) + p_0\sigma^2}{\eta_{K-1}} + \frac{1}{\eta_{K-1}}\sum_{k=0}^{K-1}\Big(\frac{L_{\phi}}{2}\eta_k^2 + \frac{8\eta_k^2}{p_k\alpha_k}+ \frac{6(L_1^2+L^2)\eta_k^2p_k + 6\sigma^2\gamma_k^2p_k}{(1-\eta_0)^2}\Big)\nonumber\\
&{\le} 3K^{2/3}\bigg(\frac{\Delta(x^0) + \sigma^2}{s_{\eta}}\nonumber\\
&\qquad + \sum_{k=0}^{K-1}\Big(\frac{s_\eta L_\phi}{18(k+1)^{4/3}} + \frac{8s_\eta/(3(3-s_\eta)) + 6(L_1^2+L^2)s_\eta/(3-s_\eta)^2 + 6\sigma^2/(s_\eta(1-s_\eta/3)^2)}{k+1}\Big)\bigg)\nonumber\\
&\le 3K^{2/3}\bigg(\frac{\Delta(x^0) + \sigma^2}{s_{\eta}}  + \frac{2s_\eta L_\phi}{9} + 2\Big(\frac{4s_\eta}{3(3-s_\eta)} + \frac{3(L_1^2+L^2)s_\eta}{(3-s_\eta)^2} + \frac{3\sigma^2}{s_\eta(1-s_\eta/3)^2}\Big)\ln(2K+1) \bigg)\nonumber\\
&\le 3\bigg(\frac{\Delta(x^0) + \sigma^2}{s_{\eta}}  + \frac{2s_\eta L_\phi}{9} + 4\Big(\frac{4s_\eta}{3(3-s_\eta)} + \frac{3(L_1^2+L^2)s_\eta}{(3-s_\eta)^2} + \frac{3\sigma^2}{s_\eta(1-s_\eta/3)^2}\Big) \bigg){K^{2/3}}\ln K\nonumber\\
&\overset{\cref{def:krm}}{=} M_{\rmrm}{K^{2/3}}\ln K/2 ,\label{ineq:ave-stat-rm-proof}
\end{align}
where the first inequality is due to \cref{ineq:ave-stat-rm} and $p_{k+1}\le 2p_k$ for all $k\ge0$, the second inequality follows from \cref{etak-Bk-rm}, \cref{def:ap-rm}, and $\alpha_k\ge(1-s_\eta/3)/(k+1)^{2/3}$ for all $k\ge0$, the third inequality is because $\sum_{k=0}^{K-1}1/(k+1)^{4/3}\le 3(2)^{1/3}<4$ and $\sum_{k=0}^K1/(k+1)\le \ln(2K+1)$ due to \cref{upbd:series-ka} with $(a,b)=(1,K)$ and $\alpha=4/3,1$, and the last inequality follows from $1\le\ln K$, and $\ln(2K+1)\le 2\ln K$ given that $K\ge3$. Since $\kappa(K)$ is uniformly drawn from $\{\lfloor K/2\rfloor,\ldots,K-1\}$, we obtain that for all $K\ge3$,
\begin{align}
&\E[\|\nabla \phi_{\mu_{\kappa(K)}}(x^{\kappa(K)}) + A^T\lambda^{\kappa(K)}\|_{x^{\kappa(K)}}^*]=\frac{1}{K-\lfloor K/2\rfloor}\sum_{k=\lfloor K/2\rfloor}^{K-1}\E[\|\nabla \phi_{\mu_k}(x^k) + A^T\lambda^k\|_{x^k}^*]\nonumber\\
&\le \frac{2}{K}\sum_{k=0}^{K-1}\E[\|\nabla \phi_{\mu_k}(x^k) + A^T\lambda^k\|_{x^k}^*] \overset{\cref{ineq:ave-stat-rm-proof}}{\le} M_{\rmrm}K^{-1/3}\ln K.\label{upbd-rd-para-rm}
\end{align}
By \cref{lem:rate-complexity} with $(\beta,u,v)=(1/3,\epsilon/(6M_{\mathrm{rm}}(1+\sqrt{\vartheta})),K)$, one can see that 
\begin{align*}
K^{-1/3}\ln K\le \frac{\epsilon}{M_{\mathrm{rm}}(1+\sqrt{\vartheta})}\qquad\forall K\ge \Big(\frac{6M_{\mathrm{rm}}(1+\sqrt{\vartheta})}{\epsilon}\ln\Big(\frac{6M_{\mathrm{rm}}(1+\sqrt{\vartheta})}{\epsilon}\Big)\Big)^3,
\end{align*}
which together with \cref{upbd-rd-para-rm} implies that
\begin{align}
&\E[\|\nabla \phi_{\mu_{\kappa(K)}}(x^{\kappa(K)}) + A^T\lambda^{\kappa(K)}\|_{x^{\kappa(K)}}^*]\le \frac{\epsilon}{1+\sqrt{\vartheta}}\nonumber\\
&\forall K\ge \max\Big\{\Big(\frac{6M_{\mathrm{rm}}(1+\sqrt{\vartheta})}{\epsilon}\ln\Big(\frac{6M_{\mathrm{rm}}(1+\sqrt{\vartheta})}{\epsilon}\Big)\Big)^3,3\Big\}.    \label{inter-cmplx-rm}
\end{align}
On the other hand, when $K\ge2((1+\sqrt{\vartheta})/\epsilon)^3$, by the definition of $\{\mu_k\}_{k\ge0}$ in \cref{etak-Bk-rm} and the fact that $\kappa(K)$ is uniformly selected from $\{\lfloor K/2\rfloor,\ldots,K-1\}$, one has that $\mu_{\kappa(K)}=\mu_{\lfloor K/2\rfloor}=\epsilon/(1+\sqrt{\vartheta})$. Combining this with \cref{inter-cmplx-rm}, we obtain that \cref{complexity-rm} holds as desired, and the proof of this theorem is complete.
\end{proof}





\appendix








\bibliographystyle{abbrv}
\bibliography{ref}

@book{NN1994IPM,
  title={Interior-Point Polynomial Algorithms in Convex Programming},
  author={Nesterov, Yurii and Nemirovski, Arkadi},
  year={1994},
  publisher={SIAM},
  address={Philadelphia}
}

@article{he2023newton,
  title={A {Newton-CG} based barrier method for finding a second-order stationary point of nonconvex conic optimization with complexity guarantees},
  author={He, Chuan and Lu, Zhaosong},
  journal={SIAM J. Optimi.},
  volume={33},
  number={2},
  pages={1191--1222},
  year={2023},
  publisher={SIAM}
}

@article{dvurechensky2024hessian,
  title={Hessian barrier algorithms for non-convex conic optimization},
  author={Dvurechensky, Pavel and Staudigl, Mathias},
  journal={Math. Program.},
  pages={1--59},
  year={2024},
  publisher={Springer}
}

@article{he2023newtonal,
  title={A {Newton-CG} based barrier-augmented {L}agrangian method for general nonconvex conic optimization},
  author={He, Chuan and Huang, Heng and Lu, Zhaosong},
  journal={Comput. Optim. Appl.},
  volume={89},
  number={3},
  pages={843--894},
  year={2024}
}

@article{nemirovski2004interior,
  title={Interior point polynomial time methods in convex programming},
  author={Nemirovski, Arkadi},
  journal={Lecture Notes},
  year={2004}
}

@manual{mosek,
   author = "MOSEK ApS",
   title = "The MOSEK optimization toolbox for MATLAB manual. Version 10.1.0.",
   year = 2019,
   url = "http://docs.mosek.com/10.1/toolbox/index.html"
 }

@article{fares2001augmented,
  title={An augmented {L}agrangian method for a class of {LMI}-constrained problems in robust control theory},
  author={Fares, Bassem and Apkarian, Pierre and Noll, Dominikus},
  journal={Int. J. Control},
  volume={74},
  number={4},
  pages={348--360},
  year={2001},
  publisher={Taylor \& Francis}
}

@article{zohrizadeh2020conic,
  title={Conic relaxations of power system optimization: {T}heory and algorithms},
  author={Zohrizadeh, Fariba and Josz, Cedric and Jin,  Ming and Madani, Ramtin and Lavaei, Javad and Sojoudi, Somayeh},
  journal={Eur. J. Oper. Res.},
  volume={287},
  number={2},
  pages={391--409},
  year={2020}
}

@book{wolkowicz2012handbook,
  title={Handbook of Semidefinite Programming: Theory, Algorithms, and Applications},
  author={Wolkowicz, Henry and Saigal, Romesh and Vandenberghe, Lieven},
  year={2012},
  publisher={Springer Science \& Business Media}
}

@inproceedings{bach2004multiple,
  title={Multiple kernel learning, conic duality, and the {SMO} algorithm},
  author={Bach, Francis R and Lanckriet, Gert RG and Jordan, Michael I},
  booktitle={ICML},
  pages={6},
  year={2004}
}

@book{sra2011optimization,
  title={Optimization for Machine Learning},
  author={Sra, Suvrit and Nowozin, Sebastian and Wright, Stephen J},
  year={2011},
  publisher={MIT Press}
}

@article{shivaswamy2006second,
  title={Second order cone programming approaches for handling missing and uncertain data},
  author={Shivaswamy, Pannagadatta K and Bhattacharyya, Chiranjib and Smola, Alexander J},
  journal={J. Mach. Learn. Res.},
  pages={1283--1314},
  year={2006},
  publisher={MIT Press}
}

@inproceedings{bidaurrazaga2021k,
  title={{$K$}-means for evolving data streams},
  author={Bidaurrazaga, Arkaitz and P{\'e}rez, Aritz and Cap{\'o}, Marco},
  booktitle={IEEE ICDM},
  pages={1006--1011},
  year={2021}
}

@article{peng2007approximating,
  title={Approximating $k$-means-type clustering via semidefinite programming},
  author={Peng, Jiming and Wei, Yu},
  journal={SIAM J. Optimi.},
  volume={18},
  number={1},
  pages={186--205},
  year={2007},
  publisher={SIAM}
}

@article{toh1999sdpt3,
  title={{SDPT}3—a {MATLAB} software package for semidefinite programming, version 1.3},
  author={Toh, Kim-Chuan and Todd, Michael J and T{\"u}t{\"u}nc{\"u}, Reha H},
  journal={Optim. Methods Softw.},
  volume={11},
  number={1-4},
  pages={545--581},
  year={1999},
  publisher={Taylor \& Francis}
}

@article{alizadeh2003second,
  title={Second-order cone programming},
  author={Alizadeh, Farid and Goldfarb, Donald},
  journal={Math. Program.},
  volume={95},
  number={1},
  pages={3--51},
  year={2003},
  publisher={Citeseer}
}

@article{zhang2021survey,
  title={A survey on multi-task learning},
  author={Zhang, Yu and Yang, Qiang},
  journal={IEEE Trans. Knowl. Data Eng.},
  volume={34},
  number={12},
  pages={5586--5609},
  year={2021},
  publisher={IEEE}
}

@article{argyriou2008convex,
  title={Convex multi-task feature learning},
  author={Argyriou, Andreas and Evgeniou, Theodoros and Pontil, Massimiliano},
  journal={Mach. Learn.},
  volume={73},
  pages={243--272},
  year={2008},
  publisher={Springer}
}

@article{xu2012optimization,
  title={Optimization under probabilistic envelope constraints},
  author={Xu, Huan and Caramanis, Constantine and Mannor, Shie},
  journal={Oper. Res.},
  volume={60},
  number={3},
  pages={682--699},
  year={2012},
  publisher={INFORMS}
}

@article{curtis2023stochastic,
  title={A stochastic-gradient-based interior-point algorithm for solving smooth bound-constrained optimization problems},
  author={Curtis, Frank E and Kungurtsev, Vyacheslav and Robinson, Daniel P and Wang, Qi},
  journal={SIAM J. Optimi.},
  volume={35},
  number={2},
  pages={1030--1059},
  year={2025},
  publisher={SIAM}
}

@article{curtis2024single,
  title={Single-Loop Deterministic and Stochastic Interior-Point Algorithms for Nonlinearly Constrained Optimization},
  author={Curtis, Frank E and Jiang, Xin and Wang, Qi},
  journal={arXiv:2408.16186},
  year={2024}
}

@article{wachter2006implementation,
  title={On the implementation of an interior-point filter line-search algorithm for large-scale nonlinear programming},
  author={W{\"a}chter, Andreas and Biegler, Lorenz T},
  journal={Math. Program.},
  volume={106},
  pages={25--57},
  year={2006},
  publisher={Springer}
}

@article{byrd1999interior,
  title={An interior point algorithm for large-scale nonlinear programming},
  author={Byrd, Richard H and Hribar, Mary E and Nocedal, Jorge},
  journal={SIAM J. Optimi.},
  volume={9},
  number={4},
  pages={877--900},
  year={1999},
  publisher={SIAM}
}

@book{wright1997primal,
  title={Primal-Dual Interior-Point Methods},
  author={Wright, Stephen J},
  year={1997},
  publisher={SIAM}
}

@article{vanderbei1999interior,
  title={An interior-point algorithm for nonconvex nonlinear programming},
  author={Vanderbei, Robert J and Shanno, David F},
  journal={Comput. Optim. Appl.},
  volume={13},
  pages={231--252},
  year={1999},
  publisher={Springer}
}

@article{potra2000interior,
  title={Interior-point methods},
  author={Potra, Florian A and Wright, Stephen J},
  journal={J. Comput. Appl. Math.},
  volume={124},
  number={1-2},
  pages={281--302},
  year={2000},
  publisher={Elsevier}
}

@article{alizadeh1995interior,
  title={Interior point methods in semidefinite programming with applications to combinatorial optimization},
  author={Alizadeh, Farid},
  journal={SIAM J. Optim.},
  volume={5},
  number={1},
  pages={13--51},
  year={1995},
  publisher={SIAM}
}

@article{kim2007interior,
  title={An interior-point method for large-scale $\ell_1 $-regularized least squares},
  author={Kim, Seung-Jean and Koh, Kwangmoo and Lustig, Michael and Boyd, Stephen and Gorinevsky, Dimitry},
  journal={IEEE J. Sel. Top. Signal Process.},
  volume={1},
  number={4},
  pages={606--617},
  year={2007},
  publisher={IEEE}
}

@article{tseng1992convergence,
  title={On the convergence of the affine-scaling algorithm},
  author={Tseng, Paul and Luo, Zhi-Quan},
  journal={Math. Program.},
  volume={56},
  number={1},
  pages={301--319},
  year={1992},
  publisher={Springer}
}

@article{sturm1999using,
  title={Using {SeDuMi} 1.02, a {MATLAB} toolbox for optimization over symmetric cones},
  author={Sturm, Jos F},
  journal={Optim. Methods Softw.},
  volume={11},
  number={1-4},
  pages={625--653},
  year={1999},
  publisher={Taylor \& Francis}
}

@article{narayanan2016randomized,
  author = {Hariharan Narayanan},
  title = {Randomized interior point methods for sampling and optimization},
  journal = {Ann. Appl. Probab.},
  volume = {26},
  number = {1},
  pages = {597--641},
  year = {2016},
}

@article{badenbroek2022complexity,
  title={Complexity analysis of a sampling-based interior point method for convex optimization},
  author={Badenbroek, Riley and de Klerk, Etienne},
  journal={Math. Oper. Res.},
  volume={47},
  number={1},
  pages={779--811},
  year={2022},
  publisher={INFORMS}
}

@inproceedings{cutkosky2019momentum,
  title={Momentum-based variance reduction in non-convex {SGD}},
  author={Cutkosky, Ashok and Orabona, Francesco},
  booktitle={NIPS},
  volume={32},
  year={2019}
}

@book{lan2020first,
  title={First-Order and Stochastic Optimization Methods for Machine Learning},
  author={Lan, Guanghui},
  year={2020},
  publisher={Springer}
}

@inproceedings{fang2018spider,
  title={Spider: Near-optimal non-convex optimization via stochastic path-integrated differential estimator},
  author={Fang, Cong and Li, Chris Junchi and Lin, Zhouchen and Zhang, Tong},
  booktitle={NIPS},
  volume={31},
  year={2018}
}

@inproceedings{cutkosky2020momentum,
  title={Momentum improves normalized {SGD}},
  author={Cutkosky, Ashok and Mehta, Harsh},
  booktitle={ICML},
  pages={2260--2268},
  year={2020}
}

@inproceedings{wang2019spiderboost,
  title={Spiderboost and momentum: Faster variance reduction algorithms},
  author={Wang, Zhe and Ji, Kaiyi and Zhou, Yi and Liang, Yingbin and Tarokh, Vahid},
  booktitle={NIPS},
  volume={32},
  year={2019}
}

@article{xu2023momentum,
  title={Momentum-based variance-reduced proximal stochastic gradient method for composite nonconvex stochastic optimization},
  author={Xu, Yangyang and Xu, Yibo},
  journal={J. Optim. Theory Appl.},
  volume={196},
  number={1},
  pages={266--297},
  year={2023},
  publisher={Springer}
}

@article{tran2022hybrid,
  title={A hybrid stochastic optimization framework for composite nonconvex optimization},
  author={Tran-Dinh, Quoc and Pham, Nhan H and Phan, Dzung T and Nguyen, Lam M},
  journal={Math. Program.},
  volume={191},
  number={2},
  pages={1005--1071},
  year={2022},
  publisher={Springer}
}

@article{ghadimi2013stochastic,
  title={Stochastic first- and zeroth-order methods for nonconvex stochastic programming},
  author={Ghadimi, Saeed and Lan, Guanghui},
  journal={SIAM J. Optimi.},
  volume={23},
  number={4},
  pages={2341--2368},
  year={2013},
  publisher={SIAM}
}

@article{berahas2023stochastic,
  title={A stochastic sequential quadratic optimization algorithm for nonlinear-equality-constrained optimization with rank-deficient {J}acobians},
  author={Berahas, Albert S and Curtis, Frank E and O’Neill, Michael J and Robinson, Daniel P},
  journal={Math. Oper. Res.},
  year={2023},
  publisher={INFORMS}
}

@article{berahas2021sequential,
  title={Sequential quadratic optimization for nonlinear equality constrained stochastic optimization},
  author={Berahas, Albert S and Curtis, Frank E and Robinson, Daniel and Zhou, Baoyu},
  journal={SIAM J. Optimi.},
  volume={31},
  number={2},
  pages={1352--1379},
  year={2021},
  publisher={SIAM}
}

@article{berahas2023accelerating,
  title={Accelerating stochastic sequential quadratic programming for equality constrained optimization using predictive variance reduction},
  author={Berahas, Albert S and Shi, Jiahao and Yi, Zihong and Zhou, Baoyu},
  journal={Comput. Optim. Appl.},
  volume={86},
  number={1},
  pages={79--116},
  year={2023},
  publisher={Springer}
}

@article{curtis2021inexact,
  title={Inexact sequential quadratic optimization for minimizing a stochastic objective function subject to deterministic nonlinear equality constraints},
  author={Curtis, Frank E and Robinson, Daniel P and Zhou, Baoyu},
  journal={arXiv:2107.03512},
  year={2021}
}

@article{fang2024fully,
  title={Fully stochastic trust-region sequential quadratic programming for equality-constrained optimization problems},
  author={Fang, Yuchen and Na, Sen and Mahoney, Michael W and Kolar, Mladen},
  journal={SIAM J. Optimi.},
  volume={34},
  number={2},
  pages={2007--2037},
  year={2024},
  publisher={SIAM}
}

@article{na2023adaptive,
  title={An adaptive stochastic sequential quadratic programming with differentiable exact augmented {L}agrangians},
  author={Na, Sen and Anitescu, Mihai and Kolar, Mladen},
  journal={Math. Program.},
  volume={199},
  number={1},
  pages={721--791},
  year={2023},
  publisher={Springer}
}

@article{na2022asymptotic,
  title={Asymptotic convergence rate and statistical inference for stochastic sequential quadratic programming},
  author={Na, Sen and Mahoney, Michael W},
  journal={arXiv: 2205.13687},
  year={2022},
  publisher={arXiv: 2205.13687}
}

@inproceedings{alacaoglu2024complexity,
  title={Complexity of single loop algorithms for nonlinear programming with stochastic objective and constraints},
  author={Alacaoglu, Ahmet and Wright, Stephen J},
  booktitle={AISTATS},
  pages={4627--4635},
  year={2024}
}

@article{li2024stochastic,
  title={Stochastic inexact augmented {L}agrangian method for nonconvex expectation constrained optimization},
  author={Li, Zichong and Chen, Pin-Yu and Liu, Sijia and Lu, Songtao and Xu, Yangyang},
  journal={Comput. Optim. Appl.},
  volume={87},
  number={1},
  pages={117--147},
  year={2024},
  publisher={Springer}
}

@article{lu2024variance,
  title={Variance-reduced first-order methods for deterministically constrained stochastic nonconvex optimization with strong convergence guarantees},
  author={Lu, Zhaosong and Mei, Sanyou and Xiao, Yifeng},
  journal={arXiv preprint arXiv:2409.09906},
  year={2024}
}

@inproceedings{carmon2017convex,
  title={“{C}onvex until proven guilty”: {D}imension-free acceleration of gradient descent on non-convex functions},
  author={Carmon, Yair and Duchi, John C and Hinder, Oliver and Sidford, Aaron},
  booktitle={ICML},
  pages={654--663},
  year={2017}
}

@inproceedings{zhang2012convex,
  title={A convex formulation for learning task relationships in multi-task learning},
  author={Zhang, Yu and Yeung, Dit-Yan},
  booktitle={UAI},
  year={2010}
}

@inproceedings{lu2014generalized,
  title={Generalized nonconvex nonsmooth low-rank minimization},
  author={Lu, Canyi and Tang, Jinhui and Yan, Shuicheng and Lin, Zhouchen},
  booktitle={IEEE CVPR},
  pages={4130--4137},
  year={2014}
}

@article{forsgren2002interior,
  title={Interior methods for nonlinear optimization},
  author={Forsgren, Anders and Gill, Philip E and Wright, Margaret H},
  journal={SIAM Review},
  volume={44},
  number={4},
  pages={525--597},
  year={2002},
  publisher={SIAM}
}

@article{sun2021convex,
  title={Convex clustering: Model, theoretical guarantee and efficient algorithm},
  author={Sun, Defeng and Toh, Kim-Chuan and Yuan, Yancheng},
  journal={J. Mach. Learn. Res.},
  volume={22},
  number={9},
  pages={1--32},
  year={2021}
}

@inproceedings{li2021page,
  title={{PAGE}: A simple and optimal probabilistic gradient estimator for nonconvex optimization},
  author={Li, Zhize and Bao, Hongyan and Zhang, Xiangliang and Richt{\'a}rik, Peter},
  booktitle={ICML},
  pages={6286--6295},
  year={2021}
}

\end{document}